\documentclass[reqno]{amsart}
\usepackage{amssymb,amsthm,amsfonts,amstext}
\usepackage{amsmath}
\usepackage{newcent}       
\usepackage{url}
\usepackage{bbm}

\makeatletter \@addtoreset{equation}{section} \makeatother

\renewcommand\thefigure{\thesection.\@arabic\c@figure}
\renewcommand\thetable{\thesection.\@arabic\c@table}
\newtheorem{theorem}{Theorem}[section]
\newtheorem{lemma}[theorem]{Lemma}
\newtheorem{proposition}[theorem]{Proposition}
\newtheorem{corollary}[theorem]{Corollary}

\newtheorem{remark}[theorem]{Remark}

\newtheorem{asser}[theorem]{Assertion}

\newcommand{\mc}[1]{{\mathcal #1}}
\newcommand{\mf}[1]{{\mathfrak #1}}
\newcommand{\mb}[1]{{\mathbf #1}}
\newcommand{\bb}[1]{{\mathbb #1}}

\usepackage{color}
\definecolor{bblue}{rgb}{.2,0.2,.8}

\newcommand{\<}{\langle}
\renewcommand{\>}{\rangle}

\title[The stochastic heat equation as the limit of stirring + voter
dynamics]{The
  stochastic heat equation as the limit of a stirring 
  dynamics perturbed by a voter model}

\author{Milton Jara, Claudio Landim}

\address{IMPA, Estrada Dona Castorina 110, CEP 22460 Rio de Janeiro, Brasil\\
  e-mail: \texttt{mjara@impa.br}}

\address{IMPA, Estrada Dona Castorina 110, CEP 22460 Rio de Janeiro,
  Brasil
  and CNRS UMR 6085, Université de Rouen, France. \\
  e-mail: \texttt{landim@impa.br} }

\begin{document}

\maketitle

\begin{abstract}
We prove that in dimension $d\le 3$ a modified density field of a stirring
dynamics perturbed by a voter model converges to the stochastic heat
equation
\end{abstract}

\section{Introduction}

While the equilibrium fluctuations of the density field are well
understood since Brox and Rost \cite{BR84} and Chang \cite{c94}
(cf. references and comments in Chapter 11 of \cite{kl}),
nonequilibrium fluctuations are considered to be one of the main open
problems in the theory of hydrodynamic limit of interacting particle
systems.

For almost three decades, no progress has been made in this subject.
The few known results were restricted to one dimension and their
proofs relied either on special features of the dynamics, such as
duality or integrability of certain quantities, or required strong
estimates, as a logarithmic Sobolev inequality in Chang and Yau
\cite{cy92}.  

With the recent developments in the theory of non-linear stochastic
partial differential equations, this problem became even more
interesting. Indeed, to define solutions of these non-linear
equations, it has been proposed to smooth the noise by convolving it
with a smooth kernel and then to show that, after a renormalization,
the limit exists and does not depend on the kernel of the convolution
(cf. \cite{h14, gip15, cw17} and references therein).

Since interacting particle systems possess an in-built noise, it is
natural to expect that the density fields converge to the normalized
solutions of the SPDEs derived in the theories mentioned above. This
question has attracted much attention recently and many problems
remain unsolved \cite{bprs93, mw, hm18, sw18, zz18, mp19}.

In this article, we pursue in this direction by considering the
fluctuations of a gradient exclusion dynamics perturbed by a voter
model. One of the novelties lies in the definition of the density
field, which is not normalized by the square root of the degrees of
freedom, and on the non-conservative noise which appears in the
limiting equation. Indeed, even if the stochastic PDE which describes
the asymptotic behavior of the density fluctuation is linear, the
noise is non-conservative, in contrast with most previous results
\cite{kl}.

In both models, the exclusion process and the voter dynamics, on the
diffusive time-scale, the density of particles evolve according the
solution of a linear parabolic PDE.  The hydrodynamic behavior of the
voter model has been derived by Presutti and Spohn in \cite{ps83}, and
we refer to \cite{kl} for references on the corresponding result for
exclusion dynamics.

We consider here the exclusion process on the diffusive time-scale and
the voter model evolving in a slower time-scale.  This dynamics has
two absorbing states: the empty configuration and the full
one. Nevertheless, as the voter part evolves in a slower scale, the
global evolution can be understood as a small perturbation of the
exclusion process, and, starting from a state close to an equilibrium
state of the exclusion dynamics, the homogeneous Bernoulli product
measures, one expects that at a later time the state of process
remains close to the equilibrium state of the exclusion dynamics.
One of the main results of this article provides a quantitive estimate
for this closeness.

The main obstacle in the proof of the fluctuations lies in the
replacement of a space-time average of cylinder functions by a
space-time average of the density of particles. This is the so-called
Boltzmann-Gibbs principle.

In equilibrium, this replacement is derived using a classic bound on
the variance of an additive functional of a Markov process
\cite[Proposition A1.6.1]{kl}. In non-equilibrium this tool is not
available and one has to rely on entropy bounds of the state of the
process with respect to a reference measure. In our context, as
mentioned above, the Bernoulli product measures.

To obtain such bounds we rely on the approach introduced recently by
Jara and Menezes \cite{jm1, jm2}, which improved the estimate on the
entropy production obtained by Yau \cite{y91} in the context of
interacting particles systems.  These bounds are keen enough to permit
the derivation of Boltzmann-Gibbs principle (and the tightness of the
density fluctuation field) in dimension $d\le 3$.

\section{Notation and results}

Denote by $\bb T^d_n = (\bb Z / n \bb Z)^d$,
$n \in \bb N =\{1, 2, \dots \}$, the $d$-dimensional discrete torus
with $n^d$ points. We consider a particle system which describes
voters with a binary opinion, $0$ or $1$, evolving on $\bb T^d_n$.

Let $\color{bblue} \Omega_n = \{0,1\}^{\bb T^d_n}$ be the state
space. Elements of $\Omega_n$ are represented by the Greeck letters
$\eta = (\eta_x : x\in \bb T^d_n)$, $\xi$. Hence, $\eta_x =1$ if the
voter at $x$ for the configuration $\eta$ has the opinion $1$.

Let $L_n^V$ be the generator of the voter model in $\Omega_n$:
\begin{equation*}
(L_n^V f)(\eta)  \;=\; \sum_{x\in \bb T^d_n}  \sum_{y: \Vert y-x \Vert=1}
(\eta_y-\eta_x)^2 \,\big[\,  f(\sigma^x \eta) \,-\, f(\eta)\, \big]
\end{equation*}
for all $f: \Omega_n \to \bb R$.  In this formula, the second sum is
carried over all neighbours $y$ of $x$: $y\in \bb T^d_n$,
$\Vert x - y\Vert = 1$, and $\Vert\,\cdot\,\Vert$ stands for the
$\ell^1$ norm:
$\color{bblue} \Vert (z_1, \dots, z_d)\, \Vert = \sum_{1\le j\le d}
|\, z_j\,|$. Moreover, $\sigma^x\eta$ represents the configuration
obtained from $\eta$ by flipping the value of $\eta_x$:
\begin{equation*}
(\sigma^{x}\eta)_{z} \;=\;
\begin{cases}
\eta_{z} & z\not =x\;,\\
1\,-\, \eta_{x} & z=x\;.
\end{cases}
\end{equation*}

Denote by $\color{bblue} \{e_1,\dots,e_d\}$ the canonical basis of
$\bb R^d$.  Let $c_j: \{0,1\}^{\bb Z^d} \to \bb R$, $1\le j\le d$, be
strictly positive, cylinder functions [functions which depend only on a
finite number of variables $\eta_z$]:
\begin{equation}
\label{14}
c_j(\eta) \;\ge\; \mf c_0 \;>\; 0
\end{equation}
for all $\eta\in \{0,1\}^{\bb Z^d}$, $1\le j\le d$.  Assume that $c_j$
does not depend on the variables $\eta_0$ and $\eta_{e_j}$ and that
the following gradient conditions are in force. For each $j$, there
exist cylinder functions $h_{j,k}$, $1\le k\le d$, such that
\begin{equation}
\label{11b}
c_j (\eta) \, [\, \eta_{e_j} \,-\, \eta_{0}\,] \,=\, 
\sum_{k=1}^{d} \big\{\, (\tau_{e_k} \, h_{j,k})(\eta) \;-\;
h_{j,k}(\eta)\, \big\}\;.
\end{equation}
In this formula, $\{\tau_z : z\in \bb Z^d\}$ represents the group of
translations acting on the configurations:
\begin{equation}
\label{12}
(\tau_x \eta)_z \;=\;\eta_{x+z}\;, \quad x\,,\, z\,\in\, \bb Z^d\;,
\;\; \eta\,\in\, \{0,1\}^{\bb Z^d} \;.
\end{equation}

Denote by $L_n^S$ the generator of the speed-change, symmetric
exclusion process given by
\begin{equation}
\label{10}
(L_n^Sf)\, (\eta)\;=\; \sum_{x\in\bb T_n^d} \sum_{j=1}^d 
c_j (\tau_x \eta) \, \{f(\sigma^{x,x+e_j}\eta)-f(\eta)\}\;,
\end{equation}
In this formula, $\sigma^{ x,y}\eta$ represents the configuration of
particles obtained from $\eta$ by exchanging the values of $\eta_x$
and $\eta_y$:
\begin{equation*}
(\sigma^{x,y}\eta)_{z} \;=\;
\begin{cases}
\eta_{z} & z\not =x \,,\, y\;,\\
\eta_{x} & z=y\;, \\
\eta_{y} & z=x\;,
\end{cases}
\end{equation*}
and $\{\tau_x : x\in \bb T^d_n\}$ the translations acting on
$\Omega_n$. The summation in \eqref{10} has now to be understood
modulo $n$. We used the same notation for translations acting on
$\Omega_n$ and on $\{0,1\}^{\bb Z^d}$, but the context will clarify to
which one we are referring to.

In the special case where $c_j(\eta) = 1$ for all $1\le j\le d$, we
recover the symmetric simple exclusion proces on $\bb T^d_n$, whose
generator, denoted by $L_n^E$ to stress this particular case, can be
written as
\begin{equation*}
(L_n^E f)(\eta) \;=\; \sum_{x\in \bb T^d_n}  \sum_{y: \Vert y-x \Vert=1}
\eta_x\, (1-\eta_y)\, \big[\, f(\eta^{x,y}) \,-\, f(\eta)\, \big]\;.
\end{equation*}

Denote by $\color{bblue} \nu^n_\rho$, $0\le \rho\le 1$, the Bernoulli
product measure on $\Omega_n$ with density $\rho$. This is the product
measure whose marginals are Bernoulli distributions with parameter
$\rho$. A straightforward computation shows that these measures
satisfy the detailed balance conditions for the speed-change exclusion
process because the cylinder functions $c_j$ are assumed not to depend
on $\eta_0$, $\eta_{e_j}$. In particular, they are stationary for this
dynamics.

Fix a sequence of positive numbers $\{a_n; n \in \bb N\}$ such that
$\color{bblue} \lim_{n\to\infty} a_n = \infty$,
$\color{bblue}
\lim_{n\to\infty} a_n/n^2 = 0$.  Let
$\color{bblue} (\eta^n(t); t\geq 0)$ be the $\Omega_n$-valued,
continuous-time Markov chain whose generator, denoted by $L_n$, is
given by
\begin{equation*}
L_n \;=\; n^2\, L_n^S \;+\; a_n\, L_n^V\;.
\end{equation*}
We call this process the {\em voter model with stirring}
\cite{dfl86, dn94}. 

Denote by $\color{bblue} D([0,T], \Omega_n)$, $T>0$, the set of
right-continuous trajectories $\mf e: [0,T] \to \Omega_n$ with
left-limits, endowed with the Skorohod topology.  For a probability
measure $\mu_n$ on $\Omega_n$, denote by $\color{bblue} \bb P_{\mu_n}$
the measure on $D([0,T], \Omega_n)$ induced by the Markov chain
$\eta^n(t)$ and the initial distribution $\mu_n$.

It can be verified that the measures $\nu^n_\rho$ are not invariant
with respect to $L_n$. Actually, the only extremal measures are the
singletons supported on the empty and the full configurations.
However, as the Bernoulli product measures are stationary for the
exclusion dynamics and since the exclusion generator is accelerated by
$n^2$, while the voter one is accelerated by $a_n$, and
$a_n/n^2\to 0$, we expect the stationary state of the voter model with
stirring to be close to the Bernoulli measures.

Our aim is to study the density fluctuations of this model, when the
process starts from a measure close to a product of Bernoulli
$\nu^n_\rho$.

\subsection{The density fluctuation field}

Let $\bb T^d$ be the continuous torus of dimension $d$. Denote by
$\color{bblue} C(\bb T^d)$ the space of continous, real-valued
functions on $\bb T^d$ and by $\color{bblue} C^k(\bb T^d)$,
$1\le k\le \infty$, the space of real-valued functions on $\bb T^d$
whose $k$-th derivatives exist and are continuous. Elements of $C(\bb
T^d)$ are represented by the letters $F$, $G$. 

Denote by $\color{bblue} L^2(\bb T^d)$ the space of complex-valued,
square-integrable, measurable functions on $\bb T^d$ endowed with the
usual scalar product, represented by $\<\,\cdot\,,\, \cdot\,\>$.  Let
$\color{bblue} \mc H_r$, $r>0$, be the Hilbert space generated by
the functions in $C^\infty(\bb T^d)$ with the scalar product
$\<\,\cdot\,,\, \cdot\,\>_r$ defined by
\begin{equation*}
\< F\,,\, G\>_r \;=\; \sum_{m\in \bb Z^d}   \, \gamma_m^{r} \,
\mb F_m \, \mb G_m \;,
\end{equation*}
where $\color{bblue} \gamma_m \;:=\; 1\,+\, \Vert m\Vert^2$ and
$\mb F_m = \< F\,,\, \phi_m\>$,
$\color{bblue} \phi_m(x) = \exp\{ 2\pi i x\cdot m\}$. The sum is
finite because $F$, $G$ belong to $C^\infty(\bb T^d)$. Clearly
$\mc H_r \subset \mc H_s$ for $r\ge s$.

Denote by $\color{bblue} \mc H_{-r}$, $r>0$, the dual space of
$\mc H_r$. Elements of $\mc H_{-r}$ are represented by the letters
$X$, $Y$. For $X$ in $\mc H_{-r}$ and $F$ in $\mc H_r$, $X(F)$ can be
represented as
\begin{equation*}
X(F) \;=\; \sum_{m\in \bb Z^d} X(\phi_m) \, \mb F_m \;.
\end{equation*}

Denote by $X^n_t$ the random element of $\mc H_{-1}$ defined by
\begin{equation*}
X_t^n(F) \;=\; \frac{1}{\sqrt{n^d a_n}}
\sum_{x\in \bb T^d_n} F(x/n)\, [\, \eta_x^n(t) \,-\, \rho\, ]\;, \quad
F\,\in\, C^\infty( \bb T^d)\;.
\end{equation*}
This formula defines a $\mc H_{-r}$-valued process
$\{X_t^n; t \geq 0\}$. We call this process the {\em density
  fluctuation field}.

For a cylinder function $f: \{0,1\}^{\bb Z^d} \to \bb R$, denote by
$\tilde f\colon [0,1] \to \bb R$ the function defined by 
\begin{equation}
\label{40}
\tilde f (\rho) \;=\; E_{\nu_\rho} \big[\, f(\eta)\,\big]\;.
\end{equation}
In this formula, $\nu_\rho$ represents the Bernoulli product measure
on $\{0,1\}^{\bb Z^d}$ with density $\rho$.  Note that $\tilde f$ is a
polynomial. Its first derivative is represented by $\tilde f'$.

Recall the definition of the cylinder functions $h_{j,k}$ introduced
in \eqref{11b}. Denote by $\mc A$ the second-order linear elliptic
operator defined by
\begin{equation}
\label{41}
\mc A \, F \;=\; \sum_{j,k=1}^d  \tilde h'_{j,k} (\rho)\,
\partial^2_{x_j, x_k} F \;,
\end{equation}
for functions $F$ in $C^2(\bb T^d)$. Let
$\color{bblue} (P_t : t\ge 0)$ be the semigroup associated to the
operator $\mc A$.

The functions $\phi_m$ are eigenvectors of the operator $\mc A$,
\begin{equation}
\label{46}
\mc A \, \phi_m \;=\; -\, \lambda(m)\, \phi_m \;, \quad\text{where}\;\;
\lambda(m)\;=\; 4\,\pi^2 \, m^\dagger\, \bb H\,  m\;.
\end{equation}
In this formula, $\bb H = (\bb H_{i,j})_{1\le i,j\le d}$ stands for
the symmetric matrix whose entries are given by
$\bb H_{j,k} = (1/2) \{\, \tilde h'_{j,k} (\rho) \,+\, \tilde h'_{k,j}
(\rho)\,\}$ and $m^\dagger$ for the transpose of $m$.

Denote by $D([0,T], \mc H_{-r})$ the space of $\mc H_{-r}$-valued,
right-continuous functions with left-limits, endowed with the Skorohod
topology, and by $C([0,T], \mc H_{-r})$ the space of continuous
functions endowed with the uniform topology.

\begin{theorem}
\label{t1}
Suppose that $d=1$ or $2$, and fix $0<\rho<1$, $T>0$ and
$r>(3d+5)/2$. Assume that $(a_n:n\ge 1)$ is a sequence such that
$a_n\to\infty$, $a_n \le \sqrt{\log n}$. Let $\mu_n$ be a sequence of
probability measures on $\Omega_n$ such that
$\lim_{n\to \infty} a^{-1}_n H_n(\mu_n \,|\, \nu^n_\rho) =0$.  Then,
the sequence of probability measures
$\bb Q_n := \bb P_{\mu_n} \circ (X^n)^{-1}$ on $D([0,T], \mc H_{-r})$
converges weakly to the measure induced by the solution of the
equation
\begin{equation}
\label{58}
\left\{
\begin{aligned}
& \partial_t X_t \;=\; \mc A \, X_t   \;+\;  \sqrt{4\,d\, \chi(\rho)}
\, \xi_t \;, \\
& X_0\;=\; 0\;.
\end{aligned}
\right.
\end{equation}
In this formula, $\chi(\rho) = \rho(1-\rho)$ is the static
compressibility of the stirring dynamics and $\xi$ is a standard
space-time white noise.
\end{theorem}

In dimension $3$, we are not able to prove the convergence of the
process $X^n_t$ but only of its time integral. Let $\bb X^n_t =
\int_0^t X^n_s\, ds$.

\begin{theorem}
\label{t1b}
Suppose that $d=3$, and fix $0<\rho<1$, $T>0$ and $r>8$. Assume that
$(a_n:n\ge 1)$ is a sequence such that $a_n\to\infty$,
$a_n \le \sqrt{\log n}$. Let $\mu_n$ be a sequence of probability
measures on $\Omega_n$ such that
$\lim_{n\to \infty} a^{-1}_n H_n(\mu_n \,|\, \nu^n_\rho) =0$. Then,
under the measure $\bb P_{\mu_n}$, the finite-dimensional
distributions of $X^n$ converge to the ones of the solution of
\eqref{58}. Moreover, the sequence of probability measures
$\bb Q_n := \bb P_{\mu_n} \circ (\bb X^n)^{-1}$ on
$C([0,T], \mc H_{-r})$ converges weakly to the measure induced by the
time-integral of the solution of the equation \eqref{58}.
\end{theorem}

\begin{remark}
\label{rm3}
The condition $a_n \le \sqrt{\log n}$ is not optimal. We just need
that $e^{C_0\, a_n} / \sqrt{n} \to 0$ for all $C_0>0$. We also do not
claim that the choice of $r$ is optimal.
\end{remark}

\begin{remark}
\label{rm3d}
Note that the fluctuation at time $0$, $X^n_0$, vanishes in the
limit. The process $X^n_t$ is built exclusively by the noise.
\end{remark}

\begin{remark}
\label{rm3c}
The result is restricted to dimensions $d\le 3$ for the following
reason. As long known, the crux of the proof of the convergence of the
density fluctuation fields lies in the so called Boltzmann-Gibbs
principle, which permits the replacement of average of cylinder
functions by their projections on the density field. The proof of this
result relies on a bound on the entropy production, presented in the
next subsection, which holds in all dimensions. This bound, however,
is not strong enough in dimension $d\ge 4$ to yield the Boltzmann-Gibbs
principle. 
\end{remark}

\begin{remark}
\label{rm1}
Usually the gradient condition requires that the jump rates $c_j$,
$1\le j\le d$, fulfill the following assumption. For each $j$, there
exist cylinder functions $g_{j,p}$ and finitely-supported signed
measures $m_{j,p}$, $1\le p\le n_j$, such that
\begin{equation}
\label{11}
c_j (\eta) \, [\, \eta_{e_j} \,-\, \eta_{0}\,] \,=\, 
\sum_{p=1}^{n_j} \sum_{y\in\bb Z^d}  m_{j,p}(y) \, (\tau_y \, g_{j,p}
)(\eta)\;, \quad \sum_{y\in\bb Z^d}  m_{j,p}(y) \;=\; 0
\end{equation}
for all $1\le p\le n_j$. However, if conditions \eqref{11} are in
force, then there exist cylinder functions $h_{j,k}$ for which
\eqref{11b} hold.
\end{remark}

\begin{proof}
Fix $1\le j\le d$ and consider the formula for
$c_j (\eta) \, [\, \eta_{0} \,-\, \eta_{e_j}\,]$. We omit $j$ from the
notation from now on. As $\sum_{y\in\bb Z^d} m_{p}(y) \,=\, 0$ for all
$p$, we can write this sum as
\begin{equation*}
\sum_{p=1}^{n} \sum_{y\in\bb Z^d}  m_{p}(y) \, \big\{\, (\tau_y \, g_{p}
)(\eta) \,-\, g_{p} (\eta) \, \big\}\;.
\end{equation*}
Fix $y$ such that $m_p(y) \not = 0$. Consider a path $0 = z_0, z_1,
\dots , z_{\Vert y \Vert} = y$ such that $\Vert z_{i+1} - z_i \Vert =
1$ for $0\le i < \Vert y\Vert$. With this notation, 
\begin{equation*}
\tau_y \, g_{p} \,-\, g_{p} \;=\; 
\sum_{i= 0}^{\Vert y \Vert - 1} [\, \tau_{z_{i+1}} \, g_{p}
\,-\, \tau_{z_{i}} \, g_{p}\,]\;.
\end{equation*}
Since $\Vert z_{i+1} - z_i \Vert = 1$, there exists $1\le k\le d$ such
that $z_{i+1} - z_i = \pm e_k$.

If $z_{i+1} - z_i = e_k$, let $g_{p,i} := \tau_{z_{i}} \, g_{p}$ so
that
$\tau_{z_{i+1}} \, g_{p} \,-\, \tau_{z_{i}} \, g_{p} = \tau_{e_k}
g_{p,i} - g_{p,i}$. In contrast, if $z_{i+1} - z_i \,=\, -\, e_k$, let
$g_{p,i} \,:=\, - \, \tau_{z_{i+1}} \, g_{p}$ so that
$\tau_{z_{i+1}} \, g_{p} \,-\, \tau_{z_{i}} \, g_{p} = \tau_{e_k}
g_{p,i} - g_{p,i}$. With this notation,
\begin{equation*}
\tau_y \, g_{p} \,-\, g_{p} \;=\; 
\sum_{i= 0}^{\Vert y \Vert - 1} [\, \tau_{e_{k(p,i)}}  g_{p,i}
\,-\, g_{p,i}\,]\;.
\end{equation*}
Note that $g_{p,i}$ and $k(p,i)$ depend on $y$ but this fact has been
omitted from the notation.

To complete the proof of the remark, it remains to fix
$1\le \ell\le d$ and define $h_\ell$ as
\begin{equation*}
h_\ell \;=\; \sum_{p=1}^{n} \sum_{y\in\bb Z^d}  m_{p}(y) \sum_{i}
g_{p,i}\;, 
\end{equation*}
where the sum over $i$ is carried over all indices $i$ such that
$k(p,i) = \ell$.
\end{proof}

The main tool in the proof of Theorem \ref{t1} and \ref{t1b} consists
in an {\em a priori} bound on the entropy production of the process
$\eta^n(t)$.

\subsection{The entropy estimate}

Denote by $H_n(\mu \,|\, \nu)$ the relative entropy of the probability
measure $\mu$ with respect to $\nu$:
\begin{equation*}
H_n (\mu \,|\, \nu) \;=\; \sup_f \Big\{ \int_{\Omega_n} f \, d\mu \,-\,
\log \int_{\Omega_n} e^f \, d\nu\, \Big\}\;,
\end{equation*}
where the supremum is carried over all functions $f:\Omega_n\to \bb
R$.

It is known \cite[Theorem A1.8.3]{kl} that
\begin{equation}
\label{02}
H_n (\mu \,|\, \nu) \;=\; H_n(f) \;:=\; \int f \, \log f \; d \nu
\end{equation}
if $\mu$ is absolutely continuous with respect to $\nu$ and $f$
represents the Radon-Nikodym derivative $d\mu/d\nu$. Otherwise,
$H_n (\mu \,|\, \nu) = \infty$.

Let $\color{bblue} (S^n(t) : t\ge 0)$ be the semigroup of the voter
model with stirring. Thus, $\mu S^n(t)$ represents the distribution at
time $t$ of the process $\eta^n(\cdot)$ starting from the probability
measure $\mu$.

It is well known that, for any initial distribution $\mu$, the
relative entropy of $\mu S^n(t)$ with respect to a stationary measure
decreases in time \cite{kl}.

As stressed above, the Bernoulli product measures are not stationary
for the voter model with stirring.  Nevertheless, Theorem \ref{t2}
states that the relative entropy of $\mu S^n(t)$ with respect to a
Bernoulli measure $\nu^n_\rho$ does not grow too fast. More precisely,
for a sequence of probability measures $(\mu_n : n\ge 1)$ on
$\Omega_n$, denote by $H_n(t)$ the relative entropy of $\mu_n S^n(t)$
with respect to a Bernoulli measure $\nu^n_\rho$:
\begin{equation*}
H_n(t) \;=\; H_n \big(\, \mu_n S^n(t) \,|\, \nu^n_\rho\,\big)\;.
\end{equation*}

\begin{theorem}
\label{t2}
Fix a sequence of probability measures $(\mu_n : n\ge 1)$ on
$\Omega_n$.  Then, there exists a finite constant $C_0 = C_0(\rho)$
such that
\begin{equation*}
H'_n(t) \;\leq\; C_0\, a_n\big\{\, H_n(t)
\,+\, R_d(n) \, \big\} 
\end{equation*}
for all $t\ge 0$.  In this formula, $R_d(n)$ represents the sequence
given by
\begin{equation*}
R_d(n) \;=\;
\begin{cases}
\sqrt{a_n} & \text{for $d=1$}\;, \\
a_n \, \log n & \text{for $d=2$}\;, \\
a_n \, n^{d-2} & \text{for $d\ge 3$}\;.
\end{cases}
\end{equation*}
\end{theorem}

It follows from the previous result and Gronwall's lemma that
\begin{equation}
\label{01}
H_n(t) \;\leq\; \big\{\, H_n(0) \,+\, R_d(n)\,  \big\} \, e^{C_0 a_nt}
\end{equation}
for all $t \geq 0$.

\begin{remark}
\label{rm5}
Assume that $a_n \le \sqrt{\log n}$ and fix $\kappa>0$, $T>0$.  There
exists $n_0 = n_0(\rho, \kappa, T)$ such that
$e^{C_0 a_nt} \le n^\kappa$ for all $n\ge n_0$. In particular,
$H_n(t) \le \big\{\, H_n(0) \,+\, R_d(n)\, \big\} \, n^\kappa$ for
all $0\le t\le T$, $n\ge n_0$.
\end{remark}

The article is organized as follows. In Section \ref{sec6}, we present
a sketch of the proofs of Theorems \ref{t1} and \ref{t1b}. In Section
\ref{sec3}, we prove Theorem \ref{t2}. The proof of this result is
independent from the rest of the paper. In Section \ref{sec4}, we
prove the Boltzmann-Gibbs principle. In Section \ref{sec5}, we prove
the tightness of the sequences $\bb P_{\nu^n_\rho} \circ (X^n)^{-1}$
in dimension $1$ and $2$, and the one of
$\bb P_{\nu^n_\rho} \circ (\bb X^n)^{-1}$ in dimension $3$. In Section
\ref{sec2}, we compute the limits of the finite-dimensional
distributions of these processes, completing the proofs of Theorems
\ref{t1} and \ref{t1b}. In Section \ref{sec9}, we present entropy
bounds used in the article, and, in Section \ref{sec7}, some general
results on continuous-time Markov chains. Finally, in Section
\ref{sec8}, we provide a decomposition of a cylinder function as the
sum of polynomials of fixed degree.

\section{Sketch of the proof of Theorems \ref{t1} and \ref{t1b}}
\label{sec6}

We present in this section the main steps of the proof of Theorem
\ref{t1} and \ref{t1b}.  We first decompose the density field as the
sum of a martingale and integral processes.

Fix $r>0$ and denote by $M^n_t$ the $\mc H_{-r}$-valued process
defined by
\begin{equation}
\label{30}
M^n_t(F) \;:=\; X^n_t(F) \;-\; X^n_0(F) \;-\; \int_0^t
L_n X^n_s(F)\, ds\;, \quad F\,\in\, C^\infty(\bb T^d)\;.
\end{equation}
By \cite[Lemma A.5.1]{kl}, the process $M^n_t(F)$ is a martingale for
each $F$ in $C^\infty(\bb T^d)$. We turn to the integral term.

\begin{asser}
\label{as1}
For every function $F$ in $C^\infty(\bb T^d)$,
\begin{equation*}
\begin{aligned}
L_n X^n(F) \; &=\; \frac{1}{\sqrt{a_n n^d}} \sum_{j,k=1}^d
\sum_{x\in\bb T_n^d}   \{\, h_{j,k} (\tau_x \eta)  \,-\, \tilde
h_{j,k}(\rho)\, \} \, (\Delta^n_{j,k} F) (x/n) 
\\
&+\; \frac{a_n}{n^2}\, \frac{1}{\sqrt{a_n n^d}}
\sum_{x\in\bb T_n^d} \{\, \eta_{x} \,-\, \rho\,\}
\, (\Delta_n F)  (x/n) \;.
\end{aligned}
\end{equation*}
In this formula, $\tilde h_{j,k} (\rho)$ has been introduced in
\eqref{40}, 
\begin{equation*}
(\Delta^n_{j,k} F) (x/n) \;=\; n^2 \, \Big\{ \,
F\Big( \frac{x+e_j}{n}\,\Big) \,-\,
F\Big(\frac{x}{n}\,\Big) \,-\,
F\Big(\frac{x+e_j-e_k}{n}\,\Big) \,+\,
F\Big(\frac{x-e_k}{n}\,\Big) \Big\}\;,
\end{equation*}
and $(\Delta_n F)  \, (x/n) = \sum_{j=1}^d (\Delta^n_{j,j} F) \, (x/n)$.
\end{asser}

\begin{proof}
An elementary computation yields that
\begin{equation}
\label{e5}
\begin{aligned}
L_n X^n(F) \; &=\; \frac{n^2}{\sqrt{a_n n^d}}
\sum_{x\in\bb T_n^d} \sum_{j=1}^d  c_j (\tau_x \eta) \,
[\, \eta_{x} \,-\, \eta_{x+e_j}\,] \, [\, F([x+e_j]/n) \,-\, F(x/n)\,]
\\
&+\; \frac{a_n}{\sqrt{a_n n^d}}
\sum_{x\in\bb T_n^d} \sum_{y}  [\, \eta_{x} \,-\, \eta_{y}\,]^2
[\,1\,-\, 2\eta_x\,]  \, F(x/n) \;,
\end{aligned}
\end{equation}
where the sum over $y$ is carried over all neighbours of $x$.

Apply the gradient condition \eqref{11b} to replace in the first sum
on the right-hand side
$c_j (\eta) \, [\, \eta_{0} \,-\, \eta_{e_j}\,]$ by
$h_{j,k} (\eta) - h_{j,k} (\tau_{e_k} \eta) $. In this difference,
replace $h_{j,k} (\eta)$ by $h_{j,k} (\eta) - \tilde h
(\rho)$. Finally, sum by parts to get that the first term on the
right-hand side is of the previous equation is equal to
\begin{equation*}
\frac{1}{\sqrt{a_n n^d}} \sum_{j,k=1}^d
\sum_{x\in\bb T_n^d}   \{\, h_{j,k} (\tau_x \eta)  \,-\, \tilde
h(\rho)\, \}\, (\Delta^n_{j,k} F) (x/n) \;.
\end{equation*}

We turn to the second sum on the right-hand side of \eqref{e5}.  Write
$1- 2\eta_x$ as $(1-\eta_x) - \eta_x$, note that
$[\, \eta_{x} \,-\, \eta_{y}\,]^2 \, (1-\eta_x) = \eta_y\, (1-\eta_x)$
and $[\, \eta_{x} \,-\, \eta_{y}\,]^2 \, \eta_x= \eta_x\, (1-\eta_y)$,
to conclude that
$[\, \eta_{x} \,-\, \eta_{y}\,]^2 [\,1\,-\, 2\eta_x\,] = \eta_y -
\eta_x$. Hence, a summation by parts yields that the second term on
the right-hand side of \eqref{e5} is equal to
\begin{align*}
\frac{a_n}{n^2}\, \frac{1}{\sqrt{a_n n^d}}
\sum_{x\in\bb T_n^d} \sum_{j=1}^d  \{\, \eta_{x} \,-\, \rho\,\}
\, (\Delta^n_{j,j} F)  (x/n) \;,
\end{align*}
This completes the proof of the assertion.
\end{proof}

Recall the definition of the differential operator $\mc A$, introduced
in \eqref{41}, and the one of the projection operators $\Pi^1_\rho$,
$\Pi^{+2}_\rho$, introduced in Assertion \ref{l42}.  Write
$L_n X^n(F)$ as
\begin{equation}
\label{42}
L_n X^n(F) \;=\; R^n(F) \;+\; B^n(F) \;+\; X^n(\mc A F)\;,
\end{equation}
where
\begin{equation*}
\begin{aligned}
R^n(F) \; &=\; \frac{1}{\sqrt{a_n n^d}} \sum_{j,k=1}^d
\sum_{x\in\bb T_n^d}   \{\, h_{j,k} (\tau_x \eta)  \,-\, \tilde
h_{j,k}(\rho)\, \} \, \big\{ \, (\Delta^n_{j,k} F)  -
(\partial^2_{x_j,x_k} F) \,\big\}(x/n) 
\\
&+\; \frac{a_n}{n^2}\, \frac{1}{\sqrt{a_n n^d}}
\sum_{x\in\bb T_n^d} \{\, \eta_{x} \,-\, \rho\,\}
\, (\Delta_n F)  (x/n) \\
&+\; \frac{1}{\sqrt{a_n n^d}} \sum_{j,k=1}^d
\sum_{x\in\bb T_n^d}
(\Pi^1_\rho h_{j,k}) (\tau_x\eta)  \, (\partial^2_{x_j,x_k} F) (x/n)  \;,
\end{aligned}
\end{equation*}
and
\begin{equation*}
B^n(F) \; =\; \frac{1}{\sqrt{a_n n^d}} \sum_{j,k=1}^d
\sum_{x\in\bb T_n^d}
(\Pi^{+2}_\rho h_{j,k}) (\tau_x\eta)  \, (\partial^2_{x_j,x_k} F) (x/n)  \;.
\end{equation*}

In view of \eqref{30} and \eqref{42}, the process $X^n_t$ can be
decomposed as 
\begin{equation}
\label{e2}
X^n_t(F) \;=\; X^n_0(F) \;+\; M^n_t(F)
\;+\;  \int_0^t R^n_s(F) \, ds
\;+\;  \int_0^t B^n_s(F)\, ds\;+\;  \int_0^t X^n_s(\mc A F)\, ds\;,
\end{equation}
for $F\in C^\infty(\bb T^d)$.

We examine each term of the decomposition \eqref{e2} separately. We
start with $X^n_0$.

\begin{lemma}
\label{l48}
Fix $0<\rho<1$, and let $\mu_n$ be a sequence of probability measures
on $\Omega_n$ such
$\lim_{n\to \infty} a^{-1}_n H_n(\mu_n \,|\, \nu^n_\rho) =0$. Then,
\begin{equation*}
\lim_{n\to\infty} \bb E_{\mu_n} \big[  \, \Vert \mb
X^n_0 \Vert^2_{-r}\, \big] \;=\; 0
\end{equation*}
provided $r>d/2$.
\end{lemma}

\begin{proof}
By definition of the norm $\Vert \, \cdot\, \Vert^2_{-r}$, we have to
show that
\begin{equation*}
\lim_{n\to\infty} \sum_{m\in\bb Z^d} \gamma^{-r}_m\, \bb E_{\mu_n} \Big[  \,  \Big(
\frac{1}{\sqrt{a_n n^d}} \sum_{x\in\bb T_n^d}
\varphi_m (x/n)  \, [\eta_x - \rho] \Big)^2 \, \Big] \;=\; 0 \;,
\end{equation*}
where $\varphi_m(x) = \cos (2\pi\, x\cdot m)$,
$\sin (2\pi\, x\cdot m)$. We consider the cosine case, the other one
being identical. By the entropy inequality, the expectation in the
previous equation is bounded by
\begin{equation*}
\frac{1}{A}  H_n(\mu_n \,|\, \nu^n_\rho) \;+\; \frac{1}{A}  \log
\bb E_{\mu_n} \Big[  \,  e^{ (A/a_n) X_n(m)^2}  \, \Big]  \;,
\end{equation*}
where
$X_n(m) = n^{-d/2} \sum_{x\in\bb T_n^d} \varphi_m (x/n) \, [\eta_x -
\rho]$. Hence, by Corollary \ref{l27} below, there exist finite
constants $0<c_0<C_0<\infty$ such that
\begin{equation}
\label{68}
\bb E_{\mu_n} \big[  \,  \, \Vert \mb X^n_0 \Vert^2_{-r}\,  \, \big] \;\le\;
2\, \Big( \, \frac{1}{A}  H_n(\mu_n \,|\, \nu^n_\rho) \;+\;
\frac{C_0}{a_n}\,\Big)\, \sum_{m\in\bb Z^d} \gamma^{-r}_m
\end{equation}
provided $A<c_0 \, a_n$. Choose $A=c_0 a_n/2$ to complete the proof,
since $\gamma^{-r}_m$ is summable in $m$.
\end{proof}

Let $(\mb R^n_t : t\ge 0)$ be the $\mc H_{-r}$-valued process given by
$\mb R^n_t(F) = \int_0^t R^n_s(F)\, ds$, for $F\in C^\infty(\bb T^d)$,
$t>0$. In Lemma \ref{l26}, we prove that in dimension $d\le 3$, for
any sequence of measures $\mu_n$ such that
$H_n(\mu_n \,|\, \nu^n_\rho) \le R_d(n)$,
\begin{equation}
\label{50}
\lim_{n\to\infty} \bb E_{\mu_n} \Big[ \sup_{0\le t\le T} \Vert \mb
R^n_t \Vert^2_{-r}\, \Big] \;=\; 0
\end{equation}
provided $r>3 + (d/2)$

The next result, the so-called Boltzmann-Gibbs principle, derived by
Brox and Rost \cite{BR84} in the context of equilibrium fluctuations,
asserts that the local fields
$\{a_n\, n^ d\}^{-1/2} \sum_{x\in \bb T^d_n} G (x/n)\, [\, f(\tau_x
\eta^n(t)) - \tilde f(\rho)\,]$ are projected on the density field.
It reads as follows. Denote by $C^{j,k}(\bb R_+ \times \bb T^d)$, $j$,
$k\ge 0$, the set of continuous functions $G: \bb R_+ \times \bb T^d
\to \bb R$ which have $j$ continuous derivatives in time and $k$
continuous derivatives in space.

\begin{theorem}[Boltzmann-Gibbs principle]
\label{p01}
Assume that $d\le 3$ and fix $0<\rho<1$. Let $\mu_n$ be a sequence of
probability measures on $\Omega_n$ such
$H_n(\mu_n \,|\, \nu^n_\rho) \le R_d(n)$. Then,
\begin{equation*}
\lim_{n\to\infty} \bb E_{\mu_n} \Big[\,\Big| \int_0^t
\frac{1}{\sqrt{a_n\, n^ d}} \sum_{x\in \bb T^d_n} G (s, x/n)\, (\Xi_\rho
f) (\tau_x \eta^n(s)) \, ds \, \Big| \, \Big] \;=\; 0\;,
\end{equation*}
for all $t>0$, functions $G$ in $C^{0,1}(\bb R_+ \times \bb T^d)$, and
cylinder functions $f:\{0,1\}^{\bb Z^d} \to \bb R$. In this formula,
$\Xi_\rho$ stands for the operator $\Pi_\rho$, introduced in
\eqref{49}.
\end{theorem}

The proof of this result is given in Section \ref{sec4}, where
quantitative bounds are provided. We show in \eqref{17} that this
statement holds with the absolute value inside the time-integral for
$\Xi_\rho = \Pi^1_\rho$. This result is a simple consequence of a
summation by parts and the entropy estimate. The real challenging part
is to prove Theorem \ref{p01} for $\Xi_\rho = \Pi^{+2}_\rho$. We
stated this result with $\Xi_\rho = \Pi_\rho$ for historical reasons
and to stress that the dynamics projects averages of cylinder
functions on the density field [since
$(\Pi_\rho f)(\eta) = f(\eta) - \tilde f(\rho) - \tilde f'(\rho)\,
(\eta_0-\rho)$].

In view of \eqref{e2}, \eqref{50} and Theorem \ref{p01}, besides
tightness of the process, the proofs of Theorems \ref{t1} and
\ref{t1b} consist, essentially, in showing that the martingale part,
$M^n_t$, converges to a white-noise. This is the content of Section
\ref{sec2}.

\smallskip\noindent{\bf Concentration inequalities.}  We conclude this
section recalling some results on subgaussian random variables.  A
mean-zero random variable $X$ is said to be $\sigma^2$-subgaussian if
$E[\,\exp\{\theta X\}\,] \le \exp\{ \sigma^2 \theta^2/2\}$ for all
$\theta\in \bb R$.

By \cite[Proposition B.1]{jm1}, if $X$ is a $\sigma^2$-subgaussian random
variable,
\begin{equation}
\label{07}
E\big[\, e^{a X^2}\,\big] \;\le\; e^{8 a \sigma^2}
\end{equation}
for all $0<a< 1/4\sigma^2$.

According to Hoeffding's inequality \cite[Lemma 2.2]{BLM13}, a
mean-zero random variable taking values in the interval $[a,b]$
is $[(b-a)^2/4]$-subgaussian.

It follows from this result that if

\begin{lemma}
\label{l04}
Let $X_1, \dots, X_p$ be independent, mean-zero, random variables and
suppose that $X_j$ takes values in the interval $[a_j,b_j]$. Then
$\sum_{1\le j\le p} X_j$ is $A$-subgaussian, where
$A= (1/4) \sum_{1\le j\le p} (b_j-a_j)^2$.
\end{lemma}

This result provides an estimate in the context of the voter model
with stirring.

\begin{corollary}
\label{l27}
Fix a cylinder function $f$ and a function $F: \bb T^d_n \to \bb
R$. Then, there exist constants $0<c_0<C_0<\infty$, depending only on
the cylinder function $f$, such that
\begin{equation*}
\log E_{\nu^n_\rho} \Big[ \, \exp a \Big\{
\frac{1}{\sqrt{n^d}} 
\sum_{x\in\bb T_n^d}   \{\, f (\tau_x \eta)  \,-\, \tilde
f(\rho)\, \} \, F_x \,\Big\}^2 \, \Big] \;\le\; C_0\, a\, \Vert F\Vert^2_\infty\;.
\end{equation*}
for all $0<a< c_0/\Vert F\Vert^2_\infty$.
\end{corollary}

\begin{proof}
Let $p\ge 1$ be the smallest integer such that
$\Xi_p :=\{-p, \dots, p\}^d$ contains the support of the cylinder
function $f$. In particular, under the product measure $\nu^n_\rho$,
the random variables $\tau_x f$ and $\tau_y f$ are independent if
$y-x\not\in \Xi_{2p+1}$. Let $q=2p+1$, and write
\begin{equation*}
\sum_{x\in\bb T_n^d}   \{\, f (\tau_x \eta)  \,-\, \tilde
f(\rho)\, \} \, F_x \;=\; \sum_{z\in \Xi_q} \sum_{y}   \{\, f
(\tau_{z+qy} \eta)  \,-\, \tilde f(\rho)\, \} \, F_{z+qy}\;,
\end{equation*}
where the second sum on the right-hand side is performed over all
$y\in \bb Z^d$ such that $z+qy\in \{0, \dots, n-1\}^d$.

By Schwarz and H\"older inequalities, the expression on the left-hand
side of the statement of the lemma is bounded above by 
\begin{equation*}
\frac{1}{(2q+1)^d} \sum_{z\in \Xi_q}
\log E_{\nu^n_\rho} \Big[ \, \exp a (2q+1)^{2d}  \Big\{
\frac{1}{\sqrt{n^d}}  \sum_{y}   \{\, f
(\tau_{z+qy} \eta)  \,-\, \tilde f(\rho)\, \} \, F_{z+qy}
\,\Big\}^2 \, \Big] \;.
\end{equation*}

By Lemma \ref{l04}, under the measure $\nu^n_\rho$,
$n^{-d/2} \sum_{y} \{\, f (\tau_{z+qy} \eta) \,-\, \tilde f(\rho)\, \}
\, F_{z+qy}$ is an $A$-subgaussian random variable, where
$A= \Vert f\Vert^2_\infty\, \Vert F\Vert^2_\infty$. Thus, for $a <
1/4 (2q+1)^{2d} \Vert f \Vert^2_\infty \Vert F\Vert^2_\infty$, by
\eqref{07}, the previous expression is less than or equal to
\begin{equation*}
8\, a \, \Vert f\Vert^2_\infty\, \Vert F\Vert^2_\infty \;,
\end{equation*}
as claimed.
\end{proof}

\section{Proof of Theorem \ref{t2}}
\label{sec3}

In this section, we prove Theorem \ref{t2}. The statement of the first
result requires some notation.  Denote by $I_n$ the large deviations
rate functional given by
\begin{equation}
\label{13}
I_n(f) \;:=\;  -\,  \int (\, L_n^S \sqrt{f} \,) \, \sqrt{f} \; d \nu^n_\rho\;.
\end{equation}
As $c_j$ does not depend on the variables $\eta_0$, $\eta_{e_j}$, an
elementary computation yields that
\begin{equation*}
I_n(f)  \;=\; \frac 12\, \sum_{j=1}^d \sum_{x\in\bb T^d_n} 
\int c_j(\tau_x \eta)\, \big[\,  \sqrt{f (\sigma^{x,x+e_j} \eta)} \,-\,
\sqrt{f(\eta)}\, \big]^2 \; d \nu^n_\rho\;.
\end{equation*}

Let $L^{S,*}$, $L^{V,*}_n$ be the adjoints of the generators $L^{S}$,
$L^{V}_n$ in $L^2(\nu^n_\rho)$, respectively. Thus, for all $f$, $g
\in L^2(\nu^n_\rho)$,
\begin{equation*}
\int (\, L_n^B f \,) \, g \; d \nu^n_\rho \;=\;
\int f\, (\, L_n^{B,*} g \,) \; d \nu^n_\rho \;, 
\end{equation*}
for $B = S$ and $V$.

Since the Bernoulli measures $\nu^n_\rho$ satisfy the detailed balance
conditions for the exclusion dynamics, $L^{S,*}_n = L^{S}_n$.  On the
other hand, an explicit computation yields that
\begin{align*}
(L_n^{V,*} h)(\eta)  \; &=\; \sum_{x \in \bb T^d_n}
\sum_{y : |y-x|=1} \Big\{\, \eta_x \, \eta_y \, \frac{1-\rho}{\rho}
\,+\, (1-\eta_x)\, (1-\eta_y) \,
\frac{\rho}{1-\rho}\, \Big\} \, h(\sigma^x \eta) \\
&-\; \sum_{x \in \bb T^d_n}
\sum_{y : |y-x|=1} \big(\, \eta_x \,-\, \eta_y  \big)^2 \,
h(\eta) 
\end{align*}
for all functions $h: \Omega_n \to \bb R$.  

Denote by $\mb 1: \Omega_n \to \bb R$ the function which is constant
equal to $1$, and by $V$ the function $L_n^{V,*} \mb 1$. Note that $V$
would vanish if $\nu^n_\rho$ were invariant for $L_n^{V}$ because in
this case $L_n^{V,*}$ would be the generator of a Markov chain. Thus,
in a vague sense, $V = L_n^{V,*} \mb 1$ indicates how far is
$\nu^n_\rho$ from the stationary state for $L_n^{V}$. It follows from
the explicit formula for $L_n^{V,*}$ that
\begin{equation}
\label{05}
V(\eta) \;:=\;
(L_n^{V,*} \mb 1)\, (\eta) \;=\; \sum_{x \in \bb T^d_n}
\sum_{y : |y-x|=1} \omega_x \, \omega_y
\;= \; 2 \sum_{j=1}^d \sum_{x \in \bb T^d_n}
\omega_x \, \omega_{x+e_j} \;, 
\end{equation}
where
\begin{equation*}
\omega_x \;:=\;  \frac{\eta_x-\rho}{\sqrt{\rho(1-\rho)}} \;, \quad
x\,\in\, \bb T^d_n\;.
\end{equation*}
Notice that $\{\omega_x; x\in \bb T^d_n\}$ is an
orthonormal family with respect to the measure $\nu^n_\rho$.

\begin{proposition}
\label{l06}
Fix a probability measure $\mu_n$ on
$\Omega_n$, and let $f_t^n$, $t\ge 0$, be the density of
$\mu_n S^n(t)$ with respect to $\nu^n_\rho$,
\begin{equation*}
f_t^n \;:=\; \frac{d\, \mu_n\,  S^n(t)}{d\, \nu^n_\rho}\;\cdot
\end{equation*}
Then,
\begin{equation*}
H'_n(t) \;\leq\; -\, 2\, n^2 \, I_n(f_t^n)
\;+\;   a_n \, \int  V \, f_t^n \; d \nu^n_\rho\;,
\end{equation*}
for all $t\ge 0$. 
\end{proposition}

\begin{proof}
By \cite[equation (A1.9.1)]{kl}, the density $f^n_t$ solves the
equation
\begin{equation}
\label{03}
\frac{d}{dt} \, f_t^n \;=\; (\, n^2 \, L^{S,*}_n
\,+\, a_n \, L^{V,*}_n\,\} \, f_t^n \;,
\end{equation}
where, recall, $L^{S,*}_n$, $L^{V,*}_n$ represent the adjoints of the
generators $L^{S}_n$, $L^{V}_n$ in $L^2(\nu^n_\rho)$,
respectively. 

On the other
hand, by \eqref{02},
\begin{equation*}
H_n(t) \;=\; H_n(f^n_t) \;=\; \int f_t^n \, \log f_t^n \; d \nu^n_\rho \;.
\end{equation*}
By relative entropy bound \cite [Theorem A1.9.2]{kl} and \eqref{03}, 
\begin{equation}
\label{04}
H_n'(t) \;\leq\; -\, 2\, n^2 \, I_n(f_t^n)
\;+\;   a_n \, \int (\, L_n^V \log f_t^n\,) \, f_t^n
\; d \nu^n_\rho\;,
\end{equation}
where $I_n$ is the functional introduced in \eqref{13}.

Since $\log r \le r -1$ for $r>0$, $(\, L_n^V \log f_t^n\,) \, f_t^n
\le L_n^V  f_t^n$. The second term on the right-hand side of \eqref{04}
is thus bounded by
\begin{equation*}
a_n \, \int  L_n^V f_t^n \; d \nu^n_\rho \;=\;
a_n \, \int  (\, L_n^{V,*} \mb 1 \,) \, f_t^n \; d \nu^n_\rho
\;=\; a_n \, \int  V  \, f_t^n \; d \nu^n_\rho\;,
\end{equation*}
as claimed.
\end{proof}

Let $m_\ell$, $\ell\ge 1$, be the uniform measure on the cube
$\color{bblue} \Lambda_\ell \,:= \, \{0 \,,\, 1\,,\, \dots \,,\,
\ell-1\}^d$,
\begin{equation*}
m_\ell(z) \;:=\; \frac{1}{ \ell^{d}}
\, \chi_{\Lambda_\ell} (z)\;,
\end{equation*}
where $\chi_A$ stands for the indicator of the set $A$.

Let $m^{(2)}_\ell$ be the convolution of $m_\ell$ with itself: 
\begin{equation*}
m^{(2)}_\ell(z) = \sum_{y \in \bb T^d_n}
m_\ell(y) \, m_\ell(z-y) \;,
\end{equation*}
Notice that $m^{(2)}_\ell$ is supported on the cube
$\Lambda_{2\ell-1}$.

Denote by $\omega_x^\ell$ the average of $\omega_{x+z}$ with respect
to the measure $m^{(2)}_\ell$:
\begin{equation}
\label{20}
\omega_x^\ell
\;=\; \sum_{y \in \bb T^d_n} m^{(2)}_\ell(y) \, \omega_{x+y}
\;=\; \sum_{y \in \Lambda_{2\ell+1}} m^{(2)}_\ell(y) \, \omega_{x+y}\;,
\end{equation}
and let $V_\ell :\Omega_n \to \bb R$ be given by
\begin{equation}
\label{06}
V_\ell(\eta) \;:=\; 2\, \sum_{j=1}^d \sum_{x \in \bb T^d_n}
\omega_x \, \omega_{x+e_j}^\ell \;.
\end{equation}
A change of variables yields that
\begin{equation}
\label{16}
V_\ell(\eta) \;:=\; \sum_{j=1}^d \sum_{x \in \bb T^d_n}
\Big(\, \sum_{y \in \Lambda_\ell} m_\ell(y) \,
\omega_{x-y} \, \Big) \,
\Big(\, \sum_{z \in \Lambda_\ell} m_\ell(z)
\,\omega_{x+e_j+z}\,\Big) \;.
\end{equation}
Notice that the averages are performed over disjoint sets due to the
definition of $m_\ell$: for every $x$ and $j$, the sets
$\{x-y : y \in \Lambda_\ell\}$ and $\{x+e_j+z : z \in \Lambda_\ell\}$
are disjoints.

Let $(g_d(n) : n\ge 1)$ be the sequence defined by
\begin{equation}
\label{27}
g_d(n) =
\left\{
\begin{array}{c@{\;,\quad}l}
n & d=1\\
\log n & d=2\\
1 & d \geq 3\; . 
\end{array}
\right.
\end{equation}

\begin{proposition}
\label{l03}
There exists a finite constant $C_1(\rho)$, depending only on $\rho$
and $\mf c_0$, such that
\begin{equation*}
a_n \int \{\, V(\eta) \;-\;  V_\ell(\eta) \,\} \, f\; d\nu^n_\rho
\;\le\; \delta\, n^2 \, I_n(f) \;+\;
\frac{C_1(\rho) \, a_n^2 \, \ell^d \, g_d(\ell)}
{\delta\, n^2} \,\big\{ H_n(f) \;+\; (n/\ell)^{d} \big\}
\end{equation*}
for every $1\le \ell < n/4$, $\delta>0$ and density $f$ with respect
to $\nu^n_\rho$.
\end{proposition}

The proof of this proposition is divided in several steps.

\smallskip\noindent{\bf Integration by parts.}  For $x\in \bb T^d_n$,
$1\le j\le d$, let $I_{x,x+e_j}$ be the functional $I_n$ restricted to
the bond $\{x,x+e_j\}$: 
\begin{equation*}
I_{x,x+e_j}(h) \;=\; \frac 12\, \int
c_j(\tau_x \eta)\,  \big\{\, \sqrt{ h(\sigma^{x,x+e_j}\eta)}
\,-\, \sqrt{h(\eta)} \, \big\}^2 \; d \nu^n_\rho\;,
\end{equation*}
$h: \Omega_n \to \bb R$. The proof of the next result is omitted,
being similar to the one of \cite[Lemma 3.1]{L92}. Recall the
definition of the constant $\mf c_0$ introduced in \eqref{14}.

\begin{lemma}
\label{l2}
Fix $x \in \bb T^d_n$, $1\le j\le d$ and $h: \Omega_n \to \bb R$ such
that $h(\sigma^{x,x+e_j}\eta) = h(\eta)$ for all $\eta \in
\Omega_n$. Then,
\begin{equation*}
\int h \, [\eta_y -\eta_x]\, f\; d\nu^n_\rho
\;\leq\; \frac{\beta}{2} \, I_{x,x+e_j}(f)
\;+\; \frac{1}{2\, \mf c_0\, \beta} \, \int h^2 \, f \; d\nu^n_\rho
\end{equation*}
for all $\beta >0$ and density $f: \Omega_n \to [0,\infty)$ with
respect to $\nu^n_\rho$.
\end{lemma}

\smallskip\noindent{\bf Flows.}
Let $G$ be a finite set. For probability measures $\mu$ and $\nu$ on
$G$, a function $\Phi:G\times G\to\bb R$ is called a {\it flow
  connecting $\mu$ to $\nu$} if
\begin{enumerate}
\item $\Phi(x,y)=-\Phi(y,x)$, for all $x,y\in G$;
\item $\sum_{y\in G} \Phi(x,y) = \mu(x) - \nu(x)$, for all $x\in G$.
\end{enumerate}

Next result is \cite[Theorem 3.9]{jm1}. Recall the definition of the
sequence $g_d(\ell)$ introduced in the statement of Theorem \ref{t2}. 

\begin{lemma}
\label{l01}
There exist a finite constant $C_d$, depending only on the dimension
$d$, and, for all $\ell\ge 1$, a flow $\Phi_\ell$ connecting the Dirac
measure at the origin to the measure $m^{(2)}_\ell$ which is supported
in $\Lambda_{2\ell-1}$ and on nearest-neighbour bonds:
\begin{equation*}
\Phi_\ell(x,y) \;=\;0
\end{equation*}
if $\Vert y - x \Vert \not = 1$ and if $\{x,y\} \not\subset
\Lambda_{2\ell-1}$. 
Moreover,
\begin{equation*}
\sum_{j=1}^d  \sum_{x\in\bb T_n^d}\Phi_\ell(x,x+e_j)^2
\;\le\; C_d\; g_d(\ell)\;.
\end{equation*}
\end{lemma}

Consider the partial order $\prec$ on $\bb Z^d$ defined by
$(x_{1}, \dots, x_{d}) \prec (y_{1}, \dots, y_{d})$ if $x_j\le y_j$
for all $1\le j\le d$.  Fix a subset $A$ of $\bb Z^d$
$A = \{\mb x_k : k \in J\}$, $\mb x_k = (x_{k,1}, \dots, x_{k,d})$.  A
point $\mb x_k$ in $A$ is said to be \emph{maximal} if
$\mb x_k \prec \mb x_l$ entails that $\mb x_k = \mb x_l$.

Every finite subset $A$ of $\bb Z^d$ has at least one maximal element.
Fix a finite subset $A$ of $\bb Z^d$ with at least two elements,
$A = \{\mb x_1 , \dots , \mb x_p\}$, $p\ge 2$. Denote by $\mb x_A$ a
maximal element of $A$, and let $A_\star = A \setminus \{\mb x_A\}$.

Recall the definition of the average $\omega^\ell_{x}$ introduced in
\eqref{20}.

\begin{lemma}
\label{l17}
Fix a finite subset $A$ of $\bb Z^d$ with at least two elements, a
function $G:\bb T^d_n \to \bb R$, $b_n \in \bb R$ and $\ell\ge 1$. Let
\begin{equation*}
W \;=\; \sum_{x\in \bb T^d_n} G_x \, \omega_{x+A_\star}\,
\{\omega_{x+\mb x_A} - \omega^\ell_{x+\mb x_A}\,\}\;.
\end{equation*}
Then,
\begin{equation*}
b_n \int  W \, f\; d\nu^n_\rho \;
\le\; \frac{\beta}{2} \,  n^2\, I_{n}(f) 
\; +\; \frac{b_n^2}{2\, \mf c_0\, \beta\, n^2\, \chi(\rho)}
\sum_{k=1}^d  \sum_{x\in\bb T_n^d}
\int ( H^{(\ell)}_{k,x})^2 \, f\; d\nu^n_\rho
\end{equation*}
for all $\beta>0$, density $f$ with respect to $\nu^n_\rho$ and
$n> 4\ell$. In this formula, $\chi(\rho) = \rho (1-\rho)$ is the
static compressibility of the exclusion process, introduced in the
statement of Theorem \ref{t1}, and
\begin{equation*}
H^{(\ell)}_{k,x} \;=\; \sum_{\{y,y+e_k\} \subset \Lambda_{2\ell-1}}
\Phi_\ell (y,y+e_k) \, G(\, x-\mb x_A-y \,) \,
\omega_{x-\mb x_A-y+A_\star}
\;,
\end{equation*}
where $\Phi_\ell$ is the flow introduced in Lemma \ref{l01}.
\end{lemma}

\begin{proof}
Assume without loss of generality that $\mb x_A=0$ [Otherwise, in the
sum defining $W$ change variables as $x'=x+\mb x_A$]. This means that
the origin is a maximal point in $A$. Let $B=A_\star$, and rewrite $W$
as
\begin{equation*}
W(\eta)\; =\; \sum_{x \in \bb T^d_n} G_x \, \omega_{x+B}
\sum_{y\in \Lambda_{2\ell-1}} \omega_{x+y} \,
\{\, \delta_{0}(y)  \,-\,  m^{(2)}_\ell(y) \,\}\;,
\end{equation*}
where $\delta_0$ stands for the Dirac measure concentrated at $0$.

Denote by $\Phi_\ell$ the flow introduced in Lemma \ref{l01}. In the
previous equation, we may replace $\Lambda_{2\ell-1}$ by $\bb
Z^d$. This simplifies the summation by parts performed below. After
this replacement, since the flow connects $\delta_0$ to
$m^{(2)}_\ell$, it is anti-symmetric and supported on
nearest-neighbour bonds, the sum over $y$ becomes
\begin{align*}
& \sum_{y\in \bb Z^d} \omega_{x+y} \,
\sum_{z: \Vert z\Vert =1} \Phi_\ell (y,y+z) \\
&\quad \;=\; \sum_{k=1}^d \sum_{y\in \bb Z^d} \omega_{x+y} \,
\{\, \Phi_\ell (y,y+e_k) \,-\,
\Phi_\ell (y-e_k,y)\,\} \;.
\end{align*}
Performing a summation by parts, this last sum becomes
\begin{equation*}
\sum_{k=1}^d \sum_{y\in \Lambda_{2\ell-1}} \Phi_\ell (y,y+e_k)
\,
\{\, \omega_{x+y}  \,-\, \omega_{x+y+e_k} \,\}\;.
\end{equation*}
As $\Phi_\ell$ is supported on $\Lambda_{2\ell-1}$, we may restrict
the sum over $y$ to the set of all points in $\bb Z^d$ such that
$\{y,y+e_k\} \subset \Lambda_{2\ell-1}$.

Perform a change of variables $x'=x+y$ to conclude that
\begin{align*}
W(\eta) \; & =\; \sum_{k=1}^d  \sum_{x\in\bb T_n^d}
\{\, \omega_{x}  \,-\, \omega_{x+e_k} \,\}
\sum_{\{y,y+e_k\} \subset \Lambda_{2\ell-1}} \Phi_\ell (y,y+e_k) \,
G_{x-y}\,  \omega_{x-y+B}  \\
&= \; \sum_{k=1}^d  \sum_{x\in\bb T_n^d}
\{\, \omega_{x}  \,-\, \omega_{x+e_k} \,\} \, H^{(\ell)}_{k,x} \;,
\end{align*}
where $H^{(\ell)}_{k,x}$ has been introduced in the statement of the
lemma.  Since the origin is a maximal point of $A$, the support of
$H^{(\ell)}_{k,x}$ is disjoint from $\{x,x+e_k\}$ in the sense that
the indices $z$ of $\omega_z$ which appear in the definition of
$H^{(\ell)}_{k,x}$ are different from $x$ and $x+e_k$. In particular,
\begin{equation}
\label{15b}
H^{(\ell)}_{k,x} (\sigma^{x, x+e_k} \eta) \;=\;
H^{(\ell)}_{k,x} (\eta)\;.
\end{equation}

Fix a density $f: \Omega_n \to [0,\infty)$ with respect to
$\nu^n_\rho$. In view of the formula for $W(\eta)$, by Lemma \ref{l2}
and \eqref{15b},
\begin{align*}
b_n \int W \, f\; d\nu^n_\rho 
\; \le\; \frac{\beta\, n^2}{2} \, \sum_{k=1}^d  \sum_{x\in\bb T_n^d}
I_{x,x+e_k}(f) 
\; +\; \frac{b_n^2}{2\, \mf c_0\, \beta\, n^2\, \chi(\rho)}
\sum_{k=1}^d  \sum_{x\in\bb T_n^d}
\int ( H^{(\ell)}_{k,x})^2 \, f\; d\nu^n_\rho
\end{align*}
for all $\beta>0$, as claimed.
\end{proof}

Recall from \eqref{05}, \eqref{06} the definitions of the functions
$V$, $V_\ell$. Lemma \ref{l17} with $b_n = a_n$, $G=2$,
$A=\{0,e_j\}$, $\mb x_A = e_j$, $\beta = 2\delta n^2/d$ yields the
next result.

\begin{corollary}
\label{l18}
For all $n$ large enough and density $f$ with respect to $\nu^n_\rho$,
\begin{equation}
\label{23}
\begin{aligned}
& a_n \int \{\, V(\eta) \;-\;  V_\ell(\eta) \,\} \, f\; d\nu^n_\rho \\
&\quad \;\le\; \delta\, n^2 \, I_n(f) \;+\;
\frac{d\, a_n^2}{4\, \mf c_0\, \delta\, n^2\, \chi(\rho)}
\sum_{j,k=1}^d  \sum_{x\in\bb T_n^d}
\int ( H^{(\ell)}_{j,k,x})^2 \, f\; d\nu^n_\rho
\end{aligned}
\end{equation}
for all $\delta>0$.
where
\begin{equation*}
H^{(\ell)}_{j,k,x} \;=\; 2\, \sum_{\{y,y+e_k\} \subset \Lambda_{2\ell-1}}
\Phi_\ell (y,y+e_k) \, \omega_{x-y-e_j}  \;.
\end{equation*}
\end{corollary}

\begin{proof}[Proof of Proposition \ref{l03}]
In view of Corollary \ref{l18}, we have to estimate the second term on
the right-hand side of \eqref{23}

By the entropy inequality, the second term of this expression is
bounded by
\begin{equation}
\label{09}
\frac{1}{\gamma}\, H_n(f) \;+\;
\frac{1}{\gamma}\, \log \int \exp \Big\{
\frac{d\, a_n^2\,\gamma }{4\, \mf c_0\, \delta\, n^2\, \chi(\rho)}
\sum_{j,k=1}^d  \sum_{x\in\bb T_n^d}
\int ( H^{(\ell)}_{j,k,x})^2\, \Big\}
\; d\nu^n_\rho 
\end{equation}
for all $\gamma>0$. Under the measure $\nu^n_\rho$, the variables
$H^{(\ell)}_{j,k,x}$, $H^{(\ell)}_{l,m,y}$ are independent if
$\Vert x - y \Vert \ge 2d\ell$. Hence, rewriting the sum over $x$ as
$C_0 \ell^d$ sums of terms spaced by $2d\ell$ and applying H\"older's
inequality [see the Proof of Lemma 6.1.8 in \cite{kl} for a detailed
presentation of this step] yield that the second term of the previous
expression is bounded by
\begin{equation}
\label{08}
\frac{C_0}{\gamma\, \ell^d}\, \sum_{j,k=1}^d  \sum_{x\in\bb T_n^d} \log
\int \exp \Big\{
\frac{C_0 \, a_n^2\,\gamma\,\ell^d }{\delta\, n^2\, \chi(\rho)}
( H^{(\ell)}_{j,k,x})^2\, \Big\} \; d\nu^n_\rho
\end{equation}
for some finite constants $C_0$ depending only on the dimension.

By Lemma \ref{l04}, $H^{(\ell)}_{j,k,x}$ is a
$\sigma^2_\ell$-subgaussian random variable, where
\begin{equation*}
\sigma^2_\ell \;=\; \frac{C_0}{\chi(\rho)}\,
\sum_{\{y,y+e_k\} \subset \Lambda_{2\ell-1}}
\Phi_\ell (y,y+e_k)^2\;.
\end{equation*}
By Lemma \ref{l01}, $\sigma^2_\ell \le C_0 \, g_d(\ell)/\chi(\rho)$.
Therefore, by \eqref{07}, the sum \eqref{08} is bounded by
\begin{equation*}
C_1(\rho) \, \delta^{-1} \, a_n^2\, n^{d-2}  \, g_d(\ell) \;.
\end{equation*}
for all $\gamma$ such that $C_1(\rho) (a_n/n)^2 \ell^d g_d(\ell)
\gamma<\delta/4$.

Choose $\gamma^{-1} = C_1(\rho) (a_n/n)^2 \ell^d g_d(\ell)
\delta^{-1}$ to conclude that \eqref{09} is bounded above by
\begin{align*}
& \frac{C_1(\rho) \, a_n^2 \, \ell^d \, g_d(\ell)}
{\delta\, n^2} \,  H_n(f) \;+\; C_1(\rho)
\, \delta^{-1} \,a_n^2\, n^{d-2}  \, g_d(\ell) \\
&\quad =\; \frac{C_1(\rho) \, a_n^2 \, \ell^d \, g_d(\ell)}
{\delta\, n^2} \,\big\{ H_n(f) \;+\; (n/\ell)^{d}\, \big\} \;,
\end{align*}
as claimed.
\end{proof}

Next result is the main step in the estimation of $V_\ell$. Recall
that $\chi(\rho) = \rho\, (1-\rho)$. For a finite subset $B$ of
$\bb Z^d$, $\ell \ge 1$ and a function $G:\bb T^d_n \to \bb R$, let
\begin{equation*}
M_\ell(x) \;=\; \sum_{w \in \Lambda_\ell} m_\ell(y) \, G(x-y)\, \omega_{x-y+B}\;.
\end{equation*}

\begin{lemma}
\label{l16}
There exists a finite constant $C_0$, depending only on the dimension,
such that
\begin{equation*}
\int \sum_{x \in \bb T^d_n} M_\ell(x) ^2 \, f\; d\nu^n_\rho
\;\le\; C_0 \, \frac{ \Vert G \Vert^2_\infty }
{\chi(\rho)^{|B|}} \,\big\{ H_n(f) \;+\; (n/\ell)^d \, \big\}
\end{equation*}
for all density $f$ with respect to $\nu^n_\rho$ and all $n> 4
\ell\ge 4$.
\end{lemma}

\begin{proof}
Let
\begin{equation*}
W(\eta) \;=\; \sum_{x \in \bb T^d_n} M_\ell(x)^2\;.
\end{equation*}
By the entropy inequality,
\begin{equation*}
\int W \, f\; d\nu^n_\rho
\;\le\; \frac{1}{\gamma} \, H_n(f) \;+\; \frac{1}{\gamma} \log
\int e^{\gamma\, W } \; d\nu^n_\rho
\end{equation*}
for all $\gamma>0$. Repeating the argument presented below \eqref{09},
we bound the second term of this expression by
\begin{equation}
\label{21}
\frac{C_0}{\gamma\, \ell^d} \sum_{x \in \bb T^d_n} \log
\int \exp\Big\{  C_0 \, \gamma \, \ell^d \,
M_\ell(x) ^2 \Big\} \; d\nu^n_\rho\;.
\end{equation}

If $B$ were a singleton, under the measure $\nu^n_\rho$, the variables
$\omega_{x-y+B}$ would be independent. Since this may not be the case,
we divide the sum further. Let $p\ge 1$ be the smallest integer such
that $B\subset \Xi_p :=\{-p, \dots, p\}^d$, and rewrite $M_\ell(x)$ as
\begin{equation*}
M_\ell(x) \;=\; \sum_{z\in \Xi_p} \sum_{w \in \Lambda_\ell(z)} m_\ell(z +
p w) \, G(x-z - pw)\, \omega_{x-z - pw +B} \;:=\;
\sum_{z\in \Xi_p}  M_\ell(z,x) \;,
\end{equation*}
where the second sum is performed over all $w\in \bb Z^d$ such that
$z+pw\in \Lambda_\ell$. Now, for each fixed $x$, $z$, the variables
$\{ \omega_{x-z - pw +B} : w \in \Lambda_\ell(z)\}$ are
independent. Apply H\"older's inequality once more to bound \eqref{21}
by
\begin{equation}
\label{22}
\frac{C_0}{\gamma\, \ell^d} \sum_{x \in \bb T^d_n} \sum_{z\in \Xi_p} \log
\int \exp\Big\{  C_0 \, \gamma \, \ell^d \, M_\ell(z,x) ^2 \Big\}
\; d\nu^n_\rho\;,
\end{equation}
where the value of the constant $C_0$ has changed.

By definition of $m_\ell$,
$\sum_{y \in \Lambda_\ell} m_\ell(y)^2 \le C_0 \ell^{-d}$ for some
finite constant $C_0$. Hence, by Lemma \ref{l04}, under the measure
$\nu^n_\rho$, $M_\ell(z,x)$ is a $\sigma^2_\ell$-sub\-gauss\-i\-an
random variable, where
$\sigma^2_\ell = C_0 \, \Vert G \Vert^2_\infty /\chi(\rho)^{|B|}
\ell^{d}$ for some finite contant $C_0$.

Therefore, taking $\gamma = c_0 \chi(\rho)^{|B|} / \Vert G
\Vert^2_\infty$ for some positive constant $c_0$, depending only on
the dimension, by \eqref{07}, the sum \eqref{22} is bounded by
\begin{equation*}
C_0 \, \frac{ \Vert G \Vert^2_\infty }
{\chi(\rho)^{|B|}} \,\frac{n^d}{\ell^d}\;\cdot
\end{equation*}
To complete the proof of the lemma, it remains to recollect the
previous estimates.
\end{proof}

\begin{corollary}
\label{l07}
There exists a finite constant $C_1(\rho)$, depending only on $\rho$
and on the dimension, such that
\begin{equation*}
\int V_\ell(\eta) \, f\; d\nu^n_\rho
\;\le\; C_1(\rho)  \,\big\{ H_n(f) \;+\; (n/\ell)^d \, \big\}
\end{equation*}
for all density $f$ with respect to $\nu^n_\rho$.
\end{corollary}

\begin{proof}
The proof is similar to the one of Proposition \ref{l03}. In view of
\eqref{16}, by Young's inequality $(a+b)^2 \le (1/2) a^2 + (1/2) b^2$,
$V_\ell(\eta)$ is bounded by
$V^{(1)}_\ell(\eta) + V^{(2)}_\ell(\eta)$, where
\begin{equation*}
V^{(1)}_\ell(\eta) \;:=\; \frac{1}{2}\, \sum_{j=1}^d \sum_{x \in \bb T^d_n}
\Big(\, \sum_{y \in \Lambda_\ell} m_\ell(y) \,
\omega_{x-y} \, \Big)^2 \;.
\end{equation*}
The term $V^{(2)}_\ell$ is similar to $V^{(1)}_\ell$, with the average
inside the square replaced by
$\sum_{z \in \Lambda_\ell} m_\ell(z) \,\omega_{x+e_j+z}$.

To complete the proof, it remains to apply Lemma \ref{l16} with $G=1$.
\end{proof}

Next result follows from Proposition \ref{l03} and Corollary
\ref{l07}.

\begin{lemma}
\label{l12}
There exists a finite constant $C_1(\rho)$, depending only on $\rho$,
$\mf c_0$ and the dimension, such that
\begin{equation*}
a_n\, \int V \, f\; d\nu^n_\rho
\;\le\; \delta\, n^2 \, I_n(f) \;+\;
C_1(\rho)  \, a_n \,  \Big\{ 1 \;+\;
\frac{a_n \, \ell^d \, g_d(\ell)}{\delta\, n^2} \,\Big\}
\,\big\{ H_n(f) \;+\; (n/\ell)^{d} \, \big\} 
\end{equation*}
for all $\delta>0$ and density $f$ with respect to $\nu^n_\rho$.
\end{lemma}

\begin{proof}[Proof of Theorem \ref{t2}]
By Proposition \ref{l06} and Lemma \ref{l12} with $\delta =1$ and
$f=f^n_t$, $H'_n(t)$ is bounded by
\begin{align*}
C_1(\rho) \, a_n \, \Big\{ 1 \;+\;
\frac{a_n \, \ell^d \, g_d(\ell)}{n^2} \,\Big\}
\,\Big\{ H_n(f^n_t) \;+\; (n/\ell)^{d} \, \Big\} 
\end{align*}
for some finite constant $C_1(\rho)$.

At this point, the natural choice is $\ell = \ell_n$ so that
$\ell^d_n \, g_d(\ell_n) = n^2/a_n$. Thus, define the sequence
$(\ell_n : n\ge 1)$ by
\begin{equation}
\label{19}
\ell^d_n \;=\;
\begin{cases}
n/\sqrt{a_n} & \text{in } d\,=\, 1\;, \\
n^2/[\,a_n \, \log n\,]  & \text{in } d\,=\, 2\;, \\
n^2/a_n & \text{in } d\,\ge \, 3\;.
\end{cases}
\end{equation}
With these choices, the previous expression is bounded by
\begin{equation*}
C_1(\rho) \, a_n \, 
\,\big\{ H_n(f^n_t) \;+\; R_d(n) \, \big\} \;,
\end{equation*}
as claimed.
\end{proof}

Next result will be needed in the proof of the Boltzmann-Gibbs
principle in Section \ref{sec4}. Its proof is similar to the one of
Corollary \ref{12}.

\begin{lemma}
\label{l13}
There exists a finite constant $C_1(\rho)$, depending only on $\rho$,
$\mf c_0$ and the dimension, such that
\begin{equation*}
a_n\, \int V \, f\; d\nu^n_\rho
\;\le\; \delta\, n^2 \, I_n(f) \;+\;
C_1(\rho)  \, a_n \,  \Big\{ 1 \;+\;
\frac{a_n \, \ell^d \, g_d(\ell)}{\delta\, n^2} \,\Big\}
\,\big\{ H_n(f) \;+\; (n/\ell)^{d} \, \big\} 
\end{equation*}
for all $\delta>0$ and density $f$ with respect to $\nu^n_\rho$.
\end{lemma}

\begin{proof}
We say that a non-empty subset $A = \{\mb x_1 , \dots , \mb x_p\}$ of
$\bb Z^d$, $\mb x_k = (x_{k,1}, \dots, x_{k,d})$, $1\le k\le p$, is
\emph{bounded by the origin} if each element of $A$ has a negative
coordinate: for each $1\le k\le p$, there exists $1\le j\le d$ such
that $x_{k,j}<0$.

Clearly, any finite subset $B$ of $\bb Z^d$ with two or more elements
can be written as a translation of a set $\{0\} \cup A$, where $A$ is
bounded by the origin.
\end{proof}

\section{The Boltzmann-Gibbs principle}
\label{sec4}

We prove in this section Theorem \ref{p01}. We start proving this
result for $\Xi_\rho = \Pi^1_\rho$, and then turn to the case
$\Xi_\rho = \Pi^{+2}_\rho$. Throughout this section, $f^n_t$,
$t\ge 0$, represents the density of the measure $\mu_n S^n(t)$ with
respect to $\nu^n_\rho$, where $\mu_n$ is a probability measure on
$\Omega_n$.

\begin{lemma}
\label{l05}
Fix a function $G$ in $C^{0,1}(\bb R_+ \times \bb T^d)$ and a sequence
of measures $\mu_n$ on $\Omega_n$. Then,
\begin{align*}
& \bb E_{\mu_n} \Big[\,
\int_0^t \frac{1}{\sqrt{a_n\, n^ d}} \,\Big|\, \sum_{x\in \bb T^d_n}
G (s, x/n)\,
[\, \eta^n_x(s) - \eta^n_{x+e_j}(s)\,]  \,\Big| \,  \, ds \, \Big] \\
&\qquad \le\;
\frac{1}{\gamma} \int_0^t H_n(f^n_s) \; ds \;+\; \frac{t}{\gamma}
\log 2 \;+\;
\frac{3\, \gamma\, t}{a_n\, n^2}\,
\, \sup_{0\le s\le t} \Vert \partial_{x_j} G(s) \Vert^2 _\infty
\end{align*}
for every $1\le j\le d$, $t>0$,
$0< \gamma \le \, \sqrt{a_n\, n^ {d+2}}/
\sup_{0\le s\le t} \Vert \partial_{x_j} G(s) \Vert_\infty$, $n\ge 1$.
\end{lemma}

\begin{proof}
Let
$W (s,\eta) = \{a_n\, n^ d\}^{-1/2} \sum_{x\in \bb T^d_n} [\, G (s,x/n)
\,-\, G(s,x-e_j/N)\,] \, [\, \eta^n_x - \rho\,]$.  By the entropy
inequality, the expectation appearing in the statement of the lemma is
bounded by
\begin{equation*}
\frac{1}{\gamma} \int_0^t H_n(f^n_s) \, ds \;+\;
\int_0^t \frac{1}{\gamma}  \log \int e^{\gamma\, |  \, W(s) \,|} \;
d\nu^n_\rho \, ds
\end{equation*}
for every $\gamma>0$.

As $\exp\{\, |a|\,\} \le e^a + e^{-a}$ and
$e^b + e^{-b} \le 2 \max \{ e^b , e^{-b}\}$, by linearity of the
expectation,
\begin{equation*}
\frac{1}{\gamma} \log \int e^{\gamma\,  |\,W(s)\,|} \; d\nu^n_\rho \;\le\;
\frac{\log 2}{\gamma} \;+\; \max_{b=\pm 1} \frac{1}{\gamma}
\log \int e^{b\, \gamma\,  W(s)} \; d\nu^n_\rho\;.
\end{equation*}
We estimate the second term with $b=1$, as the argument applies to
$b=-1$. As $\nu^n_\rho$ is a product measure,
\begin{equation*}
\frac{1}{\gamma} \log \int e^{\gamma\,  W(s)} \; d\nu^n_\rho \;=\;
\frac{1}{\gamma} \sum_{x\in \bb T^d_n} \log \int e^{b_n(s)\,
  (\eta_x-\rho) } \; d\nu^n_\rho
\end{equation*}
where
$b_n(s) = \gamma [\, G (s,x/n) \,-\, G(s,x-e_j/N)\,]/\sqrt{a_n\, n^ d}$.
Since $e^a \le 1 + a + a^2 e^{|a|}$,
$E_{\nu^n_\rho}[\eta_x - \rho]=0$, and $\log (1+b) \le b$, the
previous expression is bounded by
\begin{equation*}
\frac{3\, \gamma\, \Vert \partial_{x_j} G(s) \Vert^2 _\infty}{a_n\, n^2}
\end{equation*}
provided $\gamma \, \Vert \partial_{x_j} G(s) \Vert_\infty / \sqrt{a_n\, n^
  {d+2}} \le 1$.

To complete the proof of the lemma, it remains to recollect
the previous estimates.
\end{proof}

Assume that $a_n \le \sqrt{\log n}$ and that the sequence of measures
$\mu_n$ on $\Omega_n$ satisfies the bound
$H_n(\mu_n \,|\, \nu^n_\rho) \le R_d(n)$. Fix $\kappa>0$ and $T>0$. By
Remark \ref{rm5}, there exists $n_0$ such that
$H_n(f^n_t) \; ds \le R_d(n) \, n^\kappa$ for all $0\le t\le T$,
$n\ge n_0$.

Choose
$\gamma_n = \sqrt{R_d(n)\, n^{2+\kappa} \, a_n /
\sup_{0\le s\le t} \Vert \partial_{x_j}
  G (s) \Vert_\infty}$. By definition of $R_d(n)$, $\gamma_n$ satisfies
the bound required in the lemma for all $n\ge n_0$. With this choice,
\begin{equation}
\label{17}
\begin{aligned}
& \bb E_{\mu_n} \Big[\,
\int_0^t \frac{1}{\sqrt{a_n\, n^ d}} \,\Big|\, \sum_{x\in \bb T^d_n} G (s,x/n)\,
[\, \eta^n_x(s) - \eta^n_{x+e_j}(s)\,]  \,\Big| \,  \, ds \, \Big] \\
&\qquad \;\le\;
C_0\, t\, \sqrt{\frac{R_d(n)\, n^{\kappa}}{a_n\, n^{2}}}
\, \sup_{0\le s\le t} \Vert \partial_{x_j}  G(s) \Vert_\infty
\end{aligned}
\end{equation}
for every $1\le j\le d$, $0<t\le T$, $n\ge n_0$.  By definition of
$R_d(n)$, this expression vanishes as $n\to\infty$ provided $d\le 3$.

We turn to the proof of Theorem \ref{p01} in the case where
$\Xi_\rho = \Pi^{+2}_\rho$.

\begin{proposition}
\label{l09}
Fix $0<\rho<1$ and a finite subset $A$ of $\bb Z^d$ with at least two
elements.  Then, there exists a finite constant $C_1 = C_1(\rho, A)$,
depending only on the dimension, the set $A$ and the density $\rho$
such that
\begin{align*}
\bb E_{\mu_n} \Big[\, \Big|\,
\int_s^t \, \sum_{x\in \bb T^d_n} G_x (u)\,
\omega_{A+x}(u)  \, du \, \Big| \,\Big] 
\;\le\; C_1\, \Big\{ \frac{1 + H_n(f^n_s)}{a_n} \;+\;
\mf c (G)\, \bb H_n(s,t) \, \Big\} 
\end{align*}
for all $t>s \ge 0$, $G: \bb R \times \bb T^d_n \to \bb R$,
probability measure $\mu_n$ and $n\ge 1$.  On the right-hand side,
$\mf c (G) = 1 + \sup_{s\le u\le t} \Vert G(u) \Vert^2_\infty$,
\begin{equation*}
\bb H_n(s,t) \;=\; \int_s^t H_n(f^n_u )\, du\; +\; (t-s)\,
(n/\ell_n)^d\;,
\end{equation*}
and $(\ell_n : n\ge 1)$ is the sequence introduced in \eqref{19}.
\end{proposition}

In this proposition, replacing the function $G$ by $\gamma G$,
dividing the corresponding estimate by $\gamma$ and optimizing over
$\gamma>0$ yield the next result.

\begin{corollary}
\label{l30}
Fix $0<\rho<1$ and a finite subset $A$ of $\bb Z^d$ with at least two
elements.  Then, there exists a finite constant $C_1 = C_1(\rho, A)$,
depending only on the dimension, the set $A$ and the density $\rho$
such that
\begin{align*}
& \bb E_{\mu_n} \Big[\, \Big|\,
\int_s^t \, \sum_{x\in \bb T^d_n} G_x (u)\,
\omega_{A+x}(u)  \, du \, \Big| \,\Big] \\
&\qquad\quad  \;\le\; C_1\, \sup_{s\le u\le t} \Vert G(u) \Vert_\infty\,
\sqrt{ \Big( \frac{1+ H_n(f^n_s)}{a_n} \;+\;
\bb H_n(s,t)\Big) \, \bb H_n(s,t) } \;, 
\end{align*}
for all $t>s\ge 0$, $G: \bb R_+\times \bb T^d_n \to \bb R$,
probability measure $\mu_n$, $n\ge 1$.
\end{corollary}

The proof of Proposition \ref{l09} is divided in several
steps. Recall, from Lemma \ref{l17}, the decomposition of a finite
subset $A$ of $\bb Z^d$ as $A = A_\star \cup\{\mb x_A\}$.  Recall,
furthermore, the definition of the functions $V$, $V_\ell$ introduced
in \eqref{05}, \eqref{06}, respectively, and the ones of
$H^{(\ell)}_{k,x}(u)$, $H^{(\ell)}_{j,k,x}$ presented in Lemma
\ref{l17} and Corollary \ref{l18}, respectively. Note that
$H^{(\ell)}_{k,x} (u)$ depends on time, because so does $G$. Let
$\Upsilon = \Upsilon^n_{G,A,\ell,\rho} \colon \bb R_+ \times \bb T^d_n
\times \Omega_n \to \bb R$ and
$\Psi = \Psi^n_{G,A,\ell,\rho} \colon \bb R_+ \times \bb T^d_n \times
\Omega_n \to \bb R$ be given by
\begin{equation*}
\Upsilon_x(u,\eta) \; =\; G_x(u) \, \omega_{x+A_\star}\,
\{\omega_{x+\mb x_A} - \omega^\ell_{x+\mb x_A}\,\}\;,
\end{equation*}
\begin{equation*}
\Psi_x(u,\eta) \;: =\;  \frac{a_n}{2\, \mf c_0\, n^2\, \chi(\rho)}
\sum_{k=1}^d  ( H^{(\ell)}_{k,x}(u))^2 
\; +\;  \sum_{j=1}^d \omega_x\, \omega^\ell_{x+e_j}
\;+\; \frac{d\, a_n}{8\, \mf c_0\, n^2\, \chi(\rho)}
\sum_{j,k=1}^d   ( H^{(\ell)}_{j,k,x})^2 \;.
\end{equation*}

In Lemma \ref{l19b} below, we estimate the expectation of
\begin{equation*}
\Big|\;
\int_s^t \sum_{x\in \bb T^d_n} \Upsilon_x (u,\eta^n(u)) \, du \, \Big|
\,-\, \int_s^t \sum_{x\in \bb T^d_n} \Psi_x (u,\eta^n(u)) \, du \;,
\end{equation*}
in Lemma \ref{l15} and equation \eqref{24} the one of
\begin{equation*}
\int_s^t  \Big\{ \,\Big|\, \sum_{x\in \bb T^d_n} G_x (u)\,
\omega_{x+A_\star} (u) \, \omega^\ell_{x+ \mb x_A} (u) \,\Big|
\;+\; \sum_{j=1}^d \sum_{x\in \bb T^d_n}  \omega_x(u) \,
\omega^\ell_{x+e_j} (u) \,\Big\} \; du \,,
\end{equation*}
and in Lemma \ref{l11} and equation \eqref{39} the one of
\begin{equation*}
\int_s^t \sum_{x\in \bb T^d_n} \Big\{ \sum_{j=1}^d 
\big[\, H^{(\ell)}_{k,x} \, (u,\eta^n(u)) \,\big]^2 \,+\,
\sum_{j,k=1}^d  \big[\, H^{(\ell)}_{j,k,x} (\eta^n(u)) \big]^2
\,\Big\} \, du \;.
\end{equation*}
Proposition \ref{l09} follows from these bounds

\begin{lemma}
\label{l19b}
For all $t>s>0$ and  $n\ge 1$,
\begin{equation*}
\bb E_{\mu_n} \Big[\,
\Big|\;
\int_s^t \sum_{x\in \bb T^d_n} \Upsilon_x (u,\eta^n(u)) \, du \, \Big|
\,-\, \int_s^t \sum_{x\in \bb T^d_n} \Psi_x (u,\eta^n(u)) \, du 
\,\Big] \;\le\; \frac{\log 2}{a_n} \;+\; 
\frac{1}{a_n} \, H_n(f^n_s) \;.
\end{equation*}
\end{lemma}

\begin{proof}
Rewrite the expectation on the left-hand side as
\begin{equation*}
\bb E_{\mu_n(s)} \Big[\,
\Big|\;
\int_0^{t-s} \sum_{x\in \bb T^d_n} \Upsilon_x (u,\eta^n(u)) \, du \, \Big|
\,-\, \int_o^{t-s} \sum_{x\in \bb T^d_n} \Psi_x (u,\eta^n(u)) \, du 
\,\Big] \;,
\end{equation*}
where $\mu_n(s) = \mu_n S^n(s)$. By the entropy inequality,
this expression is bounded above by
\begin{equation*}
\begin{aligned}
\frac{1}{\gamma} \, H_n(f^n_s) \;+\; 
\frac{1}{\gamma} \, \log \bb E_{\nu^n_\rho} \Big[\,
\exp \, \gamma \,  \Big\{\, \Big|\,
\int_0^{t_s} & \sum_{x\in \bb T^d_n} \Upsilon_x (s+u,\eta^n(u)) \, du \,
\Big| \\
&
\,-\, \int_0^{t_s} \sum_{x\in \bb T^d_n} \Psi_x (s+u,\eta^n(u)) \, du \,\Big\}
\,\Big] 
\end{aligned}
\end{equation*}
for all $\gamma>0$.  Here, $t_s = t-s$. As $e^{|a|} \le e^a + e^{-a}$
and $e^b + e^{-b}\le 2 \max\{e^b , e^{-b}\}$, by the linearity of the
expectation, the second term of this expression is less than or equal
to
\begin{equation*}
\frac{1}{\gamma} \, \log 2 \;+\; \max_{b=\pm 1}
\frac{1}{\gamma} \, \log \bb E_{\nu^n_\rho} \Big[\,
\exp \,\gamma\, \Big\{\, 
\int_0^{t_s} \sum_{x\in \bb T^d_n} \big\{\,  b\,  \Upsilon_x
(s+u,\eta^n(u))
\,-\, \Psi_x (s+u, \eta^n(u)) \,\big\}\, du \,\Big\} \,\Big]\;.
\end{equation*}

We estimate the second term for $b=1$. The same argument applies to
$b=-1$. By Corollary \ref{l14} below, the second term of this
expression is bounded by
\begin{equation*}
\int_0^{t_s}\, \sup_{f} \Big\{ \int W(s+u)\, f\, d\nu^n_\rho
\;+\; \frac{a_n}{2\,\gamma} \, \int V\, f\, d\nu^n_\rho
\;-\; \frac{n^2}{\gamma} \, I_n(f)\,\Big\}\, du \;,
\end{equation*}
where the supremum is carried over all densities $f$ with respect to
$\nu^n_\rho$, and
$W(s+u)\,=\, \sum_{x\in \bb T^d_n} \{ \Upsilon_x(s+u) - \Psi_x(s+u)\}$.
Choosing $\gamma = a_n$, the previous supremum becomes
\begin{equation}
\label{18}
\frac{1}{a_n} \, \int_0^t\, \sup_{f} \Big\{ \,
a_n \int \big[\, W(s+u) \, +\, (1/2)\, V\, \big]
f\, d\nu^n_\rho \;-\; n^2 \, I_n(f)\,\Big\}\, ds \;.
\end{equation}

By definition of $W(s+u)$, $a_n \, \{ \, W(s+u) \;+\; (1/2) \, V \, \}$ is
equal to
\begin{equation*}
\begin{aligned}
&  a_n \, \sum_{x\in \bb T^d_n} G_x(s+u) \, \omega_{x+A_\star}\,
\{\omega_{x+\mb x_A} - \omega^\ell_{x+\mb x_A}\,\}
\;-\; \frac{a_n^2}{2\, \mf c_0\, n^2\, \chi(\rho)}
\sum_{k=1}^d  \sum_{x\in\bb T_n^d} ( H^{(\ell)}_{k,x}(s+u))^2 \\
& \quad +\; \frac{1}{2} \,\Big\{\, 
a_n  \,\big[ \, V(\eta) \;-\;  V_\ell(\eta) \,\big]
\;-\; \frac{d\, a_n^2}{4\, \mf c_0\, n^2\, \chi(\rho)}
\sum_{j,k=1}^d  \sum_{x\in\bb T_n^d}  ( H^{(\ell)}_{j,k,x})^2\,\Big\}
\;.
\end{aligned}
\end{equation*}
Thus, by Lemma \ref{l17}, with $\beta = 1$ and $b_n = a_n$, and
Corollary \ref{l18}, with $\delta=1$, the expression inside braces in
\eqref{18} is less than or equal to $0$.  This completes the proof of
the lemma.
\end{proof}

\begin{remark}
\label{rm4}
The previous result holds in any dimension and for any sequence
$a_n$. It is a consequence of Feynman-Kac formula and the integration
by parts stated in Lemma \ref{l2}.
\end{remark}

In view of the decomposition carried out just after the statement of
Proposition \ref{l09}, it remains to estimate sums involving the
cylinder functions $\omega_{x+A_\star} \, \omega^\ell_{x+ \mb x_A}$
and sums of squares of averages with respect to the flow $\Phi$. Next
lemma handles the first type of terms.

\begin{lemma}
\label{l15}
Fix $0<\rho<1$ and a finite subset $A$ of $\bb Z^d$ with at least two
elements. Then, there exists a constant $C_1= C_1(A,\rho)$, depending
only on the density $\rho$, the dimension and the set $A$, such that
\begin{align*}
& \bb E_{\mu_n} \Big[\,
\int_s^t \,\Big|\, \sum_{x\in \bb T^d_n} G_x(u) \,
\omega_{x+A_\star} (u) \, \omega^\ell_{x+ \mb x_A} (u)\,   \Big| \,
\, du \, \Big] \\
&\quad \;\le \; C_1(\rho) \, \sup_{s\le u\le t} \Vert G(u) \Vert_\infty\,
\Big\{ \int_s^t H_n(f^n_u )\, du\; +\; (t-s) \, (n/\ell)^d\,\Big\}
\end{align*}
for all functions $G: \bb R_+ \times \bb T^d_n \to \bb R$, $t>s\ge 0$,
measure $\mu_n$ on $\Omega_n$ and $n\ge \ell \ge 1$. 
\end{lemma}

\begin{proof}
Assume, without loss of generality, that $\mb x_A =0$, so that the
origin is a maximal element of $A$, and let $B=A_\star$.  A change of
variables, similar to the one performed in \eqref{16}, yields that
\begin{equation*}
\sum_{x\in \bb T^d_n} G_x(u) \, \omega_{x+B}  \, \omega^\ell_{x}
\;=\;
\sum_{x\in \bb T^d_n}  M^{(1)}_\ell (u,x) \, M^{(2)}_\ell(x) \;,
\end{equation*}
where
\begin{equation*}
M^{(1)}_\ell (u,x) \;=\;  \sum_{y\in \Lambda_\ell} m_\ell(y) \,
G_{x-y}(u) \, \omega_{x-y+B} \;, \quad
M^{(2)}_\ell(x) \;=\; \sum_{z\in \Lambda_\ell} m_\ell(z) \, \omega_{x+z}\;.
\end{equation*}

The expectation appearing in the statement of the lemma is equal to
\begin{align*}
& \int_s^t du\; \int\, \Big|\, 
\sum_{x\in \bb T^d_n}  M^{(1)}_\ell (u,x) \, M^{(2)}_\ell(x)\, \Big|\,
\, f^n_u \, d\nu^n_\rho  \\
& \quad \le\; \frac{1}{2}\,
\sum_{i=1}^2 \gamma^{3-2i}\,
\int_s^t du\; \int \sum_{x\in \bb T^d_n}  M^{(i)}_\ell (u,x)^2 
\, f^n_u \, d\nu^n_\rho 
\end{align*}
for all $\gamma>0$. We applied here Young's inequality
$2ab \le \gamma a^2 + \gamma^{-1} b^2$.

By Lemma \ref{l16} and optimizing over $\gamma>0$, the previous
expression is bounded by
\begin{equation*}
C_1 \, \sup_{s\le u\le t} \Vert G (u) \Vert_\infty\, 
\Big\{ \int_s^t H_n(f^n_u )\, du\; +\; (t-s) \, (n/\ell)^d\,\Big\}
\end{equation*}
for some finite constant $C_1= C_1(A,\rho)$, as claimed. 
\end{proof}

The same argument yields that
\begin{equation}
\label{24}
\begin{aligned}
\bb E_{\mu_n} \Big[\,
\int_s^t 
\,\big|\, V_\ell (\eta^n(u)) \,   \big| \, \, du \, \Big]
\;\le \; C_1(\rho)\,
\Big\{ \int_s^t H_n(f^n_u )\, du\; +\; (t-s)\, (n/\ell)^d\,\Big\}\;.
\end{aligned}
\end{equation}

Recall the definition of $H^{(\ell)}_{k,x}(u,\eta)$, introduced in the
statement of Lemma \ref{l17}.

\begin{lemma}
\label{l11}
Fix $0<\rho<1$, and a finite subset $A$ of $\bb Z^d$ with at least two
elements. Then, there exists a finite constant $C_1=C_1(A,\rho)$,
depending only on the dimension, the density $\rho$ and the number of
elements of the set $A$, such that
\begin{align*}
& \bb E_{\mu_n} \Big[\,
\int_s^t \, 
\sum_{x\in \bb T^d_n} \big[\, H^{(\ell)}_{k,x} \, (u,\eta^n(u)) \,\big]^2
\, du \, \Big]  \\
& \quad \le\;  C_1 \, \sup_{s\le u\le t} \Vert G(u)\Vert^2_\infty\,
\ell^d \, g_d(\ell) \, 
\, \Big\{\, \int_s^t H_n(f^n_u)\; du  \;+\; 
(t-s) \, (n/\ell)^d\,\Big\}
\end{align*}
for all $1\le k\le d$, $t>s\ge 0$, function
$G: \bb R_+\times \bb T^d_n \to \bb R$, probability measure $\mu_n$ on
$\Omega_n$ and $n\ge 1$.
\end{lemma}

\begin{proof}
Let
$U (u,\eta) = \sum_{x\in \bb T^d_n} \big[\, H^{(\ell)}_{k,x} \, (u,\eta)
\,\big]^2$.  By the entropy inequality, the expectation appearing in
the statement of the lemma is bounded by
\begin{equation}
\label{25}
\int_s^t \Big\{ \, \frac{1}{\gamma} \, H_n(f^n_u)  \;+\;
\frac{1}{\gamma} \,
\log \int e^{\gamma\,  U(u)} \; d\nu^n_\rho \,\Big\}\; du
\end{equation}
for every $\gamma>0$.

We turn to the second term. Fix $s\le u\le t$, and recall the
definition of $H^{(\ell)}_{k,x}(u)$ and let $B=A_\star$. We repeat
here the decomposition performed in the proof of Lemma \ref{l16}.  If
$B$ were a singleton, under the measure $\nu^n_\rho$, the variables
$\{\omega_{y+B} : y\in\Lambda_{2\ell-1}\}$ would be independent. Since
this may not be the case, we divide the sum further. Let $p\ge 1$ be
the smallest integer such that $B\subset \Xi_p :=\{-p, \dots, p\}^d$,
and rewrite $H^{(\ell)}_{k,x} = H^{(\ell)}_{k,x} (u)$ as
\begin{equation*}
H^{(\ell)}_{k,x} \;=\; \sum_{z\in \Xi_p} \sum_{w \in \Lambda_\ell(z)}
\Phi_\ell(z + p w \,,\, z + p w +e_k) \, G_s(x- \mb x_A - z - pw)\,
\omega_{x-\mb x_A - z - pw +B} \;,
\end{equation*}
where the second sum is performed over all $w\in \bb Z^d$ such that
$z+pw$ and $z+pw + e_k$ belong to $\Lambda_{2\ell - 1}$. Now, for each
fixed $x$, $z$, the variables
$\{ \omega_{x- \mb x_A - z - pw +B} : w \in \Lambda_\ell(z)\}$ are
independent. Rewrite this sum as
\begin{equation*}
H^{(\ell)}_{k,x} (u) \;=:\; \sum_{z\in \Xi_p} M_\ell(u,z,x)  
\end{equation*}

Applying the arguments presented in the proof of Proposition \ref{l03}
[after equation \eqref{09}] and Lemma \ref{l16} [after equation
\eqref{21}], we obtain that the second term inside braces of
\eqref{25} is bounded by
\begin{equation}
\label{26}
\frac{C_0}{\gamma\, \ell^d} \sum_{x \in \bb T^d_n} \sum_{z\in \Xi_p} \log
\int \exp\Big\{  C_0 \, \gamma \, \ell^d \, M_\ell(u,z,x) ^2 \Big\}
\; d\nu^n_\rho\;,
\end{equation}
for some finite constant $C_0$.

At this point, the proof of the lemma is similar to the one of
Proposition \ref{l03}.  By Lemma \ref{l04}, $M_\ell(u,z,x)$
is a $\sigma^2_\ell$-subgaussian random variable, where
\begin{equation*}
\sigma^2_\ell \;=\; \frac{\Vert G(u)\Vert^2_\infty}{\chi(\rho)^{|B|}}\,
\sum_{\{y,y+e_k\} \subset \Lambda_{2\ell-1}}
\Phi_\ell (y,y+e_k)^2\;.
\end{equation*}
By Lemma \ref{l01},
$\sigma^2_\ell \le \Vert G(u)\Vert^2_\infty\,
g_d(\ell)/\chi(\rho)^{|B|}$. Let $\gamma = \gamma_n$ be given by the
identity
\begin{equation*}
\frac{ C_0 \, \sup_{0\le s\le t} \Vert G(u)\Vert^2_\infty}{\chi(\rho)^{|B|}}
\, \ell^d \, g_d(\ell) \, \gamma \;=\; \frac{1}{4}\;\cdot
\end{equation*}
By \eqref{07}, the sum \eqref{26} is bounded by
\begin{equation*}
\frac{C_0\, \Vert G(u) \Vert^2_\infty}{\chi(\rho)^{|B|}} \,
n^{d}  \, g_d(\ell) \;\le\;
\frac{C_0\, \sup_{s\le u\le t} \Vert G(u) \Vert^2_\infty}{\chi(\rho)^{|B|}} \,
n^{d}  \, g_d(\ell)
\end{equation*}
for some finite constant $C_0$.

Hence, \eqref{25} is less than or equal to
\begin{equation*}
C_1(A,\rho)\, \sup_{s\le u\le t} \Vert G(u) \Vert^2_\infty\,
\ell^d \, g_d(\ell) \, 
\int_s^t H_n(f^n_u)\; du  \;+\;
C_1(A,\rho)\, \sup_{s\le u\le t} \Vert G(u) \Vert^2_\infty\, t \,
n^d  \, g_d(\ell) 
\end{equation*}
for some finite constant $C_1(\rho)$, depending only on the dimension
and on the density $\rho$, as claimed.
\end{proof}

The same arguments exposed in the proof of Lemma \ref{l11} yield that
there exists a finite constant $C_1(\rho)$, depending only on
the dimension, the density $\rho$ and the number of elements of the
set $A$, such that
\begin{equation}
\label{39}
\begin{aligned}
& \bb E_{\mu_n} \Big[\,
\int_s^t 
\sum_{x\in \bb T^d_n} \big[\, H^{(\ell)}_{j,k,x} \, (\eta^n(u)) \,\big]^2
\, du \, \Big]   \\
& \quad \le\;  C_1(\rho)\, 
\ell^d \, g_d(\ell) \, 
\, \Big\{\, \int_s^t H_n(f^n_u)\; du  \;+\; 
(t-s) \, (n/\ell)^d\,\Big\}
\end{aligned}
\end{equation}
for all $1\le j,k\le d$, $t> s\ge 0$, probability measure $\mu_n$ on
$\Omega_n$ and $n\ge 1$. 

\begin{proof}[Proof of Proposition \ref{l09}]
The result follows from the decomposition presented just below the
statement of the proposition, Lemmata \ref{l19b}, \ref{l15},
\ref{l11}, equations \eqref{24}, \eqref{39}, and the definitions of
$g_d(\ell)$, $\ell_n$ given in \eqref{27} and \eqref{19},
respectively. The dependence on $G$ of the constant provided by the
proof is $1 + \mf c + \mf c^2$, where
$\mf c = \sup_{s\le u\le t} \Vert G(u)\Vert_\infty$, and this quantity
is bounded by $2\, (1 + \mf c^2)$.

The term $\ell^d \, g_d(\ell)$ which appears in Lemma \ref{l11} and
equation \eqref{39} cancels with the term $a_n/n^2$ which appears in
the definition of $\Psi_x(s)$ provide we choose $\ell_n$ as in \eqref{19}.
\end{proof}

\begin{proof}[Proof of Theorem \ref{p01}]
In view of the decomposition of $\Pi_\rho f$ presented in Assertion
\ref{l42}, the Boltzmann-Gibbs principle follows from \eqref{17},
Cororollary \ref{l09} with $s=0$, estimate \eqref{01} and the
definition \eqref{19} of the sequence $\ell_n$.
\end{proof}

\section{Tightness}
\label{sec5}

Throughout this section, $0<\rho<1$ and $T>1$ are fixed. Recall that
we denote by $\bb Q_n$ the probability measure on
$D([0,T], \mc H_{-r})$, $r\ge 1$, induced by the process $X^n_t$ and
the initial measure $\mu_n$. The first main result of this section
reads as follows.

\begin{proposition}
\label{l34}
Assume that $d=1$ or $2$, and fix $r> (3d+5)/2$.  Let $\mu_n$ be a
sequence of probability measures on $\Omega_n$ such
$\lim_{n\to \infty} a^{-1}_n H_n(\mu_n \,|\, \nu^n_\rho) =0$.  Then,
the sequence of probability measures $\bb Q_{n}$ is tight. Morover,
every limit point is concentrated on continuous paths.
\end{proposition}

The proof of Proposition \ref{l34} is carried out by showing that each
of the processes on the right-hand side of \eqref{e2} is tight. Note
that the only one which is not continuous in time is the martingale
process $M^n_t$.

Recall the definition of $R^n$ and $B^n$ given below \eqref{42}. The
process $R^n$ can be written as the sum of $R^{\Delta, n}$,
$R^{j,k,n}$ and $\Pi^{1, j,k,n}(F)$, $1\le j,k\le d$, where
$R^{j,k,n}(F)$, $R^{\Delta, n}(F)$, $\Pi^{j,k,n}(F)$,
$F\in C^\infty(\bb T^d)$, are given by
\begin{gather*}
R^{j,k,n}(F) \;=\; \frac{1}{\sqrt{a_n n^d}} 
\sum_{x\in\bb T_n^d}   \{\, h_{j,k} (\tau_x \eta)  \,-\, \tilde
h_{j,k}(\rho)\, \} \, \big\{ \, (\Delta^n_{j,k} F)  -
(\partial^2_{x_j,x_k} F) \,\big\}(x/n) \;,\\
R^{\Delta, n}(F) \;=\;  \frac{a_n}{n^2}\, \frac{1}{\sqrt{a_n n^d}}
\sum_{x\in\bb T_n^d} \{\, \eta_{x} \,-\, \rho\,\}
\, (\Delta_n F)  (x/n)\;,
\end{gather*}
\begin{equation*}
\Pi^{j,k,n}(F) \; =\; \frac{1}{\sqrt{a_n n^d}} 
\sum_{x\in\bb T_n^d}
(\Pi^1_\rho h_{j,k}) (\tau_x\eta)  \, (\partial^2_{x_j,x_k} F) (x/n)\;.
\end{equation*}

On the other hand, the process $B^n$ can be written as the sum of the
components $B^{j,k,n}$, where $B^{j,k,n} (F)$ is given by
\begin{gather*}
B^{j,k, n}(F) \; =\; \frac{1}{\sqrt{a_n n^d}} 
\sum_{x\in\bb T_n^d}
(\Pi^{+2}_\rho h_{j,k}) (\tau_x\eta)  \, (\partial^2_{x_j,x_k} F)
(x/n) \;,
\end{gather*}
for $F\in C^\infty(\bb T^d)$. The projections $\Pi^{1}_\rho$ and
$\Pi^{+2}_\rho$ have been introduced in Assertion \ref{l42}.

Denote by $\mb R^n_t$, $\mb B^n_t$ the $\mc H_{-r}$-valued process
given by
\begin{gather*}
\mb R^n_t(F) \;=\; \int_0^t
\Big\{\, R^{\Delta, n}_s (F)  \,+\, \sum_{j,k=1}^d R^{j,k,n}_s (F)
\,+\, \sum_{j,k=1}^d \Pi^{j,k,n}_s (F) \, \big\}\; ds \; ,\\
\mb B^n_t(F) \;=\; \sum_{j,k=1}^d  \int_0^t B^{j,k,n}_s (F) \; ds \;.
\end{gather*}
Note that $\mb B^n_t$ contains terms of degree two or larger in the
terminology of Section \ref{sec8}.

\subsection{Tightness of the process $\mb R^n_t$}

The process $\mb R^n_t$ has been expressed as the sum of three terms.
Consider the last one and fix $1\le j,k\le d$. By definition of
$\Pi^1_\rho$ and by a summation by parts, 
\begin{equation}
\label{63}
\Pi^{j,k,n}(F) \; =\; \frac{1}{n}  \, \frac{1}{\sqrt{a_n n^d}} 
\sum_{x\in\bb T_n^d} \sum_{z\in A_{j,k}} c_{j,k}(z)
\, (D^n_z \partial^2_{x_j,x_k} F) (x/n) \, [\eta_x - \rho] 
\end{equation}
for some finite subset $A_{j,k}$ of $\bb Z^d$ and real numbers
$c_{j,k}(z)$, $z\in A_{j,k}$, which depend only on the cylinder
function $h_{j,k}$. In this formula, for a function $J$ in
$C(\bb T^d)$,
\begin{equation*}
(D^n_z J) (x/n) \;=\; n\, \big[ \, J([x-z]/n) \,-\, J(x/n) \, \big]\;. 
\end{equation*}
Note the additional factor $1/n$ in \eqref{63} which appeared from the
summation by parts.

For $k\ge 0$ and a function $F$ in $C^\infty(\bb T^d)$, denote by
$\Vert F \Vert_{C^k(\bb T^d)}$ the $C^k(\bb T^d)$-norm of $F$:
\begin{equation*}
\Vert F \Vert_{C^k(\bb T^d)} \;=\; \sum_{|i|\le k} \Vert D_i F
\Vert_\infty\;, 
\end{equation*}
where the sum is carried over multi-indices
$i = (i_1, \dots, i_d) \in\bb N^d$ and $D_iF = \partial^{i_1}_{x_1}
\cdots \partial^{i_d}_{x_d} F$. 

By the explicit formula for $R^n$, \eqref{63} and Lemma \ref{l41b},
there exists a finite constant $C_0$, depending only on the cylinder
functions $h_{j,k}$ such that
\begin{equation}
\label{64}
\bb E_{\mu_n} \Big[ \, \int_0^t\, 
R^{n}_s (G_s)^2 \; ds \, \Big] \;\le\; 
\frac{C_0}{a_n n^2} \, \Big\{  \int_0^t H_n(f^n_s) \, ds
\;+\; \, t\, \Big\} \, \sup_{0\le s\le t} \Vert G_s\Vert^2_{C^3(\bb T^d)}\;,
\end{equation}
for all $0<t\le T$, smooth function
$G:[0,T] \times \bb T^d \to \bb R$, probability measure $\mu_n$ and
$n\ge 1$.

In the next lemma, we assume that $H_n(\mu_n \,|\, \nu^n_\rho) \le a_n
+ R_d(n)$. In dimension $d\ge 2$, the term $a_n$ is insignificant
being much smaller than $R_d(n)$, but in dimension $1$ it is much
larger than $R_d(n)$.

\begin{lemma}
\label{l26}
Assume that $d\le 3$.  Fix $r>3 + (d/2)$ and a sequence of measures
$\mu_n$ such that $H_n(\mu_n \,|\, \nu^n_\rho) \le a_n + R_d(n)$ for
all $n\ge 1$. Then,
\begin{equation*}
\lim_{n\to\infty} \bb E_{\mu_n} \Big[ \sup_{0\le t\le T} \Vert \mb
R^n_t \Vert^2_{-r}\, \Big] \;=\; 0\;.
\end{equation*}
\end{lemma}

\begin{proof}
By definition of the norm $\Vert \,\cdot\, \Vert_{-r}$, the
expectation appearing in the statement of the lemma is bounded by
\begin{equation*}
\sum_{m\in \bb Z^d} \gamma^{-r}_m\,
\bb E_{\mu_n} \Big[ \sup_{0\le t\le T} \Vert \mb
R^{n}_t (\phi_m) \Vert^2 \, \Big]\;.
\end{equation*}
By Schwarz inequality, this sum is bounded by
\begin{equation*}
T\, \sum_{m\in \bb Z^d} \gamma^{-r}_m\,
\bb E_{\mu_n} \Big[ \int_0^T\, \Vert\,
R^{n}_s (\phi_m)\,\Vert^2 \; ds \, \Big]\;.
\end{equation*}
By \eqref{64}, this expression is less than or equal to
\begin{equation*}
\frac{C_0\, T}{a_n\, n^2} \, 
\Big\{\int_0^T H_n(f^n_s) \, ds \;+\; T  \, \Big\}\,
\sum_{m\in \bb Z^d} \gamma^{-r}_m\,  \Vert m\Vert^6 \;.
\end{equation*}
To complete the proof, it remains to recall the estimate \eqref{01}.
\end{proof}

\subsection{H\"older spaces and Kolmogorov-\v{C}entsov
theorem}

The proof of the tightness of the process $\mb B^n_t$ is based on the
Kolmogorov-\v{C}entsov theorem stated below for convenience. We refer
to \cite[Problem 4.11]{ks91} or \cite[Proposition 7]{RacSuq01} for a
proof.

Fix $T_0>0$ and $r>0$. Denote by $C^\vartheta([0,T_0], \mc H_{-r})$,
$0<\vartheta<1$, the H\"older space of $\mc H_{-r}$-valued, continuous
functions endowed with the norm
\begin{equation*}
\Vert X \Vert_{C^\vartheta} \;=\;
\sup_{0\le t \le T_0} \Vert X_t \Vert_{-r} \;+\; 
\sup_{0\le s <t \le T_0}
\frac{\Vert X_t - X_s\Vert_{-r}}{|t-s|^\vartheta}\;\cdot
\end{equation*}
Note that this topology is stronger than the uniform topology of
$C([0,T_0], \mc H_{-r})$.

\begin{theorem}
\label{l46}
Fix $T_0>0$, $0<\vartheta<1$ and $r>0$. A sequence of probability measures
$\bb M_n$ on $C^\vartheta([0,T_0], \mc H_{-r})$ is tight if
\begin{enumerate}
\item[(i)] There exist a constant $a>0$ such that
$\sup_{n\ge 1} \bb M_{n} \big[ \, \Vert \, X_0 \,
\Vert^a_{-r} \,\big] \;<\, \infty$;
\item[(ii)] There exist a finite constant $C_0$ and positive constants
$b>0$, $c>0$ such that $c/b > \vartheta$ and 
$\sup_{n\ge 1} \bb M_n \big[ \, \Vert \, X_t \,-\,
X_s \, \Vert^b_{-r} \,\big] \;\le\, C_0 \, |t-s|^{1+c}$ for
all $0\le s, t\le T_0$.
\end{enumerate}
\end{theorem}

We used here the same notation $\bb M_n$ to represent a probability on
$C^\vartheta([0,T_0], \mc H_{-r})$ and the expectation with respect to
this probability.

\subsection{Tightness of the process $\mb B^n_t$ for $d=1$, $2$}
\label{ss5.2}

Recall the definition of the $\mc H_{-r}$-valued process $\mb B^n_t$
introduced at the beginning of this section. Note that this process is
continuous in time.

Fix $1\le j,k\le d$.  By definition of the operator $\Pi^{+2}_\rho$,
there exists finite collection $\mc E_{j,k}$ of subsets of $\bb Z^d$
with at least two elements such that
\begin{gather*}
(\Pi^{+2}_\rho h_{j,k})(\eta)
\;=\; \sum_{D \in \mc E_{j,k}} c_{j,k,D}(\rho) \,\omega_D \;.
\end{gather*}

In particular, to prove the tightness of the process $\mb B^n_t$, it
is enough to prove this property for the processes
$\mb B^{i,j,A,n}_t$, $A$ a finite subset of $\bb Z^d$ with at least
two elements, where
\begin{equation*}
\mb B^{i,j,A,n}_t (F) \;:=\; \int_0^t B^{A,n}_s (\partial^2_{x_i,x_j} F) \; ds
\end{equation*}
and 
\begin{equation}
\label{65}
B^{A,n}_t (G) \;:=\; \frac{1}{\sqrt{a_n n^d}}\,
\sum_{x\in \bb T^d_n} G_x  \, \omega_{x+A}(t) 
\end{equation}
for $G:\bb T^d_n \to \bb R$.

Fix $1\le i,j\le d$ until the end of this subsection. We omit these
indices from the notation hereafter and represent $\mb B^{i,j,A,n}_t$,
$B^{i,j,A,n}_t$ simply as $\mb B^{A,n}_t$, $B^{A,n}_t$, respectively.
The tightness of the process $\mb B^{A,n}_t$ relies on the next
estimate.

\begin{lemma}
\label{l28}
Fix a finite subset $A$ of $\bb Z^d$. Then, there exists a finite
constant $C_0 = C_0(\rho, A)$ such that
\begin{equation*}
\bb E_{\mu_n} \Big[ \, \Big|\, \sum_{x\in \bb T^d_n} F_x \,
\omega_{x+A}(t) \,\Big|^\alpha\,\Big] \;\le\; C_0 \,
n^{\alpha d/2} \, \Vert F\Vert^\alpha _\infty \, \big( \, 1 \,+\,
H_n(f^n_t)\,\big)^{\alpha /2} 
\end{equation*}
for all $1\le \alpha\le 2$, $t>0$, $F:\bb T^d_n \to\bb R$, probability
measure $\mu_n$ and $n\ge 1$. Here, $f^n_t$ stands for density of
$\mu_n S^n(t)$ with respect to a Bernoulli measure $\nu^n_\rho$.
\end{lemma}

\begin{proof}
Fix $1\le \alpha \le 2$.  By H\"older inequality, the expectation is
bounded by
\begin{equation*}
\bb E_{\mu_n} \Big[ \, \Big\{\, \sum_{x\in \bb T^d_n} F_x \,
\omega_{x+A}(t) \,\Big\}^2\,\Big]^{\alpha /2}\;.
\end{equation*}
By the entropy inequality, this expectation is less than or equal to
\begin{equation}
\label{45}
\frac{1}{\gamma}\, H_n(f^n_t) \;+\;
\frac{1}{\gamma} \log E_{\nu^n_\rho} \Big[ \exp \gamma \,
\Big\{ \sum_{x\in \bb T^d_n} F_x \, \omega_{x+A} \, \Big\}^2\, \Big]
\end{equation}
for all $\gamma>0$.

Consider the second term. The function $\omega_{A}$ can be written as
$\sum_{1\le j\le p} C_j (\rho) \, \{ \, f_j(\eta) - \tilde
f_j(\rho)\,\}$, where each $f_j$ is a cylinder function. Hence, by
Schwarz inequality [to move the sum over $j$ out of the square],
H\"older's inequality [to move the sum over $j$ out of the
exponential], and Corollary \ref{l27}, the second term of the previous
displayed equation is bounded above by
\begin{equation*}
C_0(\rho, A) \, \Vert F\Vert^2_\infty\, n^d
\end{equation*}
provided $\gamma n^d \Vert F\Vert^2_\infty < c_0(\rho, A)$. Here
$0<c_0 < C_0 < \infty$ are constants which depend on $A$ and $\rho$
only. Therefore, setting $\gamma n^d \Vert F\Vert^2_\infty = c_0/2$
yields that \eqref{45} is bounded by
\begin{equation*}
C_0 \, n^d\, \Vert F\Vert^2_\infty \, \Big\{ 1\;+\; H_n(f^n_t)
\,\Big\}\;, 
\end{equation*}
which completes the proof of the lemma because $\alpha \le 2$ [so that
$C_0^{\alpha/2} \le C_0$].
\end{proof}

Next result is a simple consequence of Lemma \ref{l28} and H\"older's
inequality. 

\begin{corollary}
\label{l29}
Fix a finite subset $A$ of $\bb Z^d$. Then, there exists a finite
constant $C_0 = C_0(\rho, A)$ such that
\begin{align*}
& \bb E_{\mu_n} \Big[ \, \Big|\, \int_s^t \sum_{x\in \bb T^d_n} F_x(u) \,
\omega_{x+A}(u) \; du \,\Big|^\alpha\,\Big] \\
&\qquad \;\le\; C_0 \,
n^{\alpha d/2} \, |t-s|^\alpha \,
\sup_{s\le u\le t} \Vert F(u) \Vert^\alpha _\infty \,
\sup_{s\le u\le t} \big( \, 1 \,+\,
H_n(f^n_u)\,\big)^{\alpha /2} 
\end{align*}
for all $1\le \alpha\le 2$, $0<s<t$, $F:\bb T^d_n \to\bb R$,
probability measure $\mu_n$ and $n\ge 1$.
\end{corollary}

The proof of the tightness of the process $\mb B^{A,n}_t$ relies on
the following estimate. Recall the definition of the process
$B^{A,n}_t$ introduced in \eqref{65}. 

\begin{lemma}
\label{l44}
Fix $T>0$, a finite subset $A$ of $\bb Z^d$ with at least two
elements and $\delta>0$.  Fix also $0<\theta<1$ and $1<\alpha \le
2$. Let $(\mu_n : n\ge 1)$ be a sequence of probability measures such
that $H_n(\mu_n \,|\, \nu^n_\rho) \le a_n + R_d(n)$. Then, there exist
a finite constant $C_1 = C_1(\rho, A, T)$ and an integer
$n_0=n_0(\delta, T)$ such that
\begin{equation*}
\bb P_{\mu_n}
\Big[ \, \Big | \int_s^t \, B^{A,n}_u (F(u)) \; du \, \Big |
\,>\, \lambda \, \Big] \;\le\;
C_1 \, \frac{1}{\lambda^{\beta}}  \; \Omega^n_{s,t}\, 
\sup_{0\le u\le T} \Vert F(u) \Vert^{\beta}_\infty
\end{equation*}
for all $\lambda>0$, $0\le s<t\le T$, $n\ge n(\delta,T)$ and
$F:[0,T]\times \bb T^d_n \to \bb R$. In this formula,
\begin{equation*}
\Omega^n_{s,t} \;=\; \frac{n^{\delta \gamma} \, \{a_n + R_d(n)\}^{\gamma}}
{(a_n\, n^d)^{\theta/2}}\,
|\,t-s\,|^{(\theta/2) + \alpha (1-\theta)} 
\end{equation*}
and $\beta = \theta + (1-\theta) \alpha$,
$\gamma = \theta + \alpha (1-\theta) /2$ and $R_d(n)$ is the sequence
introduced in Theorem \ref{t2}.
\end{lemma}

\begin{proof}
Write the probability appearing in the statement of the lemma, denoted
by $p$, as $p^\theta \, p^{1-\theta}$. Apply Chebyshev inequality to
both terms to obtain that the previous integral is less than or equal
to
\begin{equation}
\label{44c}
\begin{aligned}
\frac{1}{\lambda^{\beta}} \,
\bb E_{\mu_n} \Big[ \,
\Big | \int_s^t \, B^{A,n}_u (F(u)) \; du \, \Big |
\, \Big]^\theta \;
\bb E_{\mu_n} \Big[ \,
\Big | \int_s^t \, B^{A,n}_u (F(u)) \; du \, \Big |^\alpha
\, \Big]^{1-\theta} \;,
\end{aligned}
\end{equation}
where $\beta = [\, \theta + (1-\theta) \alpha]$.

By Corollary \ref{l30}, the first expectation is bounded by
\begin{equation*}
C_1(A, \rho)  \, \frac{1}{(a_n\, n^d)^{\theta/2}}\, \sup_{s\le u\le t}
\Vert F(u) \Vert^\theta_\infty \, 
\Big\{ \Big( \frac{1 + H_n(f^n_s)}{a_n} \;+\;  \bb H_n(s,t)\Big) \,
\bb H_n(s,t)  \Big\}^{\theta/2}\;,
\end{equation*}
where $\bb H_n(s,t) = \int_s^t H_n(f^n_r)\, dr + (t-s) (n/\ell_n)^d$.
Fix $\delta>0$. Since $a_n \le \sqrt{\log n}$, there exists
$n(\delta, T) \ge 1$ such that, by \eqref{01} and \eqref{19},
$\bb H_n(s,t) \le n^\delta\, (t-s) \, \{ a_n + R_d(n)\}$ for all
$n\ge n(\delta, T)$, $0\le s \le t\le T$. We estimate
$[1 + H_n(f^n_s)]/a_n + \bb H_n(s,t)$ by
$3 n^\delta \, (1+T) \, \{ a_n + R_d(n)\}$ and the second
$\bb H_n(s,t)$ by $n^\delta\, (t-s) \, \{ a_n + R_d(n)\}$. Hence, the
expression appearing in the previous displayed equation is bounded by
\begin{equation*}
C_1(A, \rho) \,(1+T) \, \sup_{0\le u\le t}
\Vert F(u) \Vert^\theta_\infty \,
\frac{n^{\delta \theta} \,
\{a_n + R_d(n)\}^{\theta}}{(a_n\, n^d)^{\theta/2}} \, |\,t-s\,|^{\theta/2} 
\end{equation*}
for all $0\le s, t\le T$, $n\ge n(\delta,T)$. We estimated
$(1+T)^{\theta/2}$ by $1+T$, as $\theta<1$.

By Corollary \ref{l29}, the second expectation in \eqref{44c} is
bounded by 
\begin{equation*}
C_1(\rho, A)  \, |\,t-s\,|^{\alpha (1-\theta)} \,
\sup_{0\le u\le t} \Vert F(u) \Vert^{\alpha   (1-\theta)}_\infty\, 
\sup_{0\le u\le T} \big( \, 1 \,+\,
H_n(f^n_u)\,\big)^{\alpha (1-\theta) /2} \;.
\end{equation*}
Thus, by the bound on $H_n(f^n_u)$ presented in the previous
paragraph, this expression is bounded by
\begin{equation*}
C_1(\rho, A, T) \, |\,t-s\,|^{\alpha (1-\theta)} \,
\sup_{0\le u\le t} \Vert F(u) \Vert^{\alpha (1-\theta)}_\infty
\, n^{\delta \alpha  (1-\theta)/2} \,
\{a_n + R_d(n)\}^{\alpha (1-\theta) /2}
\end{equation*}
for all $0\le s, t\le T$, $n\ge n(\delta,T)$.

Putting together the previous estimates yields that \eqref{44c} is
bounded above by
\begin{equation*}
C_1(\rho, A, T) \,
\frac{1}{\lambda^{\beta}}  \;
\frac{n^{\delta \gamma} \, \{a_n + R_d(n)\}^{\gamma}}
{(a_n\, n^d)^{\theta/2}}\,
|\,t-s\,|^{(\theta/2) + \alpha (1-\theta)} \, 
\sup_{0\le u\le t} \Vert F(u) \Vert^{\beta}_\infty
\end{equation*}
for all $n\ge n(\delta,T)$, where
$\gamma = \theta + \alpha (1-\theta) /2$ and, recall,
$\beta = \theta + \alpha (1-\theta)$. This is the assertion of the
lemma.
\end{proof}

The previous bound provides two estimates. The first one will be used
later in the proof of the tightness of the process $X^n_t$ in
dimension $d=3$. The second one is needed in the proof of the
tightness of the process $\mb B^n_t$ in dimensions $d=1$, $2$.
Recall the definition of $\Omega^n_{s,t}$ introduced in the previous
lemma. 

\begin{corollary}
\label{l43}
Fix a finite subset $A$ of $\bb Z^d$ with at least two elements and
$\delta>0$.  Fix also $\upsilon>1$, $0<\theta<1$ and $1<\alpha \le 2$
such that
\begin{equation*}
\beta\;:=\; \theta \,+\, \alpha \, (1-\theta) \;>\; \upsilon\;.
\end{equation*}
Let $(\mu_n : n\ge 1)$ be a sequence of probability measures such that
$H_n(\mu_n \,|\, \nu^n_\rho) \le a_n + R_d(n)$.  Then, there exist a
finite constant $C_1 = C_1(\rho, A, T)$ and an integer
$n_0=n_0(\delta, T)$ such that
\begin{equation*}
\bb E_{\mu_n} \Big[ \, \Big | \int_s^t
\, B^{A,n}_u (F(u)) \; du \, \Big |^\upsilon \, \Big]^{\beta/\upsilon} \;\le \;
C_1\, \Big(\frac{\beta}{\upsilon}\Big)^{\beta/\upsilon} \, \Omega^n_{s,t} \, 
\sup_{0\le u\le T} \Vert F(u) \Vert^{\beta}_\infty
\end{equation*}
for all $0\le s<t\le T$, $n\ge n(\delta,T)$ and
$F:[0,T]\times \bb T^d_n \to \bb R$. 
\end{corollary}

\begin{proof}
Fix $0\le s \le t \le T$, and write
\begin{equation*}
\bb E_{\mu_n} \Big[ \, \Big | \int_s^t
\, B^{A,n}_u (F(u)) \; du \, \Big |^\upsilon \, \Big] \;=\;
\int_0^\infty \bb P_{\mu_n}
\Big[ \, \Big | \int_s^t \, B^{A,n}_u (F(u)) \; du \, \Big |
\,>\, \kappa^{1/\upsilon} \, \Big] \, d\kappa\;.
\end{equation*}
Fix $\mf a>0$. We estimate the above probability by $1$ on the
interval $0\le \kappa \le \mf a$. On the other hand, by Lemma
\ref{l44} with $\lambda = \kappa^{1/\upsilon}$, the previous integral
restricted to the interval $[\mf a,\infty)$, is bounded above by
\begin{equation*}
C_1(\rho, A, T) \, \Omega^n_{s,t}\, 
\sup_{0\le u\le t} \Vert F(u) \Vert^{\beta}_\infty \,
\int_{\mf a}^\infty \frac{1}{\kappa^{\beta/\upsilon}} \; d\kappa 
\end{equation*}
for all $n\ge n(\delta,T)$, where, recall,
$\beta = \theta + \alpha (1-\theta)$. As $\beta > \upsilon$, this
expression is equal to
$C_1\, (b-1) \, \mf a^{1-b} \, \Omega^n_{s,t} \, \sup_{0\le u\le t}
\Vert F(u) \Vert^{\beta}_\infty$, where $b=\beta/\upsilon>1$.

Up to this point, we proved that the expectation appearing in the
statement of the corollary is bounded by
$\mf a + C_1\, (b-1) \, \mf a^{1-b} \, \Omega^n_{s,t} \, \sup_{0\le
  u\le t} \Vert F(u) \Vert^{\beta}_\infty$ for all $\mf a>0$. It
remains to optimize over $\mf a$ to complete the proof.
\end{proof}

Recall that $1\le i,j \le d$ are fixed and that $\mb B^{A,n}_t$ stand
for $\mb B^{i,j,A,n}_t$. 

\begin{corollary}
\label{l31}
Fix a finite subset $A$ of $\bb Z^d$ with at least two elements and
$\delta>0$. Fix also $\upsilon>1$, $0<\theta<1$ and $1<\alpha \le 2$
such that
\begin{equation*}
\beta\;:=\; \theta \,+\, \alpha \, (1-\theta) \;>\; \upsilon\;.
\end{equation*}
Finally, fix $r>0$ such that
\begin{equation*}
r\;>\; \frac{d+5}{2} \;+\; \frac{d} {\beta}\;\cdot
\end{equation*}
Let $(\mu_n : n\ge 1)$ be a sequence of probability measures such that
$H_n(\mu_n \,|\, \nu^n_\rho) \le a_n + R_d(n)$.  Then, there exist a
finite constant $C_1 = C_1(\rho, A, T, \upsilon, \theta, \alpha, r)$
and an integer $n_0=n_0(\delta, T)$ such that
\begin{equation*}
\bb E_{\mu_n} \Big[ \, \Vert \, \mb B^{A,n}_t \,-\, \mb B^{A,n}_s \,
\Vert^\upsilon_{-r} \, \Big]^{\beta/\upsilon} \;\le\; C_1\, \Omega_{n,s,t} \,
\end{equation*}
for all $0\le s < t \le T$, $n\ge n_0$.
\end{corollary}

\begin{proof}
Fix $0\le s \le t \le T$, and write
\begin{equation*}
\bb E_{\mu_n} \Big[ \, \Vert \, \mb B^{A,n}_t \,-\, \mb B^{A,n}_s \,
\Vert^\upsilon_{-r} \, \Big] \;=\;
\int_0^\infty \bb P_{\mu_n} \Big[ \, \Vert \, \mb B^{A,n}_t \,-\, \mb B^{A,n}_s \,
\Vert_{-r} \,>\, \kappa^{1/\upsilon} \, \Big] \, d\kappa\;.
\end{equation*}
Fix $\mf a>0$ and estimate the above probability by $1$ on the
interval $[0,\mf a]$. 

We turn to the interval $[\mf a, \infty)$.  Let
$\{\mf w_m : m\in \bb Z^d\}$ be the sequence given by
$\color{bblue} \mf w_m = \mf w \, (1 + \Vert m\Vert)^{-(d+1)}$, where
$\mf w$ is the normalizing constant chosen so that
\begin{equation}
\label{66}
\sum_{m\in \bb Z^d} \mf w_m \;=\; 1\;.
\end{equation}
The probability on the right-hand side of the first displayed equation
of the proof is equal to
\begin{align*}
& \bb P_{\mu_n} \Big[ \, \sum_{m\in \bb Z^d} \gamma^{-r}_m\,
\Vert \, \mb B^{A,n}_t(\phi_m) \,-\, \mb B^{A,n}_s (\phi_m) \,
\Vert^2 \,>\, \kappa^{2/\upsilon} \, \Big] \\
& \qquad \le\; \sum_{m\in \bb Z^d} 
\bb P_{\mu_n} \Big[ \,
\Vert \, \mb B^{A,n}_t(\phi_m) \,-\, \mb B^{A,n}_s (\phi_m) \,
\Vert \,>\, \kappa^{1/\upsilon} \,
\{\mf w_m \, \gamma^{r}_m\}^{1/2} \; \Big] \;.
\end{align*}

By Lemma \ref{l44} with
$\lambda = \kappa^{1/\upsilon} \, \{\mf w_m \, \gamma^{r}_m\}^{1/2}$,
$F(u) = \partial^2_{x_i, x_j} \phi_m$, since
$\Vert \partial^2_{x_i, x_j} \phi_m \Vert_\infty \le C_0 \, \Vert
m\Vert^2$, the previous expression is bounded above by
\begin{equation*}
C_1(\rho, A, T) \, \Omega_{n,s,t} \,
\frac{1}{\kappa^{\beta/\upsilon}}\, \sum_{m\in \bb Z^d} 
\frac{1}{(\mf w_m \, \gamma^{r}_m)^{\beta/2}} \,
\,\Vert m \Vert^{2\beta}
\end{equation*}
for all $n\ge n(\delta,T)$. As $r> (d+5)/2 + (d/\beta)$, the sum over
$m$ is finite. 

Up to this point, we proved that the expectation appearing in the
statement of the corollary is bounded by
$\mf a + C_2(b-1) \mf a^{1-b} \Omega_{n,s,t}$ for all $\mf a>0$. Here,
$b= \beta/\upsilon>1$ and $C_2$ is a constant which depends on $\rho$,
$A$, $T$ and also on $\theta$, $\alpha$, $\upsilon$, $r$ [through
$\beta$].  It remains to optimize over $\mf a$ to complete the proof.
\end{proof}

Fix a finite subset $A$ of $\bb Z^d$ with at least two elements and
$r>0$.  Let $(\mu_n : n\ge 1)$ be a sequence of probability measures
such that $H_n(\mu_n \,|\, \nu^n_\rho) \le a_n + R_d(n)$.  Denote by
$\bb Q^{A,\mb B}_n$ the measure on $D([0,T], \mc H_{-r})$ induced by
the process $\mb B^{A,n}_t$ and the measure $\bb P^n_{\mu_n}$.


\begin{corollary}
\label{l32}
In dimension $1$ and $2$, the sequence of probability measures
$\bb Q^{A,\mb B}_n$ is tight in $C([0,T], \mc H_{-r})$ for
$r>(3d+5)/2$.
\end{corollary}

\begin{proof}
The proof relies on Theorem \ref{l46}.  The first condition of this
theorem is satisfied because $\mb B^{A,n}_0=0$.

We turn to condition (ii). In Corollary \ref{l31}, set $\alpha =2$, fix
$\theta>0$ small, let $\upsilon = 2(1-\theta)$, $\delta =
d\theta/4$. With these choices, $\beta = 2-\theta > \upsilon$,
$\gamma=1$, and, taking $r>(3d+5)/2$,
\begin{equation*}
\bb E_{\mu_n} \Big[ \, \Vert \, \mb B^{A,n}_t \,-\, \mb B^{A,n}_s \,
\Vert^{2(1-\theta)}_{-r} \, \Big] \;\le\; C_1\,
\Big( \frac{a_n + R_d(n)}{n^{d\theta/4}} \Big)^{2(1-\theta)/(2-\theta)}
\, |\,t-s\,|^{1+c(\theta)} 
\end{equation*}
for all $0\le s < t \le T$, $n\ge n_0$. Here,
$c(\theta)= [2-6\theta+3\theta^2]/(2-\theta) >0$. By definition of the
sequences $a_n$, $R_d(n)$, in dimensions $1$ and $2$,
$[a_n + R_d(n)]/n^{d\theta/4} \to 0$.  The ratio is therefore
bounded. Thus, condition (ii) holds [for $n\ge n_0$] with
$b=2(1-\theta)>0$ and $c= c(\theta)>0$.
\end{proof}

Note that we actually proved that the sequence $\bb Q^{A,\mb B}_n$ is
tight in $C^\vartheta([0,T], \mc H_{-r})$ for $\vartheta <
c(\theta)/[2(1-\theta)] = (1/2) + O(\theta)$. 

\begin{remark}
\label{l33}
Unfortunately, the previous argument does not apply to the dimension
$d=3$. The reason is the term $1/a_n$ in the estimate stated in
Proposition \ref{l09}, which does not depend on time. In particular,
for $|t-s|$ very small, the estimate is bad.

This term $1/a_n$, independent of time, gives rise to the power $1/2$
in the expression $|t-s|^{\theta/2}$ when we estimate the first
expectation in equation \eqref{44c}. In Proposition \ref{l09}, if, we
had, instead of $1/a_n$, a term which decreases linearly in time, as
$t\to 0$, we would get $|t-s|^{\theta}$ in equation \eqref{44c},
instead of $|t-s|^{\theta/2}$, and the proof of the tightness could be
carried out in dimension $3$ taking $\theta = 2/3 + \epsilon$ in
Corollary \ref{l31}.
\end{remark}

\subsection{Tightness of the martingale $M^n_t$}
\label{ssec5.3}

By \cite[Lemma A.5.1]{kl},
\begin{equation*}
M^n_t(F)^2 \;-\; \int_0^t \big\{
L_n X^n_s(F)^2 \,-\, 2\, X^n_s(F)  \, L_n X^n_s(F) \,\big\}\; ds
\end{equation*}
is a martingale. Denote by $\Gamma^n(F)$ the expression inside braces.
A straightforward computation yields that
\begin{equation}
\label{47}
\begin{aligned}
\Gamma^n(F) \;&=\; \frac{1}{a_n\, n^d} \sum_{j=1}^d \sum_{x\in \bb T^d_n}
c_j(\tau_x \eta) \, [\eta_{x+e_j} - \eta_x]^2\,
[(\Delta_{n,j} F)(x/n)]^2 \\
& +\; \frac{1}{n^d} \sum_{x\in \bb T^d_n} \sum_y
F(x/n)^2 \, [\eta_{y} - \eta_x]^2 \;,
\end{aligned}
\end{equation}
where $\nabla_{n,j} F$ stands for the discrete partial derivative
given by
$\color{bblue} (\nabla_{n,j} F)(x) = n[F((x+e_j)/n) - F(x/n)]$.  Thus,
$|\,\Gamma^n(F)\,|$ is bounded by
$C_0 \,\{ \Vert F \Vert^2_\infty + \Vert \nabla F \Vert^2_\infty\,\}$
for some finite constant $C_0$, and
\begin{equation}
\label{35}
\bb E^n_{\eta} \big[\, M^n_t(F)^2\,\big] \;\le\; C_0\, t\, 
\,\{ \, \Vert F \Vert^2_\infty + \Vert \nabla F \Vert^2_\infty\,\}
\end{equation}
for all $t>0$.

\begin{lemma}
\label{l22}
Fix $d\ge 1$ and $r>1+(d/2)$. Then, there exists and finite constant
$C_0$ such that
\begin{equation*}
\lim_{n\to \infty} \, \sup_{\eta\in \Omega_n}\,
\bb E^n_{\eta} \big[\, \sup_{0\le t\le T}
\Vert\, M^n_{t} \,\Vert^2_{-r} \,\big] \;\le\; C_0\, T \;.
\end{equation*}
Moreover,
\begin{equation*}
\limsup_{p\to \infty} \limsup_{n\to \infty} \,
\sup_{\eta\in \Omega_n}\,
\bb E^n_{\eta} \big[\, \sup_{0\le t\le T}
\sum_{\Vert m\Vert \ge p} \gamma_m^{-r}\;
\Vert\, M^n_{t} (\phi_m) \,\Vert^2 \,\big] \;=\; 0\;.
\end{equation*}
\end{lemma}

\begin{proof}
By the formula for the $\mc H_{-r}$-norm and since
$\sup_t \sum_m a_m(t) \le \sum_m \sup_t a_m(t)$, the first expectation
is bounded by
\begin{equation*}
\sum_{m\in \bb Z^d}
\gamma_m^{-r}\;
\bb E^n_{\eta} \big[\, \sup_{0\le t\le T}  \,
\Vert\, M^n_{t} (\phi_m)\,\Vert^2 \,\big]\;.
\end{equation*}
Since $M^n_{t}(\phi_m)$ is a martingale for each $m\in \bb Z^d$, by
Doob's inequality, this sum is bounded by
\begin{equation*}
4\, \sup_{\eta\in \Omega_n} \sum_{m\in \bb Z^d}
\gamma_m^{-r}  \, 
\bb E^n_{\eta} \big[\,  \Vert\, M^n_{T}(\phi_m) \,\Vert^2 \,\big]\;.
\end{equation*}
By \eqref{35} and by definition of $\phi_m$, this expression is less
than or equal to
\begin{equation*}
C_0\, T\, \sum_{m\in \bb Z^d}
\gamma_m^{-r}  \, ( 1 + \Vert m \Vert^2)\;.
\end{equation*}
This proves the first assertion of the lemma since $r>1+(d/2)$. The
proof of the second one is similar.
\end{proof}

Fix $r>0$, a sequence of probability measures $\mu_n$ on
$\Omega_n$, and denote by $\bb Q^M_n$ the measure on
$D([0,T], \mc H_{-r})$ induced by the martingale $M^n_t$ and the
measure $\bb P^n_{\mu_n}$.

\begin{lemma}
\label{l23}
Fix $d\ge 1$ and $r> 1+(d/2)$. The sequence of measures $\bb Q^M_n$ is
tight in $D([0,T], \mc H_{-r})$. Moreover, all limit points are
concentrated on continuous trajectories.
\end{lemma}

\begin{proof}
According to \cite[Lemma 11.3.2]{kl}, we have to show that
\begin{equation*}
\lim_{A\to \infty} \, \limsup_{n\to \infty}\,
\bb P^n_{\mu_n} \big[\, \sup_{0\le t\le T}
\Vert\, M^n_{t} \,\Vert_{-r'} \, >\, A \,\big] \;=\; 0 
\end{equation*}
for some $r'<r$ and that for every $\epsilon>0$,
\begin{equation*}
\lim_{\delta\to 0} \, \limsup_{n\to \infty}\,
\bb P^n_{\mu_n} \big[\, \sup_{s,t}
\Vert\, M^n_{t} \,-\, M^n_s\,\Vert_{-r}
\,>\,\epsilon \, \big] \;=\; 0 
\end{equation*}
where the supremum is carried over all $0\le s, t \le R$ such that
$|t-s|\le \delta$.

The first condition follows from the first assertion of Lemma
\ref{l22}. By the second assertion of this lemma, to prove the second
condition, it is enough to show that for all $m\in \bb Z^d$,
$\epsilon>0$, 
\begin{equation*}
\lim_{\delta\to 0} \, \limsup_{n\to \infty}\,
\bb P^n_{\mu_n} \big[\, \sup_{s,t}
\Vert\, M^n_{t}(\phi_m) \,-\, M^n_s(\phi_m)\,\Vert
\,>\,\epsilon \, \big] \;=\; 0 \;.
\end{equation*}
The proof of this statement is similar to the one of \cite[Lemma
11.3.7]{kl} and left to the reader.
\end{proof}

\subsection{Tightness of the process $\mb I^n_t$ in dimension $1$ and $2$}

Denote by $\mb I^n_t$ the process defined by
\begin{equation*}
\mb I^n_t (F) \;=\; \int_0^t X^n_s(\mc A F)\; ds\;, \quad F\,\in\,
C^\infty(\bb T^d)\;.
\end{equation*}
The main result of this section asserts that the process $\mb I^n_t$
is tight in $D([0,T], \mc H_{-r})$ for $r> (3d+5)/2$. The proof is
based on a representation of $\mb I^n_t$.

Recall that we denote by $(P_t:t\ge 0)$ the semigroup associated to
the operator $\mc A$. Fix $0<t\le T$. For $F$ in $C^\infty(\bb T^d)$,
let $\color{bblue} F_{s,t} = P_{t-s} F$, $0\le s\le t$. Denote by
$M^{n,t}_s (F)$, $0\le s\le t$, the martingale defined by
\begin{equation}
\label{51}
M^{n,t}_s(F) \;:=\; X^n_s(F_{s,t}) \;-\; X^n_0(F_{0,t}) \;-\; \int_0^s
(\partial_u + L_n) X^n_u(F_{u,t})\, du\;, 
\end{equation}
for $F\,\in\, C^\infty(\bb T^d)$.  Taking $s=t$ yields a
representation for $X^n_t(F)$:
\begin{equation}
\label{52}
X^n_t(F) \;=\; X^n_t(F_{t,t}) \;=\; X^n_0(F_{0,t})
\;+\; M^{n,t}_t(F) \;+\; \int_0^t
(\partial_u + L_n) X^n_u(F_{u,t})\, du\;.
\end{equation}

Since $\partial_u X^n_u(F_{u,t}) \,=\, -\, X^n_u(\mc A F_{u,t})$, by
\eqref{42},
\begin{equation}
\label{57}
(\partial_u + L_n) X^n_u(F_{u,t}) \;=\;
R^n_u(F_{u,t}) \;+\; B^n_u(F_{u,t})\;, 
\end{equation}
where $R^n$ and $B^n$ have been introduced just below \eqref{42}. In
particular, the tightness of the process $\mb I^n_t$ follows from the
tightness of the processes $\mb I^{p,n}_t$, $1\le p\le 4$, where
\begin{gather*}
\mb I^{1,n}_t (F) \;:=\; \int_0^t X^n_0(P_s \mc A  F) \; ds\;, \quad
\mb I^{2,n}_t (F) \;:=\;  \int_0^t M^{n,s}_s(\mc A F) \; ds\;, \\
\mb I^{3,n}_t (F)\;:=\; \int_0^t  \int_0^s R^n_u((\mc A F) _{u,s})
 \; du \, ds\;,
\quad
\mb I^{4,n}_t (F) \;:=\;  \int_0^t  \int_0^s
B^n_u((\mc A F) _{u,s})  \; du \, ds\;.
\end{gather*}

Consider the process $\mb I^{1,n}_t$. Recall from \eqref{46} that
$\phi_m$ is an eigenvector of $\mc A$ so that
$P_s \mc A \phi_m = -\, \lambda(m)\, e^{-\lambda(m) s} \, \phi_m$ and 
$X^n_0(P_s \mc A \phi_m) = -\, \lambda(m)\, e^{-\lambda(m) s} \,
X^n_0(\phi_m)$. In particular, by definition, of the
$\mc H_{-r}$-norm, for $0\le s \le t$,
\begin{equation*}
\big\Vert\, \mb I^{1,n}_t \,-\, \mb I^{1,n}_s\, \big\Vert^{2}_{-r}
\;=\; \sum_{m\in\bb Z^d} \gamma^{-r}_m \, 
\Big| \, \lambda(m)\, \int_s^t e^{-\lambda(m) u} \, du \,\Big|^2\,
\big\Vert\, X^n_0(\phi_m) \, \big\Vert^{2} \;,
\end{equation*}
so that
$\Vert\, \mb I^{1,n}_t \,-\, \mb I^{1,n}_s\, \Vert^{2}_{-r} \,\le\,
C_0 (t-s)^2 \, \Vert\, X^n_0 \,\Vert^{2}_{-r+1}$ for some finite
constant $C_0$. Therefore, by \eqref{68}, for $r> (d/2)+1$, 
\begin{equation*}
\bb E_{\mu_n} \big[\, \sup_{|t-s|\le \delta}
\Vert\, \mb I^{1,n}_t \,-\, \mb I^{1,n}_s\, \Vert^{2}_{-r}\,
\big] \;\le\; C_0(\rho) \, \delta^2 \, \frac{1}{a_n}\,
\big( H_n(\mu_n\,|\, \nu^n_\rho) \, +\, C_0\,\big) \;.
\end{equation*}
The tightness of the process $\mb I^{1,n}$ in $D([0,T], \mc H_{-r})$
for $r> 1+(d/2)$ follows from this estimate and the fact that
$\mb I^{1,n}_0=0$.

\smallskip We turn to the process $\mb I^{3,n}_t$. Since $\phi_m$ is
an eigenvector for $\mc A$, a change of the order of summations yields
that
\begin{equation*}
\mb I^{3,n}_t (\phi_m)\;=\; -\, \int_0^t  R^n_u(\phi_m)
\big\{\, 1 \,-\, e^{-\lambda(m)\, [t-u]} \,\big\} \; du \;.
\end{equation*}
Hence, by definition of the norm in $\mc H_{-r}$ and Schwarz
inequality,
\begin{equation*}
\Vert\, \mb I^{3,n}_t \, \Vert^{2}_{-r} \;\le\; t\, 
\int_0^t \Vert\, R^{n}_s \, \Vert^{2}_{-r}\; ds
\end{equation*}
for all $t>0$, $r>0$. In particular, the tightness of the process
$\mb I^{3,n}$ in $D([0,T], \mc H_{-r})$, for $r>3 + (d/2)$, follows
from Lemma \ref{l26}.

\smallskip We turn to the process $\mb I^{4,n}_t$. As for $\mb
I^{3,n}_t$, we have that
\begin{equation*}
\mb I^{4,n}_t (\phi_m)\;=\; -\, \int_0^t  B^n_u(\phi_m)
\big\{\, 1 \,-\, e^{-\lambda(m)\, [t-u]} \,\big\} \; du \;.
\end{equation*}
We may therefore repeat the arguments presented in Subsection
\ref{ss5.2} to prove that, in dimension $1$ and $2$, the process
$\mb I^{4,n}$ is tight in $D([0,T], \mc H_{-r})$ provided
$r> (3d+5)/2$.

\smallskip
Finally, we consider the process $\mb I^{2,n}_t$. A computation,
similar to the one performed at the beginning of Subsection
\ref{ssec5.3}, yields that the predictable quadratic variation of the
martingale $M^{n,t}(F)$, denoted by $\<M^{n,t}(F)\>$, is given by
\begin{equation}
\label{59}
\<M^{n,t}(F)\>_s \;=\; \int_0^s \Gamma^n_u(F_{u,t})\; du\;,
\end{equation}
where $\Gamma^n(F)$ has been introduced in \eqref{47}.

By Schwarz inequality and since $\phi_m$ is an eigenvector of $\mc A$
associated to the eigenvalue $\lambda(m)$, for  $m\in \bb Z^d$,
\begin{equation*}
\Vert\, \mb I^{2,n}_t (\phi_m) \, \Vert^{2} \;\le\;
\lambda(m)^2\, t\, \int_0^t \Vert\,  M^{n,u}_u (\phi_m) \, 
\Vert^{2}\; du\;.
\end{equation*}

It follows from the two previous displayed equations that for every
$m\in\bb Z^d$, $T>0$,
\begin{equation*}
\sup_{\eta\in \Omega_n}\,
\bb E^n_{\eta} \big[\, \sup_{0\le t\le T}
\Vert\,  \mb I^{2,n}_t (\phi_m) \,\Vert^2 \,\big]
\;\le\; C_0\, T \, \lambda(m)^2 \int_0^T
\bb E^n_{\eta} \Big[\, 
\int_0^u \,  \Gamma^n_u ([\phi_m]_{v,u}) \, dv  \,\Big] \; du
\;.
\end{equation*}

This bound is the key estimate in the proofs of Lemmata \ref{l22} and
\ref{l23}. We just obtained an additional factor $\lambda(m)^2$. We may,
therefore, repeat the arguments presented in Subsection \ref{ssec5.3}
to deduce that the process $\mb I^{2,n}_t$ is tight in $D([0,T], \mc
H_{-r})$ for $r> 3 + (d/2)$.

We have proved the following result.  Let $\mu_n$ be a sequence of
probability measures on $\Omega_n$. Denote by $\bb Q^I_n$ the measure
on $D([0,T], \mc H_{-r})$ induced by the process $\mb I^n_t$ and the
measure $\bb P^n_{\mu_n}$.

\begin{lemma}
\label{l25}
Assume that $d=1$ or $2$. Let $\mu_n$ be a sequence of probability
measures on $\Omega_n$ such
$\lim_{n\to \infty} a^{-1}_n H_n(\mu_n \,|\, \nu^n_\rho) =0$.  The
sequence of probability measures $\bb Q^I_n$ is tight in
$D([0,T], \mc H_{-r})$ for $r> (3d+5)/2$.
\end{lemma}

\subsection{Proof of Proposition \ref{l34}}

Fix $r> (3d+5)/2$ and assume that $d=1$ or $2$. The tightness of the
measure $\bb Q_n$ in $D([0,T], \mc H_{-r})$ follows from the
decomposition \eqref{e2}, and Lemma \ref{l26}, Corollary \ref{l32} and
Lemmata \ref{l23}, \ref{l25}.

The fact that any limit point is concentrated on continuous
trajectories follows from Lemma \ref{l23} and the fact that $M^n_t$ is
the unique process which is not continuous. \qed

\begin{remark}
\label{rm6}
The proof does not hold in dimension $3$ only because we are not able
to prove that the $\mc H_{-r}$-valued process $\mb B^n_t$ is tight in
$C([0,T], \mc H_{-r})$.
\end{remark}

\subsection{Tightness in dimension $3$}
\label{sec5b}

In dimension 3, we prove that the time integral of $X^n_t$ is tight
under the measure $\bb P_{\mu_n}$.  We start with a bound in $L^p$ for
$p>1$. Throughout this subsection, $\mu_n$ is a sequence of
probability measures on $\Omega_n$ such
$\lim_{n\to \infty} a^{-1}_n H_n(\mu_n \,|\, \nu^n_\rho) =0$.

\begin{lemma}
\label{l40}
For every $p<4/3$, there exists a finite constant
$C_0=C_0(\rho, A, T, p)$ such that
\begin{equation*}
\bb E_{\mu_n} \big[ \, \big|\,  X^n_t(F) \,\big|^{p}\, \big] \;\le\;
C_0 \, \Vert F \Vert^p_{C^3(\bb T^d)}
\end{equation*}
for all $0\le t\le T$, $F\in C^\infty(\bb T^d)$, $n\ge 1$.
\end{lemma}

\begin{proof}
In view of \eqref{52} and \eqref{57}, we have to estimate four terms.
The first one is easy. By the proof of Lemma \ref{l48}, and since
$\Vert \, P_t F \, \Vert_\infty \le \Vert \, F \, \Vert_\infty$,
\begin{equation}
\label{69}
\bb E_{\mu_n} \big[ \, X^n_0(P_t F)^2\, \big]
\;\le\; \frac{C_0}{a_n} \, 
\big(\, H_n(\mu_n \,|\, \nu^n_\rho)
\,+\, 1\,\big) \, \Vert \, F \, \Vert^2_\infty\;.
\end{equation}
for some finite constant $C_0$, depending only on $\rho$. By
hypothesis, this expression vanishes as $n\to\infty$. Note that this
estimate holds also in dimension $1$ and $2$.

The martingale term is also simple to estimate. By \eqref{59} and
\eqref{47}, 
\begin{equation*}
\bb E_{\mu_n} \big[ \, M^{n,t}_0(F)^2\, \big]
\;\le\; C_0 \,\Big\{\,  \frac{1}{a_n} \,
\sup_{0\le s\le t} \Vert \, \nabla P_s F \, \Vert^2_\infty \;+\; 
\Vert \, F \, \Vert^2_\infty \, \Big\}
\end{equation*}
for some finite constant $C_0$ depending only on the dimension. As
$P_s$ commutes with the gradient, the right-hand side is bounded by
$ C_0 \, \Vert \, F \, \Vert^2_{C^1(\bb T^d)}$.

We turn to the term $R^n$. By Schwarz inequality and \eqref{64},
\begin{equation*}
\bb E_{\mu_n} \Big[ \, \Big( \int_0^t R^{n}_s(P_{t-s} F)\; ds
\, \Big)^2\, \Big]
\;\le\; \frac{C_0\, T}{a_n n^2} \, \Big\{  \int_0^t H_n(f^n_s) \, ds
\;+\; \, t\, \Big\} \, \sup_{0\le s\le t} \Vert P_{s} F \Vert^2_{C^3(\bb T^d)}
\end{equation*}
for a finite constant which depends only on the cylinder functions
$h_{j,k}$. As $P_s$ commutes with the spatial derivatives, $\sup_{0\le
  s\le t} \Vert P_{s} F \Vert^2_{C^3(\bb T^d)} \le \Vert  F
\Vert^2_{C^3(\bb T^d)}$. Hence, by \eqref{01}, for every $\delta>0$, 
\begin{equation*}
\bb E_{\mu_n} \Big[ \, \Big( \int_0^t R^{n}_s(P_{t-s} F)\; ds
\, \Big)^2\, \Big]
\;\le\; \frac{C_0(T)}{n^{4-d-\delta}} \, \Vert F \Vert^2_{C^3(\bb T^d)}
\end{equation*}
for all $n$ large enough.

It remains to consider the term $B^n$. In Corollary \ref{l43}, fix
$\epsilon>0$ small, and let $s=0$, $F(u) = P_{t-u}F$, $\alpha =2$,
$\delta = \epsilon /2$, $\theta = (2/3) + \epsilon$,
$\upsilon = (4/3) - 2 \epsilon$. With this choice,
$\beta = (4/3) - \epsilon > \upsilon$, and there exists a constant
$C_1=C_1(\rho, A, T)$ such that
\begin{equation*}
\bb E_{\mu_n} \Big[ \, \Big | \int_0^t
\, B^{A,n}_s P_{t-s}(F) \; ds \, \Big |^\upsilon \, \Big]
\;\le \;
C_1\, \frac{\beta}{\upsilon} \,
(\Omega^n_{T})^{\upsilon/\beta} \, 
\sup_{0\le u\le T} \Vert P_sF \Vert^{\upsilon}_\infty
\end{equation*}
where
\begin{equation*}
\Omega^n_{T} \;=\; \frac{n^{\delta \gamma} \, R_d(n)^{\gamma}}
{(a_n\, n^d)^{\theta/2}}\, T^{2} \;.
\end{equation*}
As $d=3$ and $\gamma = (\theta/2) + (\beta/2)$, by definition of
$R_d(n)$, the constant on front of $n$ is equal to
$a^{\beta/2}_n \, n^{\delta \gamma} n^{(\beta/2) - \theta}$. Since
$\gamma\le 2$, by definition of $\delta$, $\beta$ and $\theta$, this
quantity is bounded by $a^{2/3}_n n^{-\epsilon/2}$, which vanishes as
$n\to\infty$.

To complete the proof of the lemma, it remains to apply H\"older's
inequality to the first three terms to derive $L^p$ estimates from
$L^2$ ones.
\end{proof}

Recall that $\bb X^n_t$ is the $\mc H_{-r}$-valued process defined by
\begin{equation*}
\bb X^n_t (F) \;=\; \int_0^t X^n_s (F)\; ds \;, \quad
F\,\in\, C^\infty(\bb T^d)\;.
\end{equation*}

\begin{corollary}
\label{l45}
Fix $T>0$, $r>(3d+7)/2$, $0<\vartheta<1/4$.  Let $\mu_n$ be a sequence
of probability measures on $\Omega_n$ such
$\lim_{n\to \infty} a^{-1}_n H_n(\mu_n \,|\, \nu^n_\rho) =0$.  The
sequence of probability measures
$\bb P^n_{\mu_n} \circ (\bb X^n)^{-1}$ on
$C^\vartheta([0,T], \mc H_{-r})$ is tight.
\end{corollary}

\begin{proof}
The proof is based on Theorem \ref{l46}. The first condition of this
result is in force because $\bb X^n_0=0$. The second one follows from
\eqref{67} below. Fix $q$ so that $(q-1)/q =\vartheta$ and note that
$1<q<4/3$. We claim that there exists a finite constant $C_0$,
depending on $\rho$, $T$, $q$ and $r$ such that
\begin{equation}
\label{67}
\bb E_{\mu_n} \Big[ \, \Vert \, \bb X^{n}_t \,-\, \bb X^{n}_s \,
\Vert^q_{-r} \, \Big] \;\le \; C_0 \, |\, t-s\,|^q 
\end{equation}
for all $0\le s<t\le T$.

The proof of this claim is similar to the one of Corollary \ref{l31}.
Fix $0\le s \le t \le T$, $1<q<p<4/3$, and write
\begin{equation*}
\bb E_{\mu_n} \Big[ \, \Vert \, \bb X^{n}_t \,-\, \bb X^{n}_s \,
\Vert^q_{-r} \, \Big] \;=\;
\int_0^\infty \bb P_{\mu_n} \Big[ \, \Vert \, \bb X^n_t \,-\, \bb X^n_s \,
\Vert_{-r} \,>\, \kappa^{1/q} \, \Big] \, d\kappa\;.
\end{equation*}
Fix $\mf a>0$ and estimate the above probability by $1$ on the
interval $[0,\mf a]$. 

We turn to the interval $[\mf a, \infty)$.  Recall the definition of
the sequence $\{\mf w_m : m\in \bb Z^d\}$ introduced in \eqref{66}.
The probability on the right-hand side of the previous displayed
equation is equal to
\begin{align*}
& \bb P_{\mu_n} \Big[ \, \sum_{m\in \bb Z^d} \gamma^{-r}_m\,
\Vert \, \bb X^n_t(\phi_m) \,-\, \bb X^n_s (\phi_m) \,
\Vert^2 \,>\, \kappa^{2/q} \, \Big] \\
& \qquad \le\; \sum_{m\in \bb Z^d} 
\bb P_{\mu_n} \Big[ \,
\Vert \, \bb X^n_t(\phi_m) \,-\, \bb X^n_s (\phi_m) \,
\Vert^p \,>\, \kappa^{p/q} \,
\{\mf w_m \, \gamma^{r}_m\}^{p/2} \; \Big] \;.
\end{align*}
By Chebyshev, followed by H\"older inequality and by Lemma \ref{l40},
as $\Vert \phi_m \Vert_{C^3(\bb T^d)} \le C_0 \, \Vert m\Vert^3$,
this sum is bounded by
\begin{equation*}
C_1(\rho, A, T,p) \, \frac{1}{\kappa^{p/q}}\, 
\, \sum_{m\in \bb Z^d} 
\frac{\Vert m\Vert^{3p}}{(\mf w_m \, \gamma^{r}_m)^{p/2}}
\, |t-s|^p\;.
\end{equation*}
As $r>(3d+7)/2 > (d+7)/2 + (d/p)$, the sum over $m$ is finite.

Up to this point, we proved that the expectation appearing in
\eqref{67} is bounded by $\mf a + C_2(b-1) \mf a^{1-b} |t-s|^p$ for
all $\mf a>0$. Here, $b= p/q>1$ and $C_2$ is a constant which depends
on $\rho$, $T$, $p$, $q$ and $r$.  It remains to optimize over $\mf a$
to derive the bound \eqref{67}.
\end{proof}

\section{Proof of Theorems \ref{t1} and \ref{t1b}}
\label{sec2}

The proof is divided in two parts. In the previous section, we proved
that the sequence $X^n$, $\bb X^n$ are tight in dimension $d\le 2$ and
$d=3$, respectively. Proposition \ref{p03b} below characterizes the
limit points of the sequence $\bb P_{\mu_n} \circ (X^n)^{-1}$ in
dimension $d\le 2$, and Lemma \ref{l47} the ones of the sequence
$\bb P_{\mu_n} \circ (\bb X^n)^{-1}$ in $d=3$. Throughout this section
$T>0$ and $0<\rho<1$ are fixed, and, \emph{unless otherwise stated,
  all results hold for $d\le 3$}.

We start with a simple consequence of the entropy inequality. The
proof of this result is similar to the one of Lemma \ref{l05} and left
to the reader. At the end of the argument, one has to optimize over
the parameter $\gamma$ introduced in the proof of Lemma \ref{l05}.

\begin{lemma}
\label{l35}
Fix a cylinder function $f$. Then, there exists a finite constant $C_0
= C_0(f, \rho)$ such that
\begin{align*}
& \bb E_{\mu_n} \Big[\,
\int_0^t \,\Big|\, \sum_{x\in \bb T^d_n} J_x \,
[\, f(\tau_x \eta^n(s)) - \tilde f(\rho) \,]  \,\Big| \,  \, ds \, \Big] \\
&\qquad \le\;
\int_0^t \Big\{ \frac{1}{\gamma}
\big[\, H_n(f^n_s) \,+\, \log 2 \,\big] \;+\;
C_0 \, \gamma\, n^d\, \Vert  J \Vert^2_\infty\,
e^{C_0 \, \gamma\,  \Vert  J \Vert_\infty} \,\Big\} \; ds
\end{align*}
for every function $J:\bb T^d_n \to \bb R$, $t>0$, $n\ge 1$,
$\gamma>0$.
\end{lemma}

Fix a function $F$ in $C^\infty(\bb T^d)$, and recall the definition
of the process $R^n(F)$, introduced in \eqref{42}. It is expressed as
the sum of two terms.  To estimate the first one, in the previous
lemma, set
$J_x = \{a_n n^d\}^{-1/2} \big\{ \, (\Delta^n_{j,k} F) -
(\partial^2_{x_j,x_k} F) \,\big\}(x/n)$, $f=h_{j,k}$,
$\gamma = \sqrt{R_n(d) a_n} n$. To estimate the second one, let
$J_x = (a_n/n^2)$ $\{a_n n^d\}^{-1/2} (\Delta_n F) (x/n)$,
$f(\eta) = \eta_0$, $\gamma = \sqrt{R_n(d) /a_n} n^2$. Putting
together the two estimates yield the following result.

\begin{corollary}
\label{l36}
Consider a sequence of measures $\mu_n$ on $\Omega_n$ such that
$H_n(\mu_n \,|\, \nu^n_\rho) \le R_d(n)$.  For every function $F$ in
$C^\infty(\bb T^d)$, $t>0$,
\begin{equation*}
\lim_{n\to\infty} \bb E_{\mu_n} \Big[\,
\int_0^t \,\big|\,  R^n_s(F)  \,\big| \,  \, ds \, \Big]
\;=\;0\;.
\end{equation*}
\end{corollary}

Recall the definition of the martingale $M^n_t(F)$,
$F\in C^\infty(\bb T^d)$, introduced in \eqref{30}. Let
$H:\bb R_+ \times \bb T^d \to \bb R$ be a smooth function, and denote
by $H_t: \bb T^d_n \to \bb R$ the function given by $H_t(x) = H(t,x)$,
and by $M^n_t(H_s)$ the value of the martingale $M^n_t(F)$ for
$F=H_s$.

\begin{lemma}
\label{l37}
For every $t>0$,
\begin{equation*}
M^n_t(H_t) \;=\; X^n_t(H_t) \,-\, X^n_0(H_0) \;-\;
\int_0^t \big\{ L_n X^n_s(H_s) \,+\, X^n_s(\partial_s H_s)\,\big\}\; ds
\;+\; \int_0^t M^n_s(\partial_s H_s)\; ds\;.
\end{equation*}
\end{lemma}

\begin{proof}
Fix $t>0$ and let $F = H_t$. By definition of the martingale
$M^n_t(F)$,
\begin{equation*}
M^n_t(H_t) \;=\;  X^n_t(H_t) \,-\, X^n_0(H_t) \;-\;
\int_0^t  L_n X^n_s(H_t) \; ds\;.
\end{equation*}
Writing $H_t$ as $H_u + \int_{[u,t]} \partial_v H_v\, dv$, for $u=0$
and $u=s$, yields that the right-hand side is equal to
\begin{equation*}
X^n_t(H_t) \,-\, X^n_0(H_0) \;-\; \int_0^t  X^n_0(\partial_s H_s) \; ds
\;-\; \int_0^t  L_n X^n_s(H_s) \; ds \;-\; \int_0^t  ds\int_{s}^t du\,
L_n X^n_s(\partial_u H_u) \;.
\end{equation*}
Change the order of the integrals in the last term. The sum of this
integral with the third term is equal to
\begin{equation*}
-\, \int_0^t  \Big\{ X^n_0(\partial_s H_s)  \;+\;
\int_0^s L_n X^n_u(\partial_s H_s) \, du \Big\} \, ds
\;=\; \int_0^t  \Big\{ M^n_s(\partial_s H_s)
\,-\, X^n_s(\partial_s H_s)  \Big\} \, ds \;.
\end{equation*}
This completes the proof of the lemma.
\end{proof}

This lemma provides a representation of the martingale $M^{n,t}$,
introduced in \eqref{51}, in terms of the martingales $M^n$. Fix
$0<t\le T$, and a function $F$ in $C^\infty(\bb T^d)$. Let $H(s,x) =
(P_{t-s} F)(x)$, $0\le s\le t$. By \eqref{51} and Lemma \ref{l37},
since $H_t=F$ and $\partial_s H_s = - \mc A \, P_{t-s}  F = -\mc A
H_s$, 
\begin{equation}
\label{56}
M^{n,t}_t(F) \;=\; M^n_t(F) \;+\;  \int_0^t M^n_s(P_{t-s} \mc A F)\;.
\end{equation}

Denote by $\<\, \cdot \,,\, \cdot \,\>$ the scalar product in
$L^2(\bb T^d)$:
\begin{equation*}
\<\, F \,,\, G \,\> \;=\; \int_{\bb T^d} F(x) \, G(x)\; dx\;.
\end{equation*}

\begin{lemma}
\label{l38}
Fix $r> 1+(d/2)$ and a sequence of probability measures $\mu_n$ on
$\Omega_n$ such that $H_n(\mu_n\,|\, \nu^n_\rho) \le R_n(d)$. Then,
under the sequence of measures
$\bb Q^M_n = \bb P_{\mu_n} \circ (M^n)^{-1}$ on $D([0,T], \mc H_{-r})$
converges to the centered Gaussian random field whose covariances are
given by
\begin{equation*}
\bb Q^M \big[\, M_t(F) \, M_s(G)\,\big]\;=\;
4\,d\, \chi(\rho)\, (s\wedge t)\, \<\, F \,,\, G \,\> \;,
\quad F\;, \; G \;\in\; C^\infty(\bb T^d)\;.
\end{equation*}
\end{lemma}

\begin{proof}
We proved in Subsection \ref{ssec5.3} that the sequence
$\bb Q^{M}_{\mu_n}$ is tight. It remains to check the uniqueness of
limit points. Denote by $\bb Q^M$ one of them and assume, without loss
of generality, that the sequence $\bb Q^M_{\mu_n}$ converges to
$\bb Q^M$.

Fix a function $F$ in $C^\infty(\bb T^d)$. By \eqref{35}, the
martingale $M^n_t(F)$ is uniformly bounded in $L^2(\bb P^n_{\mu_n})$.
Therefore, under the measure $\bb Q^M$, $M_t(F)$ is a martingale.

Recall the definition of $\Gamma^n (F)$, introduced in \eqref{47}. We
have seen in Subsection \ref{ssec5.3} that
\begin{equation*}
M^{(2),n}_t(F) \;:=\; M^n_t(F)^2 \;-\; \int_0^t \Gamma^n_s(F) \; ds
\end{equation*}
is a martingale. On the one hand, by Lemma \ref{l35} [with, for the
main term, $J_x = n^{-d} F(x/n)^2$,
$f(\eta) = \sum_y [\eta_y-\eta_0]^2$, $\gamma = \sqrt{R_d(n) n^d}$]
and Theorem \ref{t2}, under the hypotheses of the lemma,
\begin{equation*}
\lim_{n\to\infty} \bb E_{\mu_n} \Big[\,
\int_0^t \,\Big|\, \Gamma^n_s(F)
\,-\, 4\, d\, \chi(\rho) \, \Vert F\Vert^2  \,\Big| \,  \, ds \, \Big]
\;=\;0\;.
\end{equation*}
On the other hand, by \eqref{36}, the martingale $M^{(2),n}_t(F)$ is
uniformly bounded in $L^2(\bb P^n_{\mu_n})$.  Therefore, under the
measure $\bb Q^M$,
$M_t(F)^2 \,-\, 4\, d\, \chi(\rho) \, \Vert F\Vert^2\, t$ is a
martingale. In particular, for each $F\in C^\infty(\bb T^d)$, $M_t(F)$
is a time-change of Brownian motion, and $M_t(F)$ a Gaussian random
variable.

To complete the proof of the lemma, it remains to compute the
covariance of $M_t(F)$ and $M_s(G)$ through polarization.
\end{proof}

Note that the process $M_t$ can be represented in terms of a
space-time white noise, denoted by $\{\xi(t,x) : t\in \bb R \,,\,
x\in \bb T^d\}$, as
\begin{equation}
\label{54}
M_t(F) \;=\; \sqrt{4\, d\, \chi(\rho)}\, \int_0^t ds \int_{\bb T^d}
dx\, F(x) \, \xi(s,x)\;.
\end{equation}

\begin{corollary}
\label{l39}
Fix $0< t\le T$ and a real-valued function $F$ in
$C^{0,\infty}([0,t]\times \bb T^d)$. Under the hypotheses of the
lemma, the sequence of continuous processes
\begin{equation*}
\bb M^n_s \;=\; \int_0^s M^n_u(F_u)\, du\;, \quad 0\le s\le t\;,
\end{equation*}
converges in $D([0,t], \bb R)$ to the centered Gaussian process
$\int_0^s M_u(F_u)\, du$.
\end{corollary}

\begin{proof}
This result follows from the lemma and the continuity of the function
$\Psi : D([0,t], \mc H_{-r}) \to D([0,t], \bb R)$ given by
$\Psi (M)_s = \int_{[0,s]} M_u(F_u)\, du$.
\end{proof}

Fix $0<t\le T$ and recall from \eqref{51} the definition of the
$\mc H_{-r}$-valued process $M^{n,t}_s$, $0\le s\le t$. By \eqref{56},
$M^{n,t}_t$ can be expressed in terms of the processes $M^n_t$ and
$\bb M^n_t$, where the last one has been introduced in Corollary
\ref{l39}.  The continuity argument used in the proof of this
corollary yields the following result.

Fix $p\ge 1$, $0\le t_1<t_2 < \cdots <t_p\le T$ and functions
$F_j\in C^\infty(\bb T^d)$, $1\le j\le p$.  The random vector
$(M^{n,t_1}_{t_1}(F_1), \dots, M^{n,t_p}_{t_p}(F_p))$ converges in
distribution to a centered Gaussian random vector, characterized by
its covariances which are given by
\begin{equation}
\label{53}
\text{\rm Cov\,}  \big(\, M^{t_j}_{t_j}(F_j)
\,,\, M^{t_k}_{t_k}(F_k)\,\big)\;=\;
4\,d\, \chi(\rho)\, \, \int_0^{t_j}
\<\, P_{t_j-s}F_j \,,\, P_{t_k-s}F_k \,\> \;ds
\end{equation}
for $t_j<t_k$. A similar result holds for the process $X^n$.

\begin{proposition}
\label{p03b}
Let $\mu_n$ be a sequence of probability measures on $\Omega_n$ such
$\lim_{n\to \infty} a^{-1}_n H_n(\mu_n \,|\, \nu^n_\rho) =0$.  Fix
$p\ge 1$, $0\le t_1<t_2 < \cdots <t_p\le T$ and functions
$F_j\in C^\infty(\bb T^d)$, $1\le j\le p$. Under the measure
$\bb P_{\mu_n}$, the random vector
$(X^{n}_{t_1}(F_1), \dots, X^{n}_{t_p}(F_p))$ converges in
distribution to a centered Gaussian random vector whose covariances
are given by
\begin{equation*}
\text{\rm Cov\,} \big(\, X_{t_1}(F) \,,\, X_{t_2}(G)\,\big)\;=\;
4\,d\, \chi(\rho)\, \, \int_0^{t_1}
\<\,  P_{t_1-s}F \,,\, P_{t_2-s}G \,\> \;ds
\end{equation*}
for every $0\le t_1\le t_2\le T$, $F$, $G\in C^\infty(\bb T^d)$.
\end{proposition}

\begin{proof}
Fix $F\in C^\infty(\bb T^d)$, $0<t\le T$, and recall from \eqref{52}
that
\begin{equation*}
X^n_t(F) \;=\; X^n_0(F_{0,t})
\;+\; M^{n,t}_t(F) \;+\; \int_0^t R^n_s(F_{s,t}) \; ds
\;+\; \int_0^t B^n_s(F_{s,t}) \; ds \;.
\end{equation*}
By Theorem \ref{p01} and Corollary \ref{l36}, as $n\to\infty$, the
last two terms converge to $0$ in $L^1(\bb P_{\mu_n})$. By
\eqref{69}, $X^n_0(F_{0,t})$ converges to $0$ in
$L^2(\bb P_{\mu_n})$. Hence, $X^n_t(F) \,-\, M^{n,t}_t(F)$
converges to $0$ in $L^1(\bb P_{\mu_n})$. It remains to compute
the asymptotic behavior of the finite-dimensional distributions of
$M^{n,t}_t$. This has been done in \eqref{53}, which completes the
proof of the proposition.
\end{proof}

The previous result identifies the limit points of the sequence
$\bb Q_n = \bb P_{\mu_n} \circ (X^n)^{-1}$.  A similar argument,
relying on the computation of the limit of the Fourier transform of
linear combinations of the random variables $\bb X^{n}_{t_j} (F_j)$
[based on the observation that $X^n_t(F) \,-\, M^{n,t}_t(F)$ converges
to $0$ in $L^1(\bb P_{\mu_n})$, that $M^{n,t}_t(F)$ can be
expressed in terms of the martingale $M^n_s$, and that this later
process converges], yields the following result.

\begin{lemma}
\label{l47}
Let $\mu_n$ be a sequence of probability measures on $\Omega_n$ such
that $\lim_{n\to \infty} a^{-1}_n H_n(\mu_n \,|\, \nu^n_\rho) =0$.
Fix $p\ge 1$, $0\le t_1<t_2 < \cdots <t_p\le T$ and functions
$F_j\in C^\infty(\bb T^d)$, $1\le j\le p$. Under the measure
$\bb P_{\mu_n}$, the random vector
$(\bb X^{n}_{t_1} (F_1), \dots, \bb X^{n}_{t_p}(F_p))$ converges in
distribution to a centered Gaussian random vector. The covariances are
the ones obtained by integrating the covariances \eqref{53}.
\end{lemma}

The process $X_t$ can also be represented in terms of the space-time
white noise $\{\xi(t,x) : t\in \bb R \,,\, x\in \bb T^d\}$ introduced
in \eqref{54}:
\begin{equation*}
X_t(F) \;=\; 2\, \sqrt{d\, \chi(\rho)}\, \int_0^t ds \int_{\bb T^d}
dx\, (P_{t-s}F) (x) \, \xi(s,x)\;.
\end{equation*}

The last result of this section states that the process $X_t$ solves
the stochastic differential equation \eqref{58}.  Define the
$\mc H_{-r}$-valued process $\mc M_t$ by
\begin{equation*}
\mc M_t(F) \;:=\; X_t(F) \;-\; \int_0^t X_s( \mc A\, F)\, ds\;,
\end{equation*}
$0\le t\le T$, $F\in C^\infty(\bb T^d)$.

\begin{lemma}
\label{l21}
The processes $\mc M$ is equal to $\sqrt{4\,d\, \chi(\rho)}\, \xi$,
where $\xi$ is a space-time white noise.
\end{lemma}

\begin{proof}
Since $\mc M$ is a centered Gaussian random field, it is enough to
compute its covariances. Fix $0\le s\le t\le T$ and $F$,
$G \in C^\infty(\bb T^d)$. By definition of $\mc M$ and Proposition
\ref{p03b},
$\{4\,d\, \chi(\rho)\}^{-1} \bb Q \big[\, \mc M_{t}(F) \, \mc
M_{s}(G)\,\big]$ is equal to
\begin{align*}
& \int_0^{s} du\, \<\,  P_{t-u} F \,,\, P_{s-u}G \,\> 
\;-\; \int_0^{t} du \int_0^{u\wedge s} dv\, \<\,  P_{u-v} \mc A F \,,\,
P_{s-v}G \,\> \\
&\quad \;-\; \int_0^{s} du \int_0^{u} dv\, \<\,  P_{t-v} F \,,\,
P_{u-v} \mc A G \,\> \\
&\qquad \;+\; \int_0^{t} du_1 \int_0^{s} du_2
\int_0^{u_1\wedge u_2} dv\, \<\,  P_{u_1-v} \mc A F \,,\,
P_{u_2-v} \mc A G \,\>\;.
\end{align*}
Denote by $I_j$, $1\le j\le 4$, the $j$-th term in this expression.
Recall that $\mc A$ and $P_t$ are symmetric operators in
$L^2(\bb T^d)$. Let
$c(r) = \<\, P_{r} F \,,\, G \,\> = \<\, F \,,\, P_{r}G \,\>$,
$r\ge 0$.  As $\mc A P_{t-s} = - (d/ds) P_{t-s}$, a straightforward
computation yields that
\begin{equation*}
I_1 \;=\; \frac{1}{2} \,\int_{t-s}^{t+s} c(r)\, dr\;, \quad
I_2 \;=\; \frac{1}{2} \, \Big\{ \, \int_{0}^{s} c(r)\, dr
\,+\, \int_{0}^{t-s} c(r)\, dr \,-\, \int_{s}^{t+s} c(r)\, dr
\,\Big\}  \;,
\end{equation*}
\begin{equation*}
I_3 \;=\; \frac{1}{2} \, \Big\{ \, \int_{t-s}^{t} c(r)\, dr
\;-\; \int_{t}^{t+s} c(r)\, dr \,\Big\}  \;,
\end{equation*}
\begin{align*}
I_4 \; &=\; s\, \<\, F \,,\, G \,\> \;+\; 
\frac{1}{2} \, \Big\{ \, -\, \int_{0}^{s} c(r)\, dr
\,+\, \int_{s}^{2s} c(r)\, dr \,\Big\} \\
& +\; 
\frac{1}{2} \, \Big\{ \, -\, \int_{0}^{2s} c(r)\, dr
\,-\, \int_{t-s}^{t} c(r)\, dr
\,+\, \int_{t}^{t+s} c(r)\, dr \,\Big\} \;.
\end{align*}
Denote by $I_{j,p}$ the $p$-th term of $I_j$. Observe that
$I_{2,1} + I_{4,2}=0$, $I_{3,1} + I_{4,5}=0$, $I_{3,2} + I_{4,6}=0$
and $[\, I_1 + I_{2,2}\,] \,+\, [ I_{2,3} + I_{4,3} +
I_{4,4}\,]=0$. Hence,
\begin{equation*}
\bb Q \big[\, \mc M_{t}(F) \, \mc M_{s}(G)\,\big] \;=\;
4\,d\, \chi(\rho)\, s \, \<\, F \,,\, G \,\> \;,
\end{equation*}
as claimed.
\end{proof}

\begin{proof}[Proof of Theorems \ref{t1} and \ref{t1b}]
In dimensions $d=1$ and $2$, the convergence of the process $X^n_t$
follows from the tightness proved in Section \ref{sec5} and from the
characterization of the limit points established in Proposition
\ref{p03b}. By Lemma \ref{l21}, $X$ solves the stochastic differential
equation \eqref{58}.

In dimensions $d=3$, the asymptotic behavior of the finite-dimensional
distributions of $X^n$ have been computed in Proposition
\ref{p03b}. The tightness of the process $\bb X^n$ has been
established in Corollary \ref{l45}, and uniqueness of limit points in
Lemma \ref{l47}.
\end{proof}

\section{Entropy estimates}
\label{sec9}

We present in this section some estimates, based on the entropy
inequality, used repeatedly in the article.

\begin{lemma}
\label{l41b}
Fix a cylinder function $h$. There exists a finite constant $C_0$,
depending only on $h$, such that
\begin{align*}
& \bb E_{\mu_n} \Big[ \, \int_0^t\, \Big( 
\frac{1}{\sqrt{n^d}}\sum_{x\in \bb T^d_n} J_x(s) \, \big[\,(\tau_x h)(\eta^n_s)
\,-\, \tilde h(\rho)\,\big] \Big)^2 \; ds \, \Big] \\
& \quad \;\le\; 
C_0\, \Big\{  \int_0^t H_n(f^n_s) \, ds
\;+\; \, t\, \Big\} \, \sup_{0\le s\le t} \Vert J(s) \Vert^2_\infty\;,
\end{align*}
for all $t>0$, smooth function
$J:[0,t] \times \bb T^d_n \to \bb R$, probability measure $\mu_n$ and
$n\ge 1$.
\end{lemma}

\begin{proof}
By the entropy inequality, the expectation appearing in the statement
of the lemma is bounded by
\begin{equation*}
\int_0^t \frac{1}{\gamma_s}  \Big\{\, H_n(f^n_s) 
\;+\; \log  \, \bb E_{\nu^n_\rho} \Big[ \,
\exp \gamma_s \Big\{\frac{1}{\sqrt{n^d}} \,
\sum_{x\in \bb T^d_n} J_x(s) \, \big[\,(\tau_x h)(\eta^n)
\,-\, \tilde h(\rho)\,\big] \,\Big\}^2  \, \Big]\, \Big\}\; ds
\end{equation*}
for every $\gamma_s>0$. Apply Corollary \ref{l27} with $f= h$,
$F_x = J_x (s)$, $a=\gamma_s = c_0/2 \Vert J(s) \Vert^2_\infty$ to
conclude that the previous expression is bounded by
\begin{equation*}
C_0\, \Big\{  \int_0^t H_n(f^n_s) \, ds
\;+\; \, t\, \Big\} \, \sup_{0\le s\le t} \Vert J(s) \Vert^2_\infty\;,
\end{equation*}
for some finite constant $C_0$ which depends only on $h$.  This
complete the proof.
\end{proof}

It follows from this result and Schwarz inequality that there exists a
finite constant $C_0$, depending only on $h$, such that
\begin{equation}
\label{62}
\begin{aligned}
& \bb E_{\mu_n} \Big[ \, \Big(  \int_0^t\, 
\frac{1}{\sqrt{n^d}}\sum_{x\in \bb T^d_n} J_x(s) \, \big[\,(\tau_x h)(\eta^n_s)
\,-\, \tilde h(\rho)\,\big] \; ds \Big)^2\, \Big] \\
& \quad \;\le\; 
C_0\, t \, \Big\{  \int_0^t H_n(f^n_s) \, ds
\;+\; \, t\, \Big\} \, \sup_{0\le s\le t} \Vert J(s) \Vert^2_\infty\;,
\end{aligned}
\end{equation}
for all $t>0$, smooth function
$J:[0,t] \times \bb T^d_n \to \bb R$, probability measure $\mu_n$ and
$n\ge 1$.

\section{Finite state Markov chains}
\label{sec7}

We present in this section some general results on continuous-time
Markov chains, which we could not find in the literature.

Let $E$ be a finite state-space and $X(t)$ an $E$-valued,
continuous-time Markov chain. Denote its generator by $L$:
\begin{equation*}
(L h)(x) \;=\; \sum_{y\in E} r(x,y)\, [\, h(y) \,-\, h(x)\,]\;.
\end{equation*}

We start with a Feynman-Kac formula.  For a probability measure $\mu$
on $E$, let $\Gamma_\mu (h,h)$ be the functional given by
\begin{equation*}
\Gamma_\mu (h,h) \;=\; \frac{1}{2} \sum_{x,y\in E} \mu(x) \, r(x,y)\,
[\, h(y) - h(x)\, ]^2
\end{equation*}
for $h:E \to \bb R$.  This functional is sometimes called the ``carr\'e du
champs''. In the case where the process is reversible with respect to
$\mu$, $\Gamma_\mu$ coincides with the Dirichlet form: $\Gamma_\mu
(h,h) = -\, \int h\, (Lh)\, d\mu$.

Next result is an extension of \cite[Lemma A.7.2]{kl}, as it does not
require the measure $\mu$ to be stationary for the process $X$. This
result appears as Lemma 3.5 in \cite{jm1}. We provide a slightly
different proof, based on the one of \cite[Lemma A.7.2]{kl}.

\begin{lemma}
\label{l10}
For every function $W: \bb R_+ \times E \to \bb R$,
probability measure $\mu$ on $E$ and $t>0$,
\begin{equation*}
\log\, \bb E_{\mu} \Big[ e^{\int_0^t W(s,X(s))\, ds}\, \Big] \;\le\;
\int_0^t\, \sup_{f} \Big\{ \int W(s) \, f\, d\mu \;+\; \frac{1}{2} \, \int L f\, d\mu
\;-\; \Gamma_\mu (\sqrt{f},\sqrt{f})\,\Big\}\, ds \;,
\end{equation*}
where the supremum is carried over all densities $f$ with respect to
$\mu$
\end{lemma}

\begin{proof}
By the proof of \cite[Lemma A.7.2]{kl}, the left-hand side of the
inequality appearing in the statement of the lemma is bounded by
\begin{equation*}
\int_0^t\, \sup_{h} \Big\{ \int W(s) \, h^2\, d\mu
\;+\; \int (L h) \, h\,  d\mu \,\Big\} \, ds \;,
\end{equation*}
where the supremum is carried over all functions $h: E \to \bb R$ such
that $\int h^2 \, d\mu=1$.

A straightforward computation yields that
\begin{equation*}
\frac{1}{2} \int L h^2 \, d\mu \;-\; \int (L h) \, h\,  d\mu \;=\;
\Gamma_\mu (h,h)\;.
\end{equation*}
Therefore, the previous integral is equal to
\begin{equation*}
\int_0^ t\, \sup_{h} \Big\{ \int W(s) \, h^2\, d\mu \;+\;
\frac{1}{2} \, \int L h^2 \,  d\mu
\;-\; \Gamma_\mu (h,h) \,\Big\} \, ds \;,
\end{equation*}
Since $\Gamma_\mu (|h|\,,\, |h|) \le \Gamma_\mu (h,h) $, we may
restrict the supremum to non-negative functions $h$. At this point, to
complete the proof of the lemma it remains to replace $h$ by $\sqrt{f}$.
\end{proof}

In the case of the voter model with stirring, the previous result
provides the following bound.

\begin{corollary}
\label{l14}
For every function $W: \bb R_+\times \Omega_n \to \bb R$, $0<\rho<1$
and $t>0$,
\begin{equation*}
\log\, \bb E_{\nu^n_\rho} \Big[ e^{\int_0^t W(s,\eta^n(s))\, ds}\, \Big] \;\le\;
\int_0^t\, \sup_{f} \Big\{ \int W(s) \, f\, d\nu^n_\rho \;+\; \frac{a_n}{2} \,
\int V\, f\, d\nu^n_\rho \;-\; n^2\, I_n(f) \,\Big\}\, ds\;,
\end{equation*}
where the supremum is carried over all densities $f$ with respect to
$\nu^n_\rho$ and $V$ is the function introduced in \eqref{05}. 
\end{corollary}

\begin{proof}
We have to estimate the right-hand side of the formula appearing in
the statement of Lemma \ref{l10} with $L=L_n$. On the one hand, as the
measure $\nu^n_\rho$ is invariant for the exclusion dynamics,
\begin{equation*}
\int L_n f\, d\nu^n_\rho \;=\; a_n \int L^V_n f\, d\nu^n_\rho
\;=\; a_n \int ( L^{V,*}_n \mb 1) f\, d\nu^n_\rho \;=\; a_n \int  V\,
f\, d\nu^n_\rho\;, 
\end{equation*}
by definition of $V$

On the other hand, since all terms of $\Gamma_{\nu^n_\rho}$ are
non-negative, disregarding the ones associated to the voter dynamics
yields that
\begin{equation*}
\Gamma_{\nu^n_\rho} (\sqrt{f},\sqrt{f}) \;\ge\; n^2 I_n(f)\;.
\end{equation*}
in view of the explicit formulae for $I_n$ and $\Gamma_{\nu^n_\rho}$.
This completes the proof of the corollary.
\end{proof}

\smallskip\noindent{\bf Martingales.} Denote by $\Gamma_k$, $k=2,3,4$,
the operators defined by
\begin{gather*}
\Gamma_2(h) \;=\; Lh^2 \;-\; 2\, h\, Lh \;, \quad
\Gamma_3(h) \;=\; Lh^3 \;-\; 3\, h\, Lh^2 \;+\; 3\, h^2\, Lh \;, \\
\Gamma_4(h) \;=\; Lh^4 \;-\; 4\, h\, Lh^3 \;+\; 6\, h^2\, Lh^2
\;-\; 4\, h^3\, Lh\;.
\end{gather*}
A straightforward computation yields that
\begin{gather*}
\Gamma_k(h) \;=\; \sum_{x,y\in E} r(x,y)\, \big[\, h(y) \,-\,
h(x)\,\big]^k \;. 
\end{gather*}

Fix a function $h: E\to \bb R$.  It is well known that
\begin{equation}
\label{55}
M_t(h) \;:=\; h(X(t)) \,-\, h(X(0)) \;-\; \int_0^t (Lh)(X(s))\; ds\;,
\end{equation}
\begin{gather*}
M^{(2)}_t(h) \;:=\; M_t(h)^2 \;-\; \int_0^t (\Gamma_2h)(X(s))\; ds
\end{gather*}
are martingales. The next lemma provides a formula for the quadratic
variation of $M^{(2)}_t(h) $.

\begin{lemma}
\label{l24}
Fix a function $h:E\to \bb R$. Let
\begin{equation*}
A(s) \;=\; \Gamma_4(h(X(s))) \;+\; 4\, M_s(h) \, \Gamma_3(h(X(s)))
\;+\; 6\, M_s(h)^2 \, \Gamma_2(h(X(s)))\;.
\end{equation*}
Then,
\begin{equation*}
M_t(h)^4 \;-\; \int_0^t A (s) \; ds
\end{equation*}
is a martingale which vanishes at $t=0$. In particular, the
compensator of the martingale $M^{(2)}_t(h)$, denoted by
$\<M^{(2)}(h)\>_t$, is given by
\begin{align*}
\<M^{(2)}(h)\>_t \; &=\; \int_0^t \big\{ \, 
\Gamma_4(h(X(s))) \;+\; 4\, \Gamma_2(h(X(s)))\,\big\}\; ds \\
&+\; 4\, M_t(h)\, \int_0^t \Gamma_3(h(X(s))) \; ds
+\; 4\, M^{(2)}_t(h)\, \int_0^t \Gamma_2(h(X(s))) \; ds \;.
\end{align*}
\end{lemma}

\begin{proof}
The proof of this result relies on a long computation. As in the proof
of \cite[Lemma A.5.1]{kl}, the unique ingredients are integration by
parts and the fact that the integral of a predictable process with
respect to a martingale is a martingale. Since it is an identity,
details are left to the reader.
\end{proof}

We apply Lemma \ref{l24} to the voter model with stirring. Fix a
function $F$ in $C^\infty(\bb T^d)$. Let $h:\Omega_n \to \bb R$ be
given by $h(\eta) = X^n(F)$. Then, there exists a finite constant
$C_0$ such that
$|\, \Gamma_k(h)\, |\le C_0 \{\, \Vert F\Vert^k_\infty + \Vert \nabla
F\Vert^k_\infty \,\}$ for $2\le k\le 4$, $n\ge 1$. In particular, by
Lemma \ref{l24} and by definition of the martingale $M^n_t(F)$,
introduced in \eqref{30},
\begin{equation*}
\bb E^n_\eta\big[\, M^n_t(F)^4 \,\big] \;\le\;
C_0\, \int_0^t \bb E^n_\eta\big[\, c_4 \,+\, c_3\, |\, M^n_t(F)\,|
\,+\, c_2\, M^n_t(F)^2 \,\big]\; ds\;,
\end{equation*}
where $c_k = \Vert F\Vert^k_\infty + \Vert \nabla F\Vert^k_\infty$.
Hence, by \eqref{35} and Young's inequality, there exists a finite
constant $C_0$ such that
\begin{equation}
\label{36}
\bb E^n_\eta\big[\, M^n_t(F)^4 \,\big] \;\le\;
C_0\, T^2 \, \{\, \Vert F\Vert^4_\infty + \Vert \nabla
F\Vert^4_\infty \,\} 
\end{equation}
all $0\le t\le T$, $\eta\in\Omega_n$, $n\ge 1$.

\section{Decomposition of cylinder functions}
\label{sec8}

Throughout this section, $0<\alpha<1$ is fixed.  Consider a cylinder
function $f\colon \{0,1\}^{\bb Z^d} \to \bb R$. Denote by
$A\subset \bb Z^d$ its support: $f(\eta) = f(\eta_z : z\in A)$. In
particular, there exist constants $c_B$, $B\subset A$, such that.
\begin{equation*}
f(\eta) \;=\; \sum_{B \subset A} c_B \, \eta_B\;,
\end{equation*}
where $\eta_\varnothing = 1$, $\eta_B = \prod_{x\in A} \eta_x$ and the
sum is performed over all subsets $B$ of $A$. Note that the constants
$c_B$'s may depend on $\alpha$: $f(\eta) = \eta_0 - \alpha$ is an
admissible cylinder function.  With this notation, for
$0\le \rho \le 1$,
\begin{equation}
\label{61}
\tilde f (\rho) \;=\;  \sum_{B \subset A} c_B \, \rho^{|B|}
\quad\text{and}\quad
\tilde f' (\rho) \;=\; \sum_{B \subset A \,,\, B\not = \varnothing}
c_B \, |B|\, \rho^{|B|-1}\;.
\end{equation}

Let $\color{bblue} \xi^\rho_\varnothing \,=\, 1$,
$\color{bblue} \xi^\rho_D \,=\, \prod_{x\in D} (\eta_x-\rho)$, $D$ a
finite subset of $\bb Z^d$. Since
\begin{equation*}
\eta_B \;=\; \sum_{D\subset B} \rho^{|B|-|D|}\, \xi^\rho_D\;, 
\end{equation*}
we may rewrite $f(\eta)$ as
\begin{equation}
\label{60}
f(\eta) \;=\; \sum_{D \subset A} \xi^\rho_D \,
\sum_{B: D\subset B \subset A} c_B\,  \rho^{|B|-|D|}\;.
\end{equation}
The cylinder function $f$ is said to have {\it degree $n\ge 0$ in
  $L^2(\nu_\alpha)$} if there exists a finite collection of subsets $D$
of $\bb Z^d$ with cardinality $n$ and real numbers $c'_D$ such that
\begin{equation*}
f(\eta) \;=\; \sum_{D } c'_D\, \xi^\alpha_D \;.
\end{equation*}

Denote by $\Pi_\rho$ the operator given by 
\begin{equation}
\label{49}
(\Pi_\rho f)(\eta) \;:=\; f(\eta) \;-\; \tilde f(\rho) \;-\; \tilde
f'(\rho)\, (\eta_0\,-\, \rho)\;.
\end{equation}

\begin{asser}
\label{l42}
Fix a cylinder function $f(\eta) = \sum_{B\subset A} c_B\, \eta_B$
whose support is contained in a finite subset $A$ of $\bb Z^d$. Then,
for every $0<\rho<1$,
\begin{equation*}
(\Pi_\rho f) (\eta) \;=\; (\Pi^1_\rho f) (\eta)
\;+\; (\Pi^{+2}_\rho f) (\eta) \;,
\end{equation*}
where  
\begin{gather*}
(\Pi^1_\rho f) (\eta) \; =\; \sum_{z\in A} c_{\{z\}} \, (\eta_z - \eta_0)\,
\sum_{B\subset A \,,\, B\ni z} c_B\, \rho^{|B|-1} \;, \\
(\Pi^{+2}_\rho f) (\eta) \; =\; \sum_{D\subset A \,,\, |D|\ge 2} \xi^\rho_D  
\sum_{B: D\subset B \subset A} c_B\, \rho^{|B|-|D|}\;.
\end{gather*}
\end{asser}

Note that $\Pi^1_\rho f$ corresponds to the terms of degree $1$ of
$\Pi_\rho f$ and $\Pi^{+2}_\rho f$ to the ones of degree greater
than or equal to $2$ in $L^2(\nu_\rho)$.

\begin{proof}[Proof of Assertion \ref{l42}]
By \eqref{61} and \eqref{60} $f(\eta) \,-\, \tilde f (\rho)$ is equal
to
\begin{equation*}
\sum_{z\in A} c_{\{z\}} \, \xi^\rho_z
\;+\; \sum_{B \subset A \,,\, |B| \ge 2} c_B\, \Big\{ \sum_{z\in B}
\rho^{|B|-1}\, \xi^\rho_z \;+\;
\sum_{D\subset B \,,\, |D|\ge 2} \rho^{|B|-|D|}\, \xi^\rho_D \,\Big\} \;.
\end{equation*}
and
\begin{equation*}
\tilde f' (\rho) \, (\eta_0 -\rho)
\;=\; \tilde f' (\rho) \, \xi^\rho_0
\;=\; \sum_{z \in  A } c_{\{z\}} \, \, \xi^\rho_0
\;+\; \sum_{B \subset A \,,\, |B|\ge 2} c_B \, |B|\, \rho^{|B|-1} \,\xi^\rho_0\;.
\end{equation*}
Thus,
\begin{align*}
(\Pi_\rho f) (\eta) \; &=\; \sum_{z\in A} c_{\{z\}} \, (\eta_z - \eta_0)
\;+\; \sum_{B \subset A \,,\, |B| \ge 2} c_B\, \rho^{|B|-1}\,  \sum_{z\in B}
(\eta_z - \eta_0) \\
\;& +\; \sum_{B \subset A \,,\, |B| \ge 2} c_B\, 
\sum_{D\subset B \,,\, |D|\ge 2} \rho^{|B|-|D|}\, \xi^\rho_D  \;.
\end{align*}
To complete the proof, it remains to change the order of summations.
\end{proof}

\begin{remark}
\label{rm2}
The above computation shows that the term $\tilde f (\rho)$ removes
the contants from $f(\eta)$, while the expression
$\tilde f' (\rho)\, (\eta_0-\rho)$ transforms the terms of degree $1$
of $f(\eta)$ [that is, the expressions $c_{\{z\}} \xi^\rho_z$\,] in
gradients of the form $c_{\{z\}} \, (\eta_z - \eta_0)$. 
\end{remark}

\smallskip\noindent{\bf Acknowledgments.}  M. J. was funded by the ERC
Horizon 2020 grant 715734, the CNPq grant 305075/2017-9 and the FAPERJ
grant E-29/203.012/201.  C. L. has been partially supported by FAPERJ
CNE E-26/201.207/2014, by CNPq Bolsa de Produtividade em Pesquisa PQ
303538/2014-7, by ANR-15-CE40-0020-01 LSD of the French National
Research Agency.

\end{document}